\documentclass[12pt]{iopart}
\expandafter\let\csname equation*\endcsname\relax
\expandafter\let\csname endequation*\endcsname\relax
\usepackage{enumerate}
\usepackage{multirow}
\usepackage{graphicx}
\usepackage[table,xcdraw]{xcolor}
\usepackage{amsmath, amsthm}
\usepackage{lipsum}
\usepackage{amsfonts}
\usepackage{booktabs}
\usepackage{epstopdf}
\usepackage{algorithmic}
\usepackage{bm}
\usepackage{hyperref}
\usepackage{cleveref}
\usepackage[textsize=scriptsize,backgroundcolor=white]{todonotes}
\usepackage[shortlabels]{enumitem}
\usepackage{amsopn}
\usepackage{mathtools}
\usepackage{algorithm}

\newcommand{\bbR}{\mathbb{R}}

\DeclareMathOperator{\diag}{diag}

\DeclareMathOperator{\re}{Re}
\DeclareMathOperator{\im}{Im}
\DeclareMathOperator{\spn}{span}
\DeclareMathOperator{\supp}{supp}
\DeclareMathOperator{\len}{len}
\DeclareMathOperator{\rank}{rank}
\DeclareMathOperator{\divr}{div}

\DeclareMathOperator{\ran}{ran}
\DeclarePairedDelimiter\abs{\lvert}{\rvert}
\DeclarePairedDelimiter\norm{\lVert}{\rVert}
\def\sto{\xrightarrow{\text{s}}} 
\def\nto{\xrightarrow{\text{n}}} 
\def\srto{\xrightarrow{\text{sr}}} 
\def\sdto{\xrightarrow{\text{sd}}} 

\newcounter{assump}

\newtheorem{theorem}{Theorem}[section]
\newtheorem{proposition}[theorem]{Proposition}
\newtheorem{lemma}[theorem]{Lemma}

\theoremstyle{definition}
\newtheorem{definition}[theorem]{Definition}
\newtheorem*{definition*}{Definition}

\newtheorem*{problem*}{Problem}
\newtheorem*{exercise*}{Exercise}

\theoremstyle{remark}
\newtheorem{remark}[theorem]{Remark}
\newtheorem*{remark*}{Remark}

\crefname{prop}{property}{properties}
\crefname{assump}{assumption}{assumptions}

\begin{document}
\title[Consistent spectral approximation of Koopman operators]{Consistent spectral approximation of Koopman operators using resolvent compactification}
\author{Dimitrios Giannakis$^1$ and Claire Valva$^2$}
\address{$^1$ Department of Mathematics, Dartmouth College, Hanover, New Hampshire, USA}
\address{$^2$ Courant Institute of Mathematical Sciences, New York University, New York, New York, USA}
\ead{clairev@nyu.edu}

\begin{abstract}
    Koopman operators and transfer operators represent dynamical systems through their induced linear action on vector spaces of observables, enabling the use of operator-theoretic techniques to analyze nonlinear dynamics in state space. The extraction of approximate Koopman or transfer operator eigenfunctions (and the associated eigenvalues) from an unknown system is nontrivial, particularly if the system has mixed or continuous spectrum. In this paper, we describe a spectrally accurate approach to approximate the Koopman operator on $L^2$ for measure-preserving, continuous-time systems via a ``compactification'' of the resolvent of the generator. This approach employs kernel integral operators to approximate the skew-adjoint Koopman generator by a family of skew-adjoint operators with compact resolvent, whose spectral measures converge in a suitable asymptotic limit, and whose eigenfunctions are approximately periodic. Moreover, we develop a data-driven formulation of our approach, utilizing data sampled on dynamical trajectories and associated dictionaries of kernel eigenfunctions for operator approximation. The data-driven scheme is shown to converge in the limit of large training data under natural assumptions on the dynamical system and observation modality. We explore applications of this technique to dynamical systems on tori with pure point spectra and the Lorenz 63 system as an example with mixing dynamics.

    \bigskip

    \noindent{\it Keywords\/}: Koopman operators, transfer operators, measure-preserving systems, spectral approximation, kernel methods, data-driven techniques
\end{abstract}

\ams{37A10, 37E99, 37G30}

\maketitle

\section{Introduction}
\label{sec:intro}

The operator-theoretic formulation of ergodic theory characterizes dynamical systems through their induced action on linear spaces of observables, realized via composition operators, known as Koopman operators, and their duals, known as transfer or Ruelle--Perron--Frobenius operators \cite{Baladi00,EisnerEtAl15}. Dating back at least to work of Koopman and von Neumann in the 1930s \cite{Koopman31,KoopmanVonNeumann32}, this framework has proven to be a highly powerful analytical tool, allowing to translate problems about nonlinear dynamics, such as occurrence of ergodicity and mixing, to equivalent linear operator problems. Over approximately the past two decades, advances in computational capabilities, numerical methods, and data science have spurred considerable interest in the development of methodologies utilizing these operators for data-driven analysis and forecasting of observables. These methods offer a solid theoretical foundation combined with applicability under real-world data observation modalities, and have been successfully applied in diverse areas such as fluid dynamics, climate dynamics, and molecular dynamics among other fields; see, e.g., the surveys \cite{BruntonEtAl22,KlusEtAl18,Mezic13,OttoRowley21}.

Yet, despite considerable theoretical and methodological progress, fundamental questions remain, particularly with regards to the complex spectral characteristics of evolution operators associated with deterministic dynamics. Indeed, a measure-preserving flow is weak-mixing if and only if the Koopman operator on $L^2$ has a simple eigenvalue at 1, with a constant corresponding eigenfunction, and no other eigenvalues \cite{Walters81}. This means that when faced with systems of high dynamical complexity, numerical methods must be able to consistently approximate the continuous Koopman spectrum, which has been a major challenge. Even if the Koopman operator is diagonalizable, its spectrum is oftentimes not discrete (e.g., it forms a dense subset of the unit circle in the measure-preserving case), which also presents numerical challenges. The spectral properties associated with deterministic dynamics are in stark contrast to the behavior of stochastic systems, where the Kolmogorov and Fokker-Planck operators (the stochastic analogs of the Koopman and transfer operators, respectively) have typically discrete spectra due to the presence of diffusion.

In this paper, we propose a data-driven methodology for spectrally consistent approximation of Koopman and transfer operators in continuous-time measure-preserving ergodic flows that approximates the generator by a skew-adjoint operator with compact resolvent (and thus discrete spectrum). Our approach is based on a combination of ideas from two recent papers on Koopman operator approximation: The paper \cite{DasEtAl21} introduced an approximation technique that regularizes the generator, $V$, of the unitary Koopman group $U^t = e^{tV}$ by pre- and post-composition with kernel integral operators mapping into reproducing kernel Hilbert spaces (RKHSs). This results in a family of compact, skew-adjoint operators $V_\tau$, $\tau>0$, which are diagonalizable and whose spectral measures converge to the spectral measure of $V$ in the limit of vanishing regularization parameter $\tau$. Meanwhile, in \cite{SusukiEtAl21}, Susuki, Mauroy, and Mezi\'c developed a Koopman spectral analysis technique that focuses on the \emph{resolvent} of the generator, $R_z(V) = (z - A)^{-1}$, where $z \in \mathbb C$ is an element of the resolvent set of $V$. A key element of their approach is the use of the integral representation of the resolvent,
\begin{equation}
    \label{eq:resolventintegral}
    R_z(V) = \int_0^\infty e^{-zt} U^t \, dt, \quad \re z > 0,
\end{equation}
which provides a means for computing $R_z(V)$ via quadrature involving the unitary Koopman operators $U^t$ without requiring access to the unbounded generator $V$.

Here, we use the compactification approach of \cite{DasEtAl21} and the integral representation of the resolvent employed in \cite{SusukiEtAl21} to build a compact operator $R_{z,\tau}$ which is the resolvent of a skew-adjoint operator. Using this construction, we obtain a family of skew-adjoint unbounded operators $V_{z,\tau}$ with compact resolvent that converge spectrally to $V$. We view this as a more natural approximation than the compact approximation $V_\tau$ from \cite{DasEtAl21}, which could be said to be over-regularized since $V_\tau$ is bounded whereas $V$ is unbounded. More precisely, it is known that certain one-parameter families of skew-adjoint operators with compact resolvents exhibit holomorphic dependence of the eigenvalues and corresponding eigenspaces on the parameter, which is not true for families of compact operators due to the concentration of the spectrum near 0 \cite[Remark~3.11]{Kato95}. Furthermore, our approximation is structure-preserving in the sense that $V_{z,\tau}$ generates a unitary evolution group analogously to $V$. In this paper, we also develop data-driven approximations of $V_{z,\tau}$ that are skew-adjoint and converge to $V_{z,\tau}$ in suitable asymptotic limits.

\subsection{Review of data-driven methodologies}
\label{sec:literature}

Since the early works of Mezi\'c and Banaszuk \cite{Mezic05,MezicBanaszuk04} and Dellnitz, Froyland, and Junge \cite{DellnitzJunge99,Froyland97} on Koopman and transfer operators, respectively, operator-theoretic techniques have emerged as a highly popular framework for analysis of data generated by dynamical systems, with a correspondingly large literature. In broad terms, one can categorize approximation techniques based the function space in which approximation is performed and the topology used to characterize convergence. For forecasting applications, convergence in strong or weak operator topologies is generally adequate since one is interested in approximating the action of the Koopman or transfer operators on specific forecast observables and/or functionals (e.g., expectation functionals). On the other hand, spectral decomposition techniques generally require stronger approximations to ensure convergence of the spectral measure and/or eliminate spectral pollution. For example, strong resolvent convergence, or generalized strong resolvent convergence, imply convergence of the spectral measure \cite{Bogli17,Chatelin11} (without eliminating the possiblity of spectral pollution; see \cref{rk:spectral_pollution}). In any of these settings, the choice of function space is a key consideration affecting the convergence, robustness, and amount of information that can be extracted from numerical techniques, particularly in spectral analysis applications where the Koopman and transfer operator spectra depend strongly on the space of functions on which these operators are allowed to act. Below, we outline some of the techniques that are more closely related to the methodology described in this paper.

The dynamic mode decomposition (DMD) was originally proposed by Schmid and Sesterhenn as a method for analysis of fluid mechanical snapshot data \cite{SchmidSesterhenn08,Schmid10}, was interpreted by Rowley et al.\ \cite{RowleyEtAl09} as a Koopman operator approximation technique, and is now perhaps the most popular technique for Koopman spectral analysis. It is closely related to linear inverse models \cite{Penland89,TuEtAl14}, whereby the Koopman operator is approximated by a finite-rank operator in a dictionary of linear observables. The raw DMD method provides limited convergence guarantees to the infinite-dimensional Koopman operator, and it has been modified and generalized in many ways since its inception. Among these modifications is the extended DMD (EDMD) technique of Williams et al.~\cite{WilliamsEtAl15}, which replaces the linear DMD dictionary by more general dictionaries of nonlinear observables to derive Galerkin methods for approximation of the Koopman operator on  $L^2$. The use of generalized dictionaries considerably increases the potential approximation power of the method, though the original EDMD formulation offered limited explicit mechanisms to control the asymptotic behavior of the approximation as the dictionary size increases. This becomes an especially pertinent issue in dynamical systems with invariant measures supported on geometrically complex, potentially unknown sets (e.g., nonlinear manifolds or fractal attractors), where it is not obvious how to choose a dictionary that will lead to a well-conditioned EDMD approximation.

Recently, Colbrook and Townsend \cite{ColbrookTownsend24} proposed a DMD-type algorithm called residual DMD (ResDMD) that employs a residual-based approach to eliminate spurious eigenvalues resulting, e.g., from spectral pollution or ill-conditioned EDMD dictionaries. Combined with a framework for approximation of the Koopman resolvent using high-order quadrature, ResDMD is also able to compute smoothed approximations of the spectral measure of the Koopman operator with convergence guarantees. Other variants of DMD such as mpDMD \cite{Colbrook23} and piDMD \cite{BadooEtAl23} preserve unitarity of the Koopman operator on $L^2$ for measure-preserving invertible transformations. Earlier work of Govindarajan et al.\ \cite{GovindarajanEtAl21} developed approximations of the Koopman operator based on periodic approximations of the underlying state space dynamics; such Koopman operator approximations are automatically unitary. In the paper \cite{KordaEtAl20}, Korda, Putinar, and Mezi\'c developed a technique that employs the Christoffel-Darboux kernel in spectral space to consistently approximate the spectral measure of the Koopman operator for measure-preserving systems.

On the transfer operator side, popular approximation techniques are based on the Ulam method \cite{Ulam64}. The Ulam method creates a finite-rank approximation of the transfer operator by partitioning the state space into a finite collection of subsets, and building a Markov matrix from associated transition probabilities under the dynamics. Essentially, the method can be viewed as a Galerkin approximation in subspaces of $L^p$ spanned by piecewise-constant functions. An advantage of the Ulam method is that it is positivity preserving; that is, the finite-rank approximate operator maps positive functions to positive functions which is a key structural property of transfer operators. The Ulam method has been shown to yield spectrally consistent approximations for particular classes of systems such as expanding maps and Anosov diffeomorphisms on compact manifolds \cite{Froyland97}. In some cases, spectral computations from the Ulam method have been shown to recover eigenvalues of transfer operators on anisotropic Banach spaces adapted to the expanding/contracting subspaces of such systems \cite{BlankEtAl02}; however, these results depend on carefully chosen state space partitions that may be hard to construct in high dimensions and/or under unknown dynamics. Various modifications of the basic Ulam method have been proposed that are appropriate for high-dimensional applications; e.g., set-oriented methods \cite{DellnitzJunge99} and sparse grid techniques \cite{JungeKoltai09}. Recent techniques for transfer operator approximation have employed smoothing by kernel integral operators \cite{FroylandEtAl21} and regularization of optimal transport plans \cite{KoltaiEtAl21,JungeEtAl22}.

Another family of approaches has emphasized the use of data-driven dictionaries for Koopman/transfer operator approximation. Among these is a technique called diffusion forecasting \cite{BerryEtAl15}, which approximates Koopman and transfer operators (or their Kolmogorov/Fokker Planck analogs in stochastic systems) using a data-driven basis of $L^2$ obtained via the diffusion maps algorithm \cite{CoifmanLafon06}. Leveraging spectral convergence results for kernel integral operators \cite{TrillosEtAl20,VonLuxburgEtAl08}, this approach was shown to produce a well-conditioned consistent approximation as the amount of training data and basis function increases, thus avoiding issues associated with dictionary choice in \emph{ad hoc} applications of EDMD. In \cite{GiannakisEtAl15,Giannakis19}, the diffusion forecasting technique was adopted in a framework that approximates the generator $V$ of measure-preserving ergodic flows on manifolds by an advection--diffusion operator $L_\tau = V - \tau \Delta$, where $\tau$ is a regularization parameter and $\Delta$ is a Laplace-type diffusion operator. Using a normalized diffusion maps basis appropriate for $H^1$ Sobolev regularity, a Galerkin method was developed for the eigenvalue problem of $L_\tau$ which was found to perform well for systems with pure point spectrum such as ergodic rotations on tori. While promising numerical results were also obtained for mixing systems, \cite{GiannakisEtAl15,Giannakis19} did not address the question of consistent approximation of the continuous Koopman spectrum.

Arguably, the most straightforward case for analyzing the spectral properties of diffusion-regularized generators such as  $L_\tau$ occurs when the regularizing operator $\Delta$ commutes with $V$. The papers \cite{DasGiannakis19,Giannakis19} showed that such a commuting operator $\Delta$ can be obtained from the infinite-delay limit of a family of kernel integral operators constructed using delay-coordinate maps \cite{SauerEtAl91,Takens81}. This result indicates that incorporating delays is useful for improving the efficiency of Koopman eigenfunction approximation in data-driven bases, though the infinite-delay limit was found to be trivial if $V$ is the generator of a weak-mixing system. Earlier work \cite{BerryEtAl13,BruntonEtAl17,GiannakisMajda12a} had independently explored incorporating delay coordinates (possibly with weights \cite{BerryEtAl13}) in diffusion maps and Koopman operator approximation techniques and found that the approach helps reveal slowly decorrelating observables under complex dynamics. Further asymptotic analysis \cite{Giannakis21a} related such slowly-decorrelating observables to the $\epsilon$-approximate Koopman spectrum.

The methods outlined above perform approximation of Koopman operators on $L^p$ spaces, oftentimes associated with invariant measures. In recent years, a distinct line of work has emerged that focuses on approximation of Koopman operators on RKHSs \cite{AlexanderGiannakis20,DasGiannakis20,DasEtAl21,IkedaEtAl22,Kawahara16,KlusEtAl20b,KosticEtAl22,RosenfeldEtAl19,RosenfeldEtAl22}. Unlike $L^2$ spaces whose elements are equivalence classes of functions, RKHSs are function spaces with continuous pointwise evaluation functionals. This allows leveraging techniques from learning theory \cite{CuckerSmale01} to perform spectral analysis and forecasting of observables. At the same time, a challenge with RKHS techniques is that, aside from special cases such as translations on groups, a general RKHS $\mathcal H$ is not invariant under the action of the Koopman operator, making the choice of reproducing kernel a delicate matter. At a minimum, the kernel should be chosen such that the Koopman operator on $\mathcal H$ is densely defined and closable. Provided that such a kernel can be found (which is, in general, non-trivial \cite{IkedaEtAl22}), the Koopman operator on $\mathcal H$ will typically be unbounded, introducing theoretical and numerical complications.

\subsection{Reproducing kernel Hilbert space compactification}
\label{sec:rkhs-compactification}

As a motivation of our approach, we summarize in more detail the RKHS-based Koopman approximation framework proposed in \cite{DasEtAl21}. Rather than identifying RKHSs tailored to particular dynamical systems, \cite{DasEtAl21} employed a different approach that starts from a one-parameter family of reproducing kernels  $k_\tau$, $\tau>0$, satisfying mild regularity assumptions, and uses the corresponding integral operators to perturb the Koopman generator such that it is a compact operator on the corresponding RKHS. In more detail, for a continuous-time, measure-preserving ergodic flow on a state space $X$ with invariant measure $\mu$ and generator $V$ on $L^2(\mu)$, the compactification approach of \cite{DasEtAl21} builds a one-parameter family of skew-adjoint compact operators $W_\tau : \mathcal H_\tau \to \mathcal H_\tau$, where $W_\tau = P_\tau V P_\tau^*$ and $P_\tau : L^2(\mu) \to \mathcal H_\tau$ is an integral operator mapping into the RKHS $\mathcal H_\tau$,
\begin{equation*}
    P_\tau f = \int_X p_\tau(\cdot,x) f(x)\,d\mu(x).
\end{equation*}
The operators $W_\tau$ turn out to be unitarily equivalent to $V_\tau = G_\tau^{1/2} V G_\tau^{1/2}$, $ G_\tau \equiv P_\tau^* P_\tau$, acting on $L^2(\mu)$. The latter, were shown to be compact, skew-adjoint operators converging in strong resolvent sense (and thus spectrally) to the generator $V$ as $\tau \to 0$.

The compactness of $V_\tau$ and $W_\tau$ makes these operators amenable to finite-rank approximation by projection onto finite-dimensional subspaces of $L^2(\mu)$ and $\mathcal H_\tau$, respectively, spanned by eigenvectors of $G_\tau$. Eigendecomposition of the projected operators then yields approximate Koopman eigenvalues and eigenfunctions which were shown to lie in the $\epsilon$-approximate point spectrum of the Koopman operator for a tolerance $\epsilon$ that depends on an RKHS-induced Dirichlet energy functional. In particular, eigenfunctions of the projected $V_\tau$ or $W_\tau$ with small Dirichlet energy as $\tau \to 0$ were shown to be approximately cyclical, slowly decorrelating observables under potentially mixing dynamics. Moreover, being elements of RKHSs, the eigenfunctions of $W_\tau$ can be used in conjunction with the Nystr\"om method in supervised-learning models for predicting the evolution of observables. Since $V_\tau$ and $W_\tau$ are skew-adjoint operators, the associated evolution operators $e^{t V_\tau}$ and $e^{t W_\tau}$ are automatically unitary on $L^2(\mu)$ and $\mathcal H_\tau$, respectively. Unitarity of $e^{t V_\tau}$ and $e^{tW_\tau}$ is structurally consistent with unitarity of the unperturbed Koopman operator $U^t = e^{tV}$.

Yet, despite these attractive features, the RKHS compactification schemes of \cite{DasEtAl21} suffer from two potential shortcomings. First, the compactness and skew-adjointness of $V_\tau$ and $W_\tau$ means that the spectra of these operators are confined to bounded intervals in the imaginary line containing 0. This leads to a high concentration of eigenvalues and frequent eigenvalue crossings as $\tau$ is varied (see, e.g., \cite[Fig.~3]{DasEtAl21}), contributing to sensitivity of the spectra of $V_\tau$/$W_\tau$ with respect to $\tau$, and hindering the identification of continuous curves of eigenvalues and eigenfunctions. Note that  advection--diffusion operators of the type $L_\tau$ from \cref{sec:literature} do not suffer from this issue since they are unbounded operators whose spectra are not confined to the imaginary line (though note that these methods rely on the support of the invariant measure having sufficient regularity to appropriately define a diffusion operator $\Delta$). Another possible shortcoming of \cite{DasEtAl21} pertains to the way the action of the generator $V$ is approximated in data-driven settings. Their approach utilizes finite-difference schemes applied to time-ordered data, and while the error of these approximations can be controlled in the limit of vanishing sampling interval using RKHS regularity, finite-differencing is generally detrimental to noise robustness. Avoiding a finite-difference approximation would also facilitate the processing of data at variable time intervals.

\subsection{Contributions of this work}

\begin{figure}
    \centering
    \includegraphics[width=0.5\textwidth]{"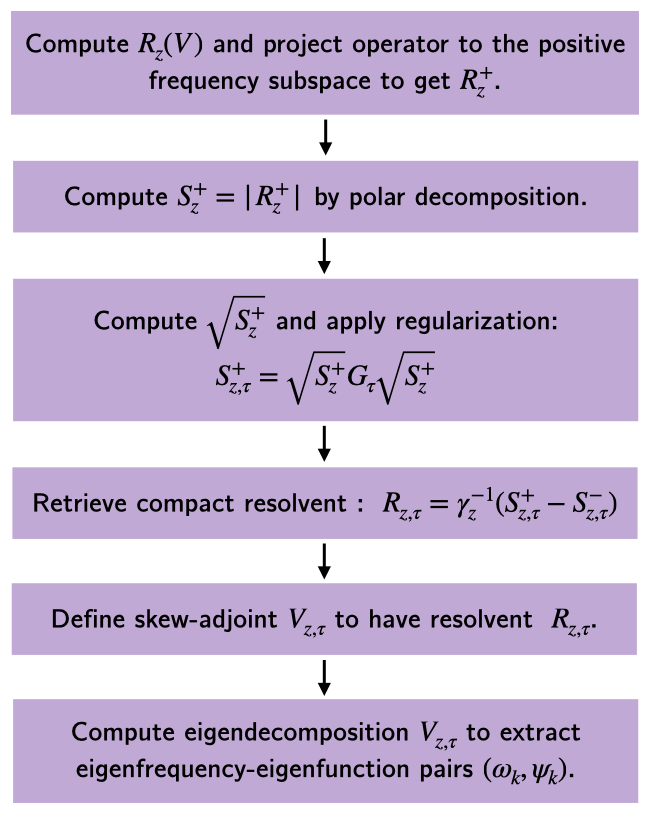"}
    \caption{Overview of the resolvent compactification scheme for spectral approximation of the Koopman generator $V$ described in the paper.}
    \label{fig:schemeoverview}
\end{figure}

Combining ideas from \cite{DasEtAl21} and \cite{SusukiEtAl21}, we develop a data-driven Koopman spectral analysis technique for continuous-time, measure-preserving ergodic flows that is based on a compactification of the resolvent $R_z(V)$ of the generator. The primary distinguishing features of this approach are outlined below. \Cref{fig:schemeoverview} summarizes the main steps of the scheme in a flowchart. 

\begin{enumerate}
    \item \emph{Spectral convergence}. By smoothing $R_z(V)$ with kernel integral operators $G_\tau$, $\tau>0$, we build a family of compact operators $R_{z,\tau} : L^2(\mu) \to L^2(\mu)$ that converge strongly to $R_z(V)$. For any nonzero real number $z$, the $R_{z,\tau}$ are resolvents of skew-adjoint unbounded operators $V_{z,\tau}$ on $L^2(\mu)$ whose spectral measures converge strongly to $V$ (see \cref{thm:main}). By standard properties of the functional calculus, the $V_{z,\tau}$ generate unitary evolution groups $\{U^t_{z,\tau} := e^{t V_{z,\tau}}\}_{t \in \mathbb{R}}$ that converge strongly to the Koopman group, i.e., $\lim_{\tau \to 0^+} U^t_{z,\tau} f = U^t f$ for any observable $f \in L^2(\mu)$ and evolution time $t\in \mathbb R$. These results hold for systems of arbitrary (pure-point, continuous, mixed) spectral characteristics, and do not require differentiable structure for the support of the invariant measure.
    \item \emph{Structure preservation}. The regularized generators $V_{z,\tau}$ preserve two important properties $V$: they are skew-adjoint and unbounded. As just mentioned, skew-adjointness makes these operators amenable to analysis via the spectral theorem and functional calculus for normal operators on Hilbert spaces. Moreover, in our view, unboundedness of $V_{z,\tau}$ represents an improvement over the compactification schemes of \cite{DasEtAl21} as it avoids issues associated with concentration of eigenvalues on bounded intervals containing 0 (see \cref{sec:rkhs-compactification}). Additionally, the unboundedness of $V_{z,\tau}$ can provide a practical benefit in the analysis of Koopman eigenfrequencies obtained from a given dynamical system. As eigenfrequencies do not cluster at zero, it can be easier to identify and interpret the lower frequencies that may be of interest.
    \item \emph{Consistent data-driven approximation.} We approximate $V_{z,\tau}$ by finite-rank, skew-adjoint operators acting on eigenspaces of the smoothing operators $G_\tau$. In data-driven applications, the finite-rank operators are represented in a well-conditioned basis of kernel eigenvectors computed from trajectory data with convergence guarantees in the large-data limit. In particular, the use of kernel eigenvectors avoids issues due to ad hoc choice of dictionaries, and the basis vectors are well-adapted to invariant measures supported on geometrically complex sets (e.g., fractal attractors) embedded in high-dimensional ambient spaces. Other approximation steps include numerical quadrature to approximate the resolvent integral in \eqref{eq:resolventintegral} and low-rank projection to constrain the numerically approximated eigenfunctions of $V_{z,\tau}$ to either positive- or negative-frequency subspaces of $L^2(\mu)$. The scheme converges in the successive limits of increasing number of samples $N$, integral quadrature nodes $Q$, kernel eigenvectors $L$, and rank parameter $M \leq L$, followed by decreasing the regularization parameter $\tau$ to 0.
    \item \emph{No finite differencing.} Through the use of the integral representation \eqref{eq:resolventintegral} for the resolvent, our approach obviates the need to perform finite-difference approximations of the generator (as done, e.g., in \cite{GiannakisEtAl15,Giannakis19,DasEtAl21}). Rather than being directly dependent on the temporal sampling interval of the data, which is oftentimes difficult to control in practical applications, our approximation utilizes numerical quadrature methods for evaluating the integral for $R_z(V)$ in \eqref{eq:resolventintegral}. This provides a flexible framework for controlling approximation accuracy by varying $z$ in relation to the sampling interval and the timespan of the training data. Additionally, the integral representation of the resolvent opens the possibility of using recently developed high-order resolvent-based schemes for approximation of spectral measures \cite{ColbrookEtAl21}.
    \item \emph{Identification of slowly decorrelating observables.} Eigendecomposition of $V_{z,\tau}$ reveals slowly decorrelating observables of the dynamical system and their associated characteristic frequencies. When the generator $V$ possesses non-trivial point spectrum, the spectra of $V_{z,\tau}$ contain sequences of eigenvalues converging to each of the eigenvalues of $V$ (including in situations when the point spectrum of $V$ is not discrete), and the corresponding eigenspaces of $V_{z,\tau}$ converge to Koopman eigenspaces. When the spectrum of $V$ has a non-trivial continuous component (i.e., the span of its eigenvectors is not dense in $L^2(\mu)$) the spectra of $V_{z,\tau}$ contain eigenfunctions that behave as approximate Koopman eigenfunctions. These eigenfunctions reveal slowly decorrelating observables under potentially mixing dynamics. In \cref{sec:examples}, we present numerical experiments illustrating the properties of these eigenfunctions in systems with pure point spectrum (ergodic rotation and skew-rotation on $\mathbb T^2$), and mixing (Lorenz 63 (L63) system).
\end{enumerate}

\subsection{Plan of the paper}
In \cref{sec:problem} we introduce the class of dynamical systems under study and state the problem of spectral approximation of Koopman and transfer operators that will be the focus of the paper. In \cref{sec:theory}, we describe our mathematical framework for Koopman spectral analysis by resolvent compactification, and state and prove our main theoretical results. In \cref{sec:finite-rank}, we construct finite-rank approximations of our compactification scheme. This is followed by a description of its data-driven numerical implementation in \cref{sec:numimplement}. Section~\ref{sec:examples} presents our numerical experiments to the ergodic torus rotation, skew rotation, and L63 system. The main text ends in \cref{sec:conclusions} with a brief conclusory discussion. \ref{app:markov} contains an overview of the kernel construction employed in the operators $G_\tau$. \ref{app:eigf} contains plots of analytically known Koopman eigenfunctions on the torus for comparison with the numerical results in \cref{sec:examples}.

\section{Problem statement}
\label{sec:problem}

In this section, we introduce the class of dynamical systems (\cref{sec:dynamics}) and the Koopman spectral analysis problem we wish to study (\cref{sec:spectral-ergodic,sec:spectral-approx}). Aside from a few exceptions, which we will explicitly point out, our setup and notation follow closely \cite{DasEtAl21}.

\subsection{Dynamical system}
\label{sec:dynamics}

We consider a continuous-time, continuous flow $\Phi^t : \mathcal M \to \mathcal M$, $t \in \mathbb R$, on a metric space $\mathcal M$. The flow is assumed to have an ergodic invariant Borel probability measure $\mu$ with compact support $X\subseteq \mathcal M$. We also assume that there is a compact set $M \subseteq \mathcal M$ which is forward-invariant (i.e., $\Phi^t(M) \subseteq M$ for any $t \geq 0$) and contains $X$. As noted in \cite{DasEtAl21}, many dynamical systems encountered in applications belong in this class, including ergodic flows on compact manifolds with regular invariant measures (in which case $\mathcal{M} = M = X$), dissipative ordinary differential equations on noncompact manifolds (e.g., the L63 system \cite{Lorenz63}, where $\mathcal{M}=\mathbb R^3$, $M$ is an absorbing ball \cite{LawEtAl15}, and $X$ a fractal attractor \cite{Tucker99}), and dissipative partial equations with inertial manifolds \cite{ConstantinEtAl89} (where $\mathcal{M}$ is an infinite-dimensional function space). We should point out that \cite{DasEtAl21} required that the compact set $M$ be a $C^1$ manifold and that the restricted flow $\Phi^t|_M$ is $C^1$. These conditions were introduced in order to ensure the well-definition of certain approximations involving the generator (whose domain includes $C^1$ functions). As mentioned in \cref{sec:intro}, the approximations developed in this work do not require direct access to the generator, and thus we may drop $C^1$ regularity assumptions.

The flow $\Phi^t$ induces a group of unitary Koopman operators on the Hilbert space $H=L^2(\mu)$ that act by composition, $U^t f = f \circ \Phi^t$. By Stone's theorem on strongly continuous 1-parameter unitary groups \cite{Stone32}, the Koopman group $U = \{ U^t \}_{t \in \bbR}$ has a skew-adjoint generator $V: D(V) \to H$ defined on a dense subspace $D(V) \subseteq H$ as
\begin{equation*}
V f = \lim_{t\to 0} \frac{U^t f - f}{t},
\end{equation*}
where the limit is taken in the $L^2(\mu)$ norm. The generator completely characterizes the Koopman group, in the sense that any other strongly continuous one-parameter group on $H$ with generator $V$ is identical to $U$. Moreover, $V$ reconstructs the Koopman operator at any time $t \in \mathbb R$ by exponentiation, $U^t = e^{tV}$, defined via the Borel functional calculus.

We note that the adjoint of $U^t:H \to H$ coincides with the time-$t$ transfer operator $P^t : H \to H$, $ P^t f = f \circ \Phi^{-t}$; that is, $(U^t)^* = U^{-t} = P^t$. Thus, working with the Koopman vs.\ transfer operator on $H$ is merely a matter of convention, though it should be kept in mind that in other function spaces and/or under non-unitary dynamics the Koopman and transfer operator families are generally different. For detailed expositions of operator-theoretic aspects of ergodic theory we refer the reader to \cite{Baladi00,EisnerEtAl15}.

In what follows, $\langle f, g\rangle = \int_X f^* g \, d\mu$ and $\norm{f}_H = \langle f, f\rangle$ are the standard inner product and norm of $H$, respectively, where the inner product is taken conjugate-linear in the first argument. Moreover, $\bm 1 : \mathcal M \to \mathbb R$ is the function on $\mathcal M$ equal to 1 everywhere, and $\tilde H = \{ f \in H: \langle \bm 1, f \rangle \equiv \int_{\mathcal M} f \,d\mu = 0\}$ is the closed subspace of $H$ consisting of zero-mean functions with respect to the invariant measure. We equip the space of continuous functions on $M$ with the standard Banach space norm $\lVert f \rVert_{C(M)} = \max_{x\in M}\lvert f(x)\rvert$, and define the closed subspace $\tilde C(M) = \{ f \in C(M): \int_M f \, d\mu = 0 \}$ analogously to $\tilde H \subset H$. We define $\iota : C(M) \to H$ as the bounded linear map that maps continuous functions on $M$ into their corresponding equivalence classes in $H$.  We will denote the space of bounded linear maps on a Banach space $\mathbb E$ as $B(\mathbb E)$, and $\lVert A\rVert$ will be the operator norm of $A \in B(\mathbb E)$. Limits in the norm topology and strong operator topology of $B(\mathbb E)$ will be denoted as $\nto$ and $\sto$, respectively. Given an operator $A: D(A) \to \mathbb E$ defined on a subspace $D(A) \subseteq \mathbb E$, $\rho(A)$ and $\sigma(A)$ will denote the resolvent set and spectrum of $A$, respectively. Given a topological space $\mathbb X$, $\mathcal B(\mathbb X)$ will denote the Borel $\sigma$-algebra generated by the open subsets of $\mathbb X$.

\subsection{Spectral decomposition of measure-preserving ergodic flows}
\label{sec:spectral-ergodic}

By the spectral theorem for skew-adjoint operators, the spectrum of the generator $V$ is a closed subset of the imaginary line, $\sigma(V) \subseteq i \mathbb R$, and there exists a projection-valued measure (PVM) $E: \mathcal B(i \mathbb R) \to B(H)$ decomposing the generator and Koopman operators via the spectral integrals
\begin{equation}
    \label{eq:specUV}
    V = \int_{i\mathbb R} \lambda\, dE(\lambda), \quad U^t = \int_{i \mathbb R} e^{t\lambda}\, dE(\lambda).
\end{equation}
A fundamental result in ergodic theory \cite{Halmos56} states that there is a $U^t$-invariant splitting $H = H_p \oplus H_c$ into orthogonal subspaces $H_p$ and $H_c$ and a decomposition $E = E_p + E_c$ into PVMs $E_p: \mathcal B(i \mathbb R) \to B(H_p)$ and $E_c: \mathcal B(i \mathbb R) \to B(H_c)$, where $E_p$ is discrete (i.e., there is a countable set $\sigma_p(V) \subseteq \sigma(V)$ such that $E_p(i\mathbb R\setminus \sigma_p(V)) = 0$), and $E_c$ is continuous, (i.e., $E_c$ has no atoms).

The set $\sigma_p(V)$ consists of the eigenvalues of $V$ which are also atoms of $E_p$; that is, given $\lambda \in i \mathbb R$ we have $E_p(\{\lambda\}) \neq 0$ iff $\lambda \in \sigma_p(V)$, in which case there is a corresponding eigenvector $\psi \in H$ such that $V\psi= \lambda\psi$. Correspondingly, $H_p$ admits an orthonormal basis $\{\psi_l\}_{l\in \mathbb Z}$ consisting of generator eigenfunctions that satisfy $V\psi_l = i \omega_l \psi_l$ for $i\omega_l \in \sigma_p(V)$. We call the real numbers $\omega_l$ eigenfrequencies. Since $V$ is a real operator, i.e., $(V f)^* = V(f^*)$ for every $f \in D(V)$, it is natural to order the eigenvectors and corresponding eigenfrequencies such that $\psi_{-l} = \psi_l^*$ and $\omega_{-l} = - \omega_l$. Moreover, since $\mu$ is an ergodic invariant measure under $\Phi^t$, every eigenvalue $i\omega_l$ is simple.

From~\eqref{eq:specUV}, we see that the Koopman evolution of an observable $f\in H_p$ is expressible as
\begin{equation}
    \label{eq:quasiperiodic}
    U^t f = \sum_{l \in \mathbb Z} e^{i\omega_l t} c_l \psi_l, \quad c_l = \langle \psi_l, f\rangle,
\end{equation}
where the summands are periodic in time $t$ with period $2\pi/\omega_l$. It should be noted that the discreteness of $E_p$ does not imply that $\sigma_p(V)$ is a discrete set; in fact, so long as $\sigma_p(V)$ contains two rationally independent elements (such as in the case of an ergodic rotation on the 2-torus; see~\eqref{eq:omega-phi_rot}), $\sigma_p(V)$ is a dense subset of the imaginary line. As noted in \cref{sec:literature}, this presents numerical challenges and typically necessitates some form of regularization of $V$ for stable computation of the eigenvalues in $\sigma_p(V)$ and the corresponding eigenfunctions. See \cite[figure~5]{DasEtAl21} for an illustration of poor behavior of naive discretizations of the generator of an ergodic rotation on $\mathbb T^2$ that occur even even when using a well-conditioned Fourier basis.

Unlike the quasiperiodic evolution from~\eqref{eq:quasiperiodic}, observables in the continuous spectrum subspace $H_c$ exhibit a decay of correlations characteristic of weak-mixing dynamics. Specifically, defining the temporal cross-correlation function $C_{fg}: \mathbb R \to \mathbb C$ of two observables $f,g \in H$ as $C_{fg}(t) = \langle f, U^t g\rangle$, it can be shown that if $g$ lies in $H_c$,
\begin{equation}
    \label{eq:weak-mixing}
    \lim_{T\to \infty} \frac{1}{T} \int_0^T \lvert C_{fg}(t)\rvert \, dt = 0, \quad \forall f \in H;
\end{equation}
see, e.g., the proof of the Mixing Theorem in \cite[p.~39]{Halmos56}. If the system is weak-mixing, i.e., $\lim_{T\to\infty}\frac{1}{T} \int_0^T  \lvert \mu(A \cap \Phi^{-t}(B)) - \mu(A)\mu(B)\rvert \, dt = 0$ for any two sets $A, B \in \mathcal B(X)$, then we have $H_c = \tilde H$ and~\eqref{eq:weak-mixing} holds for any mean-zero observable $g \in \tilde H$.

\subsection{Spectral approximation by operators with discrete spectrum}
\label{sec:spectral-approx}

Equation~\eqref{eq:weak-mixing}, in conjunction with the fact that $H_c$ has, by definition, no basis of Koopman eigenfunctions, raises the question of what should be appropriate generalizations of the Koopman eigenfunction basis of $H_p$ that are able to capture the evolution of arbitrary observables in $H$ (including, in particular, observables in the continuous spectrum subspace $H_c$) in a manner that behaves approximately like the decomposition in~\eqref{eq:quasiperiodic}. Following \cite{DasEtAl21}, we will build such bases by eigendecomposition of a family of operators that are skew-adjoint and diagonalizable (and thus have associated orthonormal eigenfunctions), and converge spectrally to the generator $V$ in a limit of vanishing regularization parameter. To that end, we recall the following notions of convergence of skew-adjoint operators \cite{Chatelin11,Oliveira09}.

\begin{definition}[Convergence of skew-adjoint operators] Let $A: D(A) \to \mathbb H$ be a skew-adjoint operator on a Hilbert space $\mathbb H$ and $A_{\tau} : D(A_{\tau}) \to \mathbb H$ a family of skew-adjoint operators indexed by $\tau$, with resolvents $R_z(A) = (zI - A)^{-1}$ and $R_z(A_\tau) = (zI - A_\tau)^{-1}$ respectively.
    \begin{enumerate}[(i)]
        \item The family $A_\tau$ is said to converge in strong resolvent sense to $A$ as $\tau\to 0^+$, denoted $A_\tau \srto A$, if for some (and thus, every) $z \in \mathbb C \setminus i \mathbb R$ the resolvents $R_z(A_\tau)$ converge strongly to $R_z(A)$; that is, $\lim_{\tau\to 0^+}R_z(A_\tau) f = R_z(A) f$ for every $f\in \mathbb H$.
        \item The family $A_\tau$ is said to converge in strong dynamical sense to $A$ as $\tau\to 0^+$, denoted $A_\tau \sdto  A$ if, for every $t \in \mathbb R$, the unitary operators $e^{tA_\tau}$ converge strongly to $e^{tA}$; that is, $\lim_{\tau\to0^+} e^{tA_\tau} f = e^{tA} f$ for every $f\in \mathbb H$.
    \end{enumerate}
    \label{def:src}
\end{definition}

It can be shown, e.g., \cite[Proposition~10.1.8]{Oliveira09}, that strong resolvent convergence and strong dynamical convergence are equivalent notions. Four our purposes, this implies that if a family of skew-adjoint operators converges to the Koopman generator $V$ in strong resolvent sense, the unitary evolution groups generated by these operators consistently approximate the Koopman group generated by $V$.

Strong resolvent convergence and strong dynamical convergence imply the following form of spectral convergence, e.g., \cite[Proposition~13]{DasEtAl21}.

\begin{theorem} With the notation of \cref{def:src}, let $\tilde E: \mathcal B(i \mathbb R) \to B(\mathbb H)$ and $\tilde E_\tau : \mathcal B(i \mathbb R) \to B(\mathbb H)$ be the spectral measures of $A$ and $A_\tau$, respectively, i.e., $A = \int_{i \mathbb R} \lambda\, d\tilde E(\lambda)$ and $A_\tau = \int_{i \mathbb R} \lambda\, d\tilde E_\tau(\lambda)$. Then, the following hold.
    \begin{enumerate}[(i)]
        \item For every element $\lambda \in \sigma(A)$ of the spectrum of $A$, there exists a sequence $\tau_1, \tau_2, \ldots \searrow 0$ and elements $\lambda_n \in \sigma(A_{\tau_n})$ of the spectra of $A_{\tau_n}$  such that $\lim_{n\to \infty} \lambda_n = \lambda$.
        \item For every bounded continuous function $h: i \mathbb R \to \mathbb C$, as $\tau\to 0^+$ the operators $h(A_\tau) = \int_{i \mathbb R} h(\lambda)\,d\tilde E_\tau(\lambda)$ converge strongly to $h(A) = \int_{i \mathbb R} h(\lambda) \,d\tilde E(\lambda)$.
        \item For every bounded Borel-measurable set $\Theta \in \mathcal B(i \mathbb R)$ such that $\tilde E(\partial \Theta) = 0$ (i.e., the boundary of $\Theta$ does not contain eigenvalues of $A_\tau$), as $\tau\to 0^+$ the projections $\tilde E_\tau(\Theta)$ converge strongly to $\tilde E(\Theta)$.
    \end{enumerate}
    \label{thm:spec-conv}
\end{theorem}

In what follows, we will build a family of densely-defined operators $V_{z,\tau}: D(V_{z,\tau}) \to H$, $D(V_{z,\tau}) \subset H$, parameterized by $z,\tau > 0$ with the following properties.
\begin{enumerate}[label=(P\arabic*)]
    \item \label[prop]{prop:p1} $V_{z,\tau}$ is skew-adjoint.
    \item \label[prop]{prop:p2} $V_{z,\tau}$ has compact resolvent.
    \item \label[prop]{prop:p3} $V_{z,\tau}$ is real, $(V_{z,\tau} f)^* = V_{z,\tau}(f^*)$ for all $f\in D(V)$.
    \item \label[prop]{prop:p4} $V_{z,\tau}$ annihilates constant functions, $V_{z,\tau} \bm 1 = 0$.
    \item \label[prop]{prop:p5} For each $z>0$, $V_{z,\tau} \srto V$ as $\tau\to 0^+$.
\end{enumerate}
By \cref{prop:p1,prop:p2}, the spectrum of $V_{z,\tau}$ consists entirely of isolated eigenvalues with finite multiplicities, and there is an orthonormal basis $ \{ \psi_{l,\tau} \}_{l \in \mathbb Z}$ of $H$ consisting of eigenfunctions,
\begin{equation}
    V_{z,\tau} \psi_{l,\tau} = i\omega_{l,\tau} \psi_{l,\tau},
    \label{eq:vzt-eig}
\end{equation}
where $\omega_{l,\tau} \in \mathbb R$ are eigenfrequencies such that $\sigma(V_{z,\tau}) = \{ i \omega_{l,\tau} \}_{l \in \mathbb Z}$. Moreover, by \ref{prop:p3} and \ref{prop:p4}, the eigenfunctions and eigenfrequencies can be chosen such that
\begin{equation}
    \omega_{-l,\tau} = - \omega_{l,\tau}, \quad \psi_{-l,\tau} = \psi^*_{l,\tau}, \quad \omega_{0,\tau} = 0, \quad \psi_{0,\tau} = \bm 1.
    \label{eq:omega-psi}
\end{equation}
Note the similarity with the corresponding properties of Koopman eigenfunctions $\psi_l$ and eigenfrequencies $\omega_l$ from \cref{sec:spectral-ergodic}. Finally, by \cref{prop:p5} and the equivalence of strong resolvent convergence and strong dynamical convergence, for any $f \in H$ we have,
\begin{equation}
    \label{eq:approx-koopman}
    U^t f = \lim_{\tau \to 0^+} \sum_{l \in \mathbb Z} e^{i\omega_{l,\tau} t} c_{l,\tau} \psi_{l,\tau}, \quad c_{l,\tau} = \langle \psi_{l,\tau}, f\rangle.
\end{equation}
We view~\eqref{eq:approx-koopman} as a generalization of the Koopman evolution of observables in the point spectrum subspace $H_p$ in~\eqref{eq:quasiperiodic}.

\begin{remark}[Spectral exactness]
    \label{rk:spectral_pollution}Two important notions in spectral approximation of linear operators are spectral inclusion and spectral exactness. Following \cite{BaileyEtAl93,Bogli17}, we will say that a family of operators $(A_n)_{n \in \mathbb N}$ with $A_n : D(A_n) \to \mathbb E$ on a Banach space $\mathbb E$ approximating an operator $A : D(A) \to \mathbb E$ is spectrally inclusive if for every $\lambda \in \sigma(A)$ there is a sequence of elements $\lambda_n \in \sigma(A_n)$ such that $ \lim_{n \to \infty} \lambda_n = \lambda$. We will say that the approximation $A_n$ has spectral pollution if there exists a complex number $\lambda \not\in \sigma(A)$ (called here a spurious eigenvalue) and an infinite subset $I \subseteq \mathbb N$ such that $\lambda_n \to \lambda$ for $n \in I$. On the basis of \cref{thm:spec-conv}(i), our spectral approximation of the generator $V$ by the family of operators $V_{z,\tau}$ is spectrally inclusive. However, it is not necessarily free of spurious eigenvalues, and this remains a possible avenue for algorithm improvement.
\end{remark}

\section{Resolvent compactification framework}
\label{sec:theory}

In this section, we describe a procedure for building the operator family $V_{z,\tau}$ that satisfies \crefrange{prop:p1}{prop:p4} laid out in \cref{sec:spectral-approx}. Central to our construction will be the use of kernel integral operators to compactify the resolvent of $V$; we introduce these operators and state their properties in \cref{sec:kernel}. Our approach will make use of an orthogonal splitting of $H$ into subspaces associated with negative, zero, and positive frequencies; we describe this splitting in \cref{sec:freq}. In \cref{sec:mainalg}, we describe the resolvent compactification procedure leading to $V_{z,\tau}$, and in \cref{sec:proof} we state and prove our main theoretical result, \cref{thm:main}. In \cref{sec:pseudospec}, we discuss a procedure for identifying slowly decorrelating observables from the eigenfunctions of $V_{z,\tau}$, and establish connections between the corresponding eigenvalues and the approximate point spectrum of the Koopman operator. We note that for the purposes of our compactification scheme, the resolvent parameter $z$ will be required to be strictly positive, rather than an arbitrary complex number in the resolvent set of $V$.

\subsection{Kernel integral operators}
\label{sec:kernel}

We consider a family of kernel functions $p_\tau: \mathcal M \times \mathcal M \to \mathbb R$, $\tau > 0$, with the following properties.
\begin{enumerate}[label=(K\arabic*)]
    \item \label[prop]{prop:K1} $p_\tau$ is measurable, and it is continuous on $M\times M$.
    \item \label[prop]{prop:K2} $p_\tau$ is Markovian with respect to $\mu$; i.e., $p_\tau \geq 0$ and for every $x\in M$ the normalization condition $\int_X p_\tau(x,\cdot)\,d\mu = 1$ holds.
    \item \label[prop]{prop:K3} The integral operators $G_\tau : H \to H$ with $G_\tau f = \int_X p_\tau(\cdot, x)f(x) \, d\mu(x)$ form a strongly continuous semigroup; i.e., $G_\tau G_{\tau'} = G_{\tau+\tau'}$ for every $\tau,\tau' > 0$ and $\lim_{\tau\to 0^+} G_\tau f = f$ for every $f \in H$.
    \item \label[prop]{prop:K4} $G_\tau$ is strictly positive; i.e., $\langle f, G_\tau f \rangle > 0$ for every nonzero $f \in H$.
    \item \label[prop]{prop:K5} $G_\tau$ is ergodic; i.e., $G_\tau f = f$ for all $\tau > 0$ iff $f$ is constant $\mu$-a.e.
\end{enumerate}
Note that every such kernel $p_\tau$ is necessarily symmetric, $p_\tau(x, x') = p_\tau(x', x)$, and the corresponding integral operator $G_\tau$ is self-adjoint and of trace class. Examples of kernels meeting these conditions are heat kernels on manifolds and translation-invariant kernels on compact abelian groups induced by positive functions on the dual group; e.g., \cite{DasGiannakis23}. The paper \cite{DasEtAl21} describes a procedure that derives the family $p_\tau$ from a suitable un-normalized kernel $k: \mathcal M \times \mathcal M \to \mathbb R_+$ (e.g., a Gaussian kernel on $\mathcal M=\mathbb R^d$, $k(x,x') = e^{-\lVert x - x'\rVert^2 / \epsilon^2}$) through a sequence of normalization steps. We summarize this approach in \cref{app:markov} as we will use it in our numerical experiments in \cref{sec:examples}.

\begin{remark}
    Ref.~\cite{DasEtAl21} required, in addition to the properties listed above, that the kernel $p_\tau$ is positive-definite, i.e., for any finite collection of points $x_1, \ldots,x_n \in \mathcal M$ the $n\times n$ kernel matrix $\bm G = [G_{ij}]$ with elements $G_{ij} = [p_\tau(x_i, x_j)]$ is positive. This condition implies that $p_\tau$ has an associated RKHS $\mathcal H_\tau$ which can be used to define compact approximations of the generator (see \cref{sec:rkhs-compactification}). Here, our focus is on approximations of the generator acting on $L^2(\mu)$, which do not require positive definiteness of $p_\tau$. See \cref{sec:conclusions} for a discussion of possible extensions of our scheme to RKHS-based approximations when the kernel $p_\tau$ is positive-definite.
\end{remark}

For our purposes, the importance of Markovianity and ergodicity of $G_\tau$ is that these operators are contractive and share with the Koopman group $U^t$ the subspace of $H$ spanned by $\mu$-a.e.\ constant functions as a one-dimensional common eigenspace. The latter implies, in particular, that the subspace $\tilde H \subset H$ of zero-mean functions is invariant under $G_\tau$, just as it is under $U^t$. The importance of the continuity of $p_\tau$ on $M \times M$ and strict positivity of $G_\tau$ lies in the fact that every eigenfunction of these operators has a continuous representative. Specifically, if $\phi_j \in H$ is an eigenfunction satisfying $G_\tau \phi_j = \lambda_j^\tau \phi_j$, where the eigenvalue $\lambda_j^\tau$ is strictly positive, it follows by continuity of $p_\tau$ and compactness of $X$ that the function $\varphi_j : M \to \mathbb C$ with $\varphi_j(x) = \lambda_j^{-\tau} \int_X p_\tau(x,\cdot) \phi_j \, d\mu$ is a continuous representative of the $L^2(\mu)$ element $\phi_j$. As a result, associated with the semigroup $G_\tau$ is an orthonormal basis $ \{ \phi_0, \phi_1, \ldots \}$ of $H$ with representatives $\varphi_j \in C(M)$ and strictly positive corresponding eigenvalues $\lambda_0^\tau = 1 > \lambda_1^\tau \geq \lambda_2^\tau \geq \cdots \searrow 0^+$ of finite multiplicity. Henceforth, we will choose the basis functions such that every $\phi_j$ is real and $\phi_0 = 1$ $\mu$-a.e.

\subsection{Frequency-sign-definite subspaces}
\label{sec:freq}

We split $H$ into negative-, zero-, and positive-frequency subspaces,
\begin{equation*}
    H = H_- \oplus H_0 \oplus H_+,
\end{equation*}
defined as $H_- = \Pi_- H$, $H_0 = \Pi_0 H$, and $H_+ = \Pi_+ H$, respectively, using the spectral projections $\Pi_- = E(i(-\infty,0))$, $\Pi_0 = E( \{ 0 \}) = \langle \bm 1, \cdot\rangle \bm 1$, and $\Pi_+ = E(i(0,\infty)) $. By ergodicity, $H_0$ is one-dimensional and $\Pi_0 = \langle \bm 1, \cdot\rangle \bm 1 $. Given any bounded, measurable function $f: i \mathbb R \to \mathbb C$, we have that $f(V) \in B(H)$ commutes with any of $\Pi_-$, $\Pi_0$, or $\Pi_+$, so that $f(V) = f _-(V) + f_0(V) + f_+(V)$, where
\begin{equation*}
    f_\bullet(V) = \Pi_\bullet f(V) = f(V) \Pi_\bullet = \Pi_\bullet f(V) \Pi_\bullet
\end{equation*}
and $\Pi_\bullet$ stands for any of $\Pi_-$, $\Pi_0$, or $\Pi_+$. This means, in particular, that for every $z \in \rho(V)$, the resolvent $R_z(V)$ decomposes as
\begin{equation}
    \label{eq:res-decomp}
    R_z(V) = R^-_z(V) + R^0_z(V) + R^+_z(V),
\end{equation}
where $R^\bullet_z(V) = \Pi_\bullet R_z(V)$. In~\eqref{eq:res-decomp}, $R^0_z(V)$ is trivially a rank-1 operator,
\begin{equation}
    R^0_z(V) = \frac{1}{z}\langle \bm 1, \cdot\rangle \bm 1,
    \label{eq:rz0}
\end{equation}
As such, to arrive at a compact approximation of $R_z(V)$ one only needs to compactify $R^+_z(V)$ and $R^-_z(V)$. In fact, by virtue of the following lemma, it suffices to consider only one of these two operators.

\begin{lemma}
    The map $J: H \to H$ that maps $f \in H$ to its complex conjugate, $J f := f^*$, is an antilinear, isometric isomorphism between $H_+$ and $H_-$. In particular, $J$ intertwines the projections $\Pi_+$ and $\Pi_-$, i.e., $ J \circ \Pi_+ = \Pi_- \circ J$.
    \label{lem:conj}
\end{lemma}

\begin{proof}
    See \cref{sec:proof-conj}.
\end{proof}

It follows from \cref{lem:conj} that $ J \circ R_z^+(V) = R_z^-(V) \circ J$ and thus
\begin{equation}
    R_z^-(V) = J \circ R_z^+(V) \circ J.
    \label{eq:r+-}
\end{equation}
We can therefore compactify $R_z^+(V)$ and use~\eqref{eq:r+-} to obtain a compactification of $R_z^-(V)$.

Before closing this subsection, we note that the decomposition~\eqref{eq:res-decomp} can be equivalently expressed as a direct sum of operators,
\begin{displaymath}
    R_z(V) = R^-_z(V)|_{H_-} \oplus R^0_z(V) |_{H_0} \oplus R^+_z(V) |_{H_+},
\end{displaymath}
where $R^\bullet_z(V) |_{H_\bullet}$ are linear maps from $H_\bullet$ into itself given by restriction of $R^\bullet_z(V)$. In what follows, will be making use of either version of the decomposition of $R_z(V)$ (as well as similar decompositions of other operators) as convenient without change of notation.

\subsection{Construction of the compactified resolvent}
\label{sec:mainalg}
By a version of the spectral mapping theorem for the resolvents of generators of strongly continuous evolution semigroups \cite[Theorem~IV.1.13]{EngelNagel00}, we have that for any $z \in \rho(V)$
\begin{equation*}
    \sigma(R_z(V)) \setminus \{ 0 \} = \left\{ \frac{1}{z - \lambda} : \lambda \in \sigma(V) \right\}.
\end{equation*}
This means that for $z >0$ the resolvent spectrum $\sigma(R_z(V))$ is a subset of the circle of radius $1/z$ centered at the point $1/(2z)$ on the positive real line; see \cref{fig:moveeigs}. We shall denote this circle by $C_z$. Denoting the closed semicircle in $C_z$ located in the upper half of the complex plane by $C_z^+ := \{ \lambda \in C_z: \im\lambda \geq 0 \} $, we have, by the spectral mapping theorem, the inclusion
\begin{equation}
    \sigma(R_z^+(V)) \subseteq C_z^+.
    \label{eq:spec-inc}
\end{equation}
Similarly, the projected resolvent $R_z^-(V)$ to the negative-frequency subspace $H_-$ satisfies $\sigma(R_z^-(V)) \subseteq C_z^-$ with $C_z^- := \{ \lambda \in C_z: \im\lambda \leq 0 \}$. For later convenience, given $z>0$ we define $\beta_z: \tilde C_z \to i \mathbb R$ on $\tilde C_z := C_z \setminus \{ 0 \}$ as
\begin{equation}
    \beta_z(\lambda) = \frac{1}{\lambda} + z.
    \label{eq:res-inv}
\end{equation}
The latter, is a left inverse of the resolvent function $R_z : i \mathbb R \to \tilde C_z \subset \mathbb C$ with $R_z(\lambda) = (z-\lambda)^{-1}$, so that $\beta_z(R_z(V)) = V$.

\begin{figure}
  \centering
  \includegraphics[width=0.7\textwidth]{"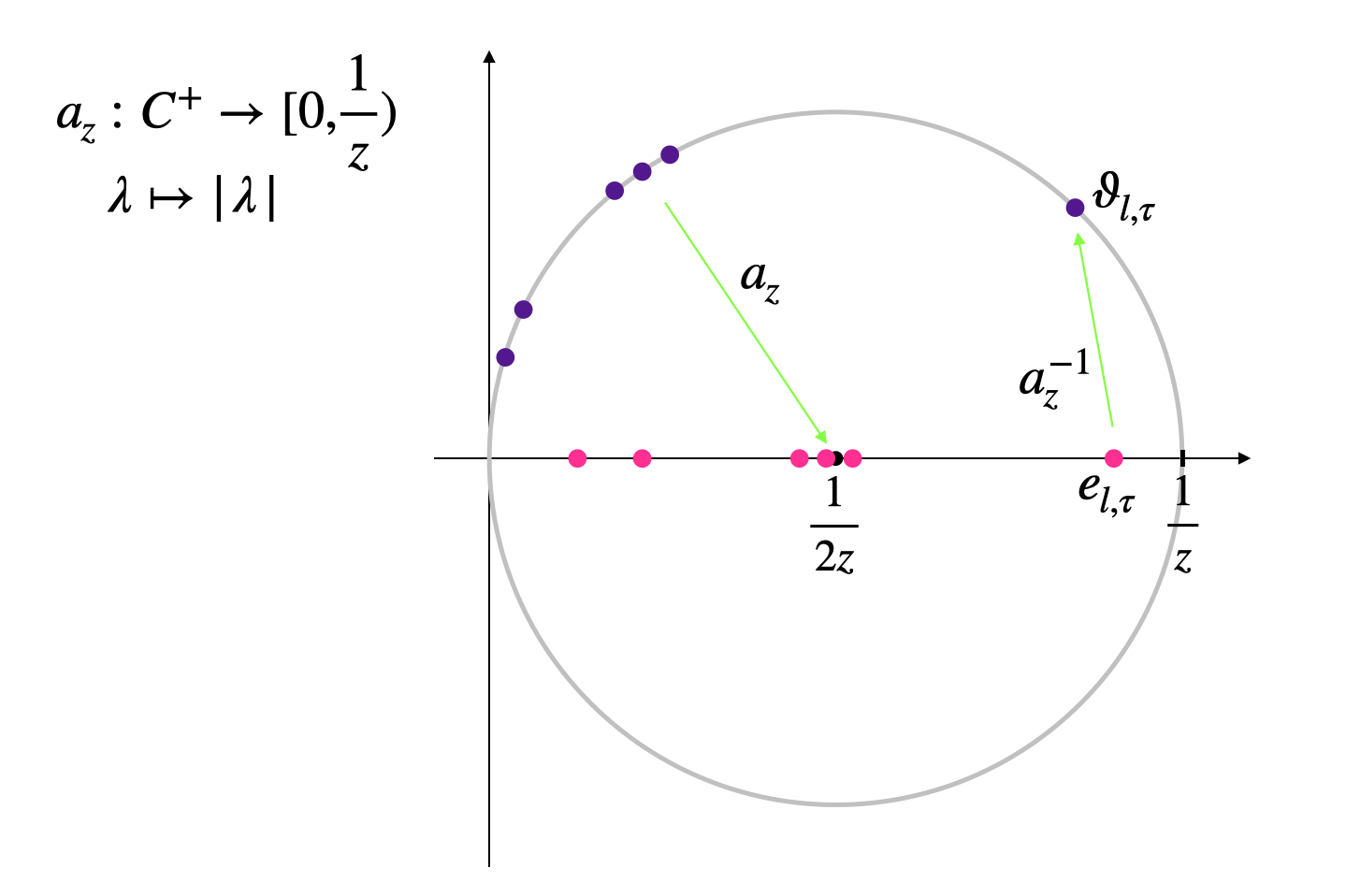"}
  \caption{Schematic of the spectrum of the resolvent $R_z(V)$ for $z>0$ and the bijective correspondence of the spectrum of the projected resolvent $R^+_z(V)$ (which is a subset of the semicircle $C^+_z$ with center $(2z)^{-1}$ and radius $z^{-1}$) and the spectrum of its modulus $S^+_z = \abs{R_z^+(V)}$ (which is a subset of the interval $[0, z^{-1}]$). This correspondence is realized through the function $a_z : C^+_z \to [0, z^{-1}]$ with $a_z(\lambda) = \abs{\lambda}$ (see~\eqref{eq:ginv} for an explicit formula for $a_z^{-1}$). Upon compactification, $S^+_z \mapsto S^+_{z,\tau}$, the operator $S^+_{z,\tau}$ acquires a discrete set of eigenvalues $ e_{l,\tau} \in [0,z^{-1}]$ (magenta dots). These eigenvalues are mapped into $C^+_z$ by an application of $a_z^{-1}$, leading to the eigenvalues $\vartheta_{l,\tau}$ of the compactified resolvent $R_{z,\tau}^+$ (purple dots).}
  \label{fig:moveeigs}
\end{figure}

Next, let $a_z: C_z^+ \to [0, z^{-1}]$ be the modulus function on $C_z^+$ with $a_z(\lambda) = \abs{\lambda}$. By~\eqref{eq:spec-inc}, we can express the modulus of $R_z^+(V)$ as
\begin{equation}
    S^+_z := \abs{R_z^+(V)} = \sqrt{(R_z^+(V))^* R_z^+(V)} = a_z(R_z^+(V)).
    \label{eq:Sz}
\end{equation}
In particular, since $a_z$ is a bijective function, we can recover $R_z^+(V)$ from its modulus i.e.,
\begin{equation*}
    R_z^+(V) = a_z^{-1}(S^+_z),
\end{equation*}
where
\begin{equation}
    a_z^{-1}(x) = x e^{i(\frac{\pi}{2} + \sin^{-1}(xz))}, \quad x \in [0, z^{-1}].
    \label{eq:ginv}
\end{equation}
Note that we can obtain $S^+_z$ by computing a polar decomposition for $R_z^+(V)$,
\begin{equation*}
    R_z^+(V) = W_z S^+_z,
\end{equation*}
where $W_z\in B(H)$ is a partial isometry.

We are now ready to compactify $R_z^+(V)$ by compactifying its modulus. Specifically, given Markov operators $G_\tau \in B(H)$ satisfying \crefrange{prop:K1}{prop:K5}, we define the trace class, positive operators
\begin{equation}
    S^+_{z,\tau} = \sqrt{S^+_z} G_\tau \sqrt{S^+_z},
    \label{eq:szt}
\end{equation}
and the compact operators
\begin{equation}
    R_{z,\tau}^+ = a_z^{-1}(S^+_{z,\tau}).
    \label{eq:rzt+}
\end{equation}
Note that the well-definition of $R_{z,\tau}^+$ in~\eqref{eq:rzt+} relies on the fact that the spectra of $S^+_{z,\tau}$ are subsets of $[0, z^{-1}]$ (i.e., the domain of definition of $a_z^{-1}$) since $G_\tau$ are contractive operators. In particular, for each $z,\tau>0$,  $S^+_{z,\tau}$ is a positive, trace class operator with spectrum contained in $[0,z^{-1}]$, range contained in $H_+$, and nullspace containing $H_-$ and $H_0$.
We can thus view $S^+_{z,\tau}$ as an operator on $H_+$ that admits an eigendecomposition
\begin{equation*}
    S^+_{z,\tau} \psi_{l,\tau} = e_{l,\tau} \psi_{l,\tau}, \quad l \in \mathbb N,
\end{equation*}
where the eigenvalues $e_{l,\tau}$ lie in $ [0, z^{-1}]$ (with 0 as the only possible accumulation point) and the corresponding eigenfunctions $\psi_{l,\tau}$ form an orthonormal basis of $H_+$. By~\eqref{eq:rzt+}, the $\psi_{l,\tau}$ are also eigenfunctions of $R_{z,\tau}^+$, corresponding to the eigenvalues
\begin{equation*}
    \vartheta_{l,\tau} = a_z^{-1}(e_{l,\tau}) \in C_z^+.
\end{equation*}
The following lemma will ensure later on that the $\vartheta_{l,\tau}$ have well-defined corresponding eigenfrequencies.
\begin{lemma}
    The eigenvalues $\vartheta_{l,\tau}$ of $R_{z,\tau}^+$ are nonzero.
    \label{lem:lambda-pos}
\end{lemma}

\begin{proof}
    Since the resolvent function $R_z : i \mathbb R \to \mathbb C$ maps $[0, \infty)$ to $C_z^+ \setminus \{ 0 \}$, 0 is not an eigenvalue of $R_z(V)$ and thus it is also not an eigenvalue of $S^+_z$. Therefore, for every nonzero $f \in H_+$, $S_z^{1/2} f $ is nonzero and $ S^+_{z,\tau} f $ is nonzero so long as $G_\tau S_z^{1/2} f$ is nonzero. The latter follows from the fact that $G_\tau$ is a strictly positive operator (\cref{prop:K4}), and we conclude that $S^+_{z,\tau} f \neq 0$ and thus that every eigenvalue $e_{l,\tau}$ of $S^+_{z,\tau}$ is strictly positive. Therefore, every eigenvalue $\vartheta_{l,\tau} = a_z^{-1}(e_{l,\tau})$ of $R_{z,\tau}^+$ is nonzero.
\end{proof}

We can build a compact approximation $R_{z,\tau}^-$ of the negative-frequency resolvent $R_z^-(V)$ by defining
\begin{displaymath}
    S_z^- = \lvert R_z^-(V) \rvert
\end{displaymath}
and carrying out analogous steps to~\eqref{eq:Sz}--\eqref{eq:rzt+} used to obtain $R_{z,\tau}^+$, replacing $S_z^+$ by $S_z^-$. Analogously to~\eqref{eq:r+-}, this operator satisfies
\begin{equation}
    R_{z,\tau}^- = J \circ R_{z,\tau}^+ \circ J.
    \label{eq:rzt-}
\end{equation}
Thus, we can conveniently write down an eigendecomposition
\begin{displaymath}
    R_{z,\tau}^- \vartheta_{l,\tau} = \vartheta_{l,\tau} \psi_{l,\tau}, \quad l \in \{ -1, -2, \ldots \},
\end{displaymath}
with eigenvalues $\vartheta_{l,\tau} \in C_z^-$ and corresponding orthonormal eigenfunctions $\psi_{l,\tau} \in H_-$ using the eigendecomposition of $R_{z,\tau}^+$:
\begin{displaymath}
    \vartheta_{l,\tau} = \vartheta_{-l,\tau}^*, \quad \psi_{l,\tau} = \psi_{-l,\tau}^*, \quad l \in \{ -1, -2, \ldots \}.
\end{displaymath}

Using~\eqref{eq:rz0}, \eqref{eq:rzt+}, and~\eqref{eq:rzt-}, we define, for each $z,\tau > 0$, $R_{z,\tau} : H \to H$ as the compact operator given by
\begin{equation}
    R_{z,\tau} = R_{z,\tau}^{-} + R_z^0 + R_{z,\tau}^+.
    \label{eq:rzt}
\end{equation}
This operator will serve as our compact approximation of the resolvent $R_z(V)$. Defining $\psi_{0,\tau} = \bm 1$, $\vartheta_{0,\tau} = z^{-1}$, and using the eigendecompositions of $R_{z,\tau}^+$ and $R_{z,\tau}^-$ just described along with \cref{lem:lambda-pos}, we deduce that $R_{z,\tau}$ is diagonal in the orthonormal eigenbasis $ \{ \psi_{l,\tau} \}_{l\in \mathbb Z}$ of $H$ with corresponding eigenvalues $ \{ \vartheta_{l,\tau} \}_{l\in\mathbb Z}$ lying in $C_z \setminus \{ 0 \}$. Correspondingly,
\begin{equation}
    V_{z,\tau} := \beta_z(R_{z,\tau})
    \label{eq:vzt}
\end{equation}
is an unbounded skew-adjoint operator on $H$ with compact resolvents $R_{\tilde{z}}(V_{z,\tau})$, $\tilde z \in \rho(V_{z,\tau})$. For $\tilde z = z$, the resolvent $R_z(V_{z,\tau})$ is equal to $R_{z,\tau}$. In particular, $V_{z,\tau}$ admits the eigendecomposition in~\eqref{eq:vzt-eig} for the eigenfunctions $\psi_{l,\tau}$ and the eigenfrequencies
\begin{equation}
    \omega_{l,\tau} = \frac{1}{i} \beta_z(\vartheta_{l,\tau}) = \frac{1}{i} \left(\frac{1}{\vartheta_{l,\tau}} + z\right).
\end{equation}

\begin{remark}
    An alternative way of compactifying $S_z^+$ would be to define $\tilde S^+_{z,\tau} = G_\tau^{1/2} S_z G_\tau^{1/2}$. This operator is positive, compact, and its spectrum is contained in $[0,z^{-1}]$ (in fact, $\sigma(\tilde S^+_{z,\tau}) = \sigma(S^+_{z,\tau})$), and thus $\tilde R_{z,\tau}^+ = a_z^{-1}(\tilde S^+_{z,\tau})$ is a well-defined compact approximation of $R_z^+(V)$. However, unless $S_z$ and $G_\tau$ happen to commute, $\tilde S^+_{z,\tau}$ does not map $H_+$ into itself, and is thus it is not suitable for providing an eigenbasis of this subspace.
\end{remark}

\subsection{Spectral convergence}
\label{sec:proof}

Our main spectral approximation result is as follows.

\begin{theorem}
    For each $z>0$, the operators $V_{z,\tau}$ in~\eqref{eq:vzt} satisfy \crefrange{prop:p1}{prop:p5}. As a result, as $\tau\to0^+$, $V_{z,\tau}$, converges spectrally to $V$ in the sense of \cref{thm:spec-conv}.
    \label{thm:main}
\end{theorem}

\begin{proof}
    \Crefrange{prop:p1}{prop:p3} follow by construction of $V_{z,\tau}$. \Cref{prop:p4} follows from the fact that $\psi_{\tau,0} = \bm 1$ and $\omega_{\tau,0} = 0$.

    To establish \cref{prop:p5}, fix a nonzero $f \in H_+$, and note that since $a_z^{-1}$ is a continuous function on the closed interval $[0, z^{-1}]$, for any $\epsilon > 0$, there exists a polynomial $p : [0,z^{-1}] \to \mathbb C$ such that $\lVert a_z^{-1} - p \rVert_\infty < \epsilon / (3 \lVert f \rVert_H)$. We have
    \begin{multline*}
        \lVert (R_z(V) - R_z(V_{z,\tau})) f \rVert_H \\
        \begin{aligned}
            &= \lVert (R_z^+(V) - R_{z,\tau}^+)f \rVert_H \\
            &= \lVert (a_z^{-1}(S^+_z) - a_z^{-1}(S^+_{z,\tau})) f\rVert_H \\
            &= \lVert (a_z^{-1}(S^+_z) - p(S^+_z) + p(S^+_z) - p(S^+_{z,\tau}) + p(S^+_{z,\tau}) - a_z^{-1}(S^+_{z,\tau})) f\rVert_H \\
            & \leq \lVert a_z^{-1}(S^+_z) - p(S^+_z)\rVert \lVert f\rVert_H + \lVert (p(S^+_z) - p(S^+_{z,\tau})) f\rVert_H + \lVert p(S^+_{z,\tau}) - a_z^{-1}(S^+_{z,\tau})\rVert \lVert f\rVert_H \\
            & < \frac{2}{3} \epsilon + \lVert (p(S^+_z) - p(S^+_{z,\tau})) f\rVert_H.
        \end{aligned}
    \end{multline*}
    Since, as $\tau \to 0^+$, $G_\tau$ converges strongly to the identity, we have that $S^+_{z,\tau} \sto S^+_z$ and thus $p(S^+_{z,\tau}) \sto p(S^+_z)$. In particular, there exists $\delta > 0$ such that for all $\tau \in (0, \delta)$, $\lVert  (p(S^+_z) - p(S^+_{z,\tau})) f\rVert_H < \epsilon / 3$, and thus
    \begin{equation*}
        \lVert (R_z(V) - R_z(V_{z,\tau})) f \rVert_H < \epsilon.
    \end{equation*}
    It therefore follows that $R_z(V_{z,\tau}) \sto R_z(V)$ on $H_+$. The case $f \in H_-$ follows from the fact that $J \circ (R_z^-(V) - R_{z,\tau}^-) f = (R_z^+(V) - R_{z,\tau}^+) \circ Jf$. The case $f \in H_0$ is trivial.
\end{proof}

As noted in \cref{sec:spectral-approx}, an immediate corollary of \cref{thm:main} is that the unitary evolution groups $\{U^t_{z,\tau} := e^{t V_{z,\tau}}\}_{t\in \mathbb R}$ generated by $V_{z,\tau}$ converge to the Koopman group $ \{ U^t \}_{t \in \mathbb R}$ in the strong operator topology of $H$. In particular, for every $t\in \mathbb R$ and $f \in H$, we have (cf.\ \eqref{eq:approx-koopman})
\begin{equation}
    U^tf = \lim_{\tau \to 0^+} U^t_\tau f, \quad U^t_\tau f = \sum_{l\in \mathbb Z} e^{i\omega_{l,\tau}t} \langle \psi_{l,\tau}, f\rangle \psi_{l,\tau}.
    \label{eq:approx-koopman2}
\end{equation}
Equation~\eqref{eq:approx-koopman2} indicates that, collectively, the eigenvalues and eigenfunctions of $U^t_{z,\tau}$ well-approximate the action of the Koopman operator $U^t$ on observables in $H$.

\subsection{Slowly decorrelating observables and $\epsilon$-approximate Koopman spectra}
\label{sec:pseudospec}

In a number of applications, e.g. \cite{FroylandEtAl21}, one is interested in identifying observables $f \in \tilde H$ with slow decay of their time-autocorrelation function,
\begin{equation}
    \label{eq:autocorr}
    C_f(t) \equiv C_{ff}(t) = \langle f, U^t f\rangle.
\end{equation}
Note that in~\eqref{eq:autocorr} we have taken $f$ to have zero mean. For observables $f \in H$ with $\bar f \equiv \int_{\mathcal M} f \, d\mu \neq 0$ one would consider the anomaly correlation $C_{f'}$, computed via~\eqref{eq:autocorr} for $f' = f - \bar f$. For a Koopman eigenfunction $\psi \in \tilde H$ with unit norm and corresponding eigenfrequency $\omega \in \mathbb{R}$, the autocorrelation function behaves as a pure complex phase,
\begin{equation}
    C_\psi(t) = e^{i\omega t},
    \label{eq:eig-autocorr}
\end{equation}
whose modulus $\abs{C_\psi(t)} = 1$ does not decay as $\abs{t}$ increases. In systems with mixing dynamics, non-trivial such elements of $\tilde H$ do not exist (see \cref{sec:spectral-ergodic}), so it is natural to seek observables $\psi$ that satisfy~\eqref{eq:eig-autocorr} approximately in some sense.

Due to the strong dynamical convergence in~\eqref{eq:approx-koopman2}, the eigenfunctions of $\psi_{l,\tau}$ of $U^t_{z,\tau}$ are natural candidate observables to have slow decay of their time-autocorrelation modulus $\lvert C_{\psi_{l,\tau}}(t)\rvert$. However, one should keep in mind that, in general, strong convergence of $U^t_{z,\tau}$ to $U^t$ does not imply that any given eigenpair $(\omega_{l,\tau}, \psi_{l,\tau})$ will approximately satisfy~\eqref{eq:eig-autocorr} with high accuracy. Our approach is therefore to identify appropriate eigenpairs from the spectra of $V_{z,\tau}$ \emph{a posteriori}, by means of the following procedure.

Given a time parameter $T_c>0$, and a candidate eigenpair $(\omega, \psi)$ with $\omega \in \mathbb R$ and $\psi$ a unit vector in $H$, we compute the non-negative number
\begin{equation}
    \varepsilon_{T_c}(\omega, \psi) = \max_{t\in[-T_c, T_c]} \re \left(1 - e^{-i\omega t}C_\psi(t)\right).
    \label{eq:eig-epsilon}
\end{equation}

Note that $\varepsilon_{T_c}(\omega, \psi) = 0$ for all $T_c \geq 0$ whenever $\psi$ is a Koopman eigenfunction satisfying~\eqref{eq:eig-autocorr}; otherwise $\varepsilon_{T_c}(\omega, \psi) \in (0,1]$. Moreover, we have $\varepsilon_{T_c}(\omega, \psi) = \varepsilon_{T_c}(-\omega, \psi^*)$. Based on these facts, for a given $T_c$, we can order the eigenpairs $ \{ (\omega_{l,\tau}, \psi_{l,\tau}) \}_{l \in \mathbb Z}$ from the spectrum of $V_{z,\tau}$ as
\begin{equation*}
    (\omega_{0,\tau}, \psi_{0,\tau}),\;(\omega_{1,\tau}, \psi_{1,\tau}),\; (\omega_{-1,\tau}, \psi_{-1,\tau}),\; (\omega_{2,\tau}, \psi_{2,\tau}),\; (\omega_{-2,\tau}, \psi_{-2,\tau}),\; \ldots,
\end{equation*}
where $0 = \varepsilon_{T_c}(\omega_{0,\tau}, \psi_{0,\tau}) \leq \varepsilon_{T_c}(\omega_{\pm 1,\tau}, \psi_{\pm 1,\tau}) \leq \varepsilon_{T_c}(\omega_{\pm 2,\tau}, \psi_{\pm 2,\tau}) \leq \cdots \leq 1$. In fact, ordering the eigenpairs in terms of $\varepsilon_{T_c}(\omega_{l,\tau}, \psi_{l,\tau})$ is closely related to identifying elements in the $\epsilon$-approximate spectrum of the Koopman operator, as we now describe.

First, we recall (e.g., \cite{ChaitinHarrabi98,Chatelin11}) that the $\epsilon$-approximate spectrum $\sigma_{\text{ap},\epsilon}(A)$ of a closed operator $A: D(A) \to \mathbb H$ on a Hilbert space $\mathbb H$ is the set of complex numbers $\lambda$ for which there exists a nonzero element $f \in \mathbb H$ such that
\begin{equation*}
    \lVert A f - \lambda f\rVert_{\mathbb H} < \epsilon \lVert f\rVert_{\mathbb H};
\end{equation*}
in other words, $\lambda \in \sigma_{\text{ap},\epsilon}(A)$ can be thought as behaving as an element of the point spectrum up to tolerance $\epsilon$. As $\epsilon$ increases, $\sigma_{\text{ap},\epsilon}(A)$ forms an increasing family of open subsets of the complex plane such that $\bigcup_{\epsilon>0}\sigma_{\text{ap},\epsilon}(A) = \mathbb C$. Moreover, if $A$ is normal and bounded then $\bigcap_{\epsilon>0} \sigma_{\text{ap},\epsilon}(A) = \sigma(A)$. In that case, $\sigma_{\text{ap},\epsilon}(A)$ coincides with the $\epsilon$-pseudospectrum of $A$, the latter defined as the set of complex numbers $\lambda$ such that $\lVert (A-\lambda)^{-1}\rVert > 1/\epsilon$, with the convention that $\lVert (A-\lambda)^{-1}\rVert = \infty$ when $\lambda \in \sigma(A)$ \cite{TrefethenEmbree05}.

One readily verifies that the eigenpairs $(\omega_{l,\tau},\psi_{l,\tau})$ from~\eqref{eq:approx-koopman2} satisfy
\begin{equation*}
    \lVert U^t \psi_{l,\tau} - e^{i\omega_{l,\tau}t} \psi_{l,\tau} \rVert_H^2 \leq \varepsilon_{T_c}(\omega_{l,\tau}, \psi_{l,\tau})
\end{equation*}
It therefore follows that $e^{i\omega_{l,\tau}t}$ lies in the closure $\overline{\sigma_{\text{ap},\epsilon}(U^t)}$ of the $\epsilon$-approximate point spectrum of the Koopman operator $U^t$ for tolerance $\epsilon = \sqrt{2 \varepsilon_{T_c}(\omega_{l,\tau}, \psi_{l,\tau})}$ for every time $t \in [-T_c, T_c]$.

Intuitively, we would like that $\epsilon_{T_c}(\psi_{l,\tau})$ is small for large values of $T_c$; that is, $\psi_{l,\tau}$ should behave approximately like a Koopman eigenfunction uniformly over large time intervals. Moreover, since $U^t_\tau$ is a unitary (and hence normal, bounded) operator, the same conclusions hold for the $\epsilon$-pseudospectrum. While criteria for selecting $T_c$ will depend on the application at hand (e.g., $T_c$ could be chosen on the basis of relevant prior knowledge about the system such as decorrelation and/or Lyapunov timescales), any given $T_c$ induces an ordering the eigenfunctions in terms of $\epsilon_{T_c}(\psi_{l,\tau})$ that we find useful in practice.

\subsection{Proof of \cref{lem:conj}}
\label{sec:proof-conj}

We recall Stone's formula for computing the action of the spectral measure $\hat E : \mathcal B(\mathbb R) \to B(\mathbb H)$ of a self-adjoint operator $A : D(A) \to \mathbb H$ on a Hilbert space $\mathbb H$ (e.g., \cite[Chapter~9]{Oliveira09}),
\begin{equation*}
    \frac{1}{2}(\hat E((a,b)) + \hat E([a,b])) f = \lim_{\epsilon\to 0^+} \int_{[a, b]} (R_{\lambda+i\epsilon}(A) - R_{\lambda-i\epsilon}(A)) f\, d\lambda, \quad \forall f \in \mathbb H,
\end{equation*}
where $-\infty < a < b < \infty$ and $A = \int_{\mathbb R} \lambda\, d\tilde E(\lambda)$. Using this result for $\mathbb H = \tilde H$, $A= \tilde V / i$ with $\tilde V = V|_{\tilde H}$ (i.e., $\tilde V$ is the restriction of the generator on the subspace $\tilde H \subset H$ of zero-mean observables), and the spectral measure $\tilde E : \mathcal B(i \mathbb R) \to B(\tilde H)$ of $\tilde V$, we get
\begin{equation*}
    \frac{1}{2}(\hat E((a,b)) + \hat E([a,b])) f = \lim_{\epsilon\to 0^+} \int_{i[a,b]} (R_{\lambda+\epsilon}(\tilde V) - R_{\lambda-\epsilon}(\tilde V)) f\, d\lambda, \quad \forall f \in \tilde H.
\end{equation*}

Next, choosing a sequence $b_n$ of positive numbers increasing to infinity which are not eigenfrequencies of $\tilde V$, and noting that $a=0$ is not an eigenvalue of $\tilde V$, we have
\begin{equation}
    \frac{1}{2}(\tilde E(i(0,b_n)) + \tilde E(i[0,b_n])) = \tilde E(i[a,b_n])
\end{equation}
and thus, for any $f \in \tilde H$,
\begin{equation}
    \begin{aligned}
        \Pi_+ f &= \lim_{n\to\infty}  \tilde E(i[0, b_n]) f = \lim_{n\to\infty}\lim_{\epsilon\to 0^+}\int_{i[0,b_n]} (R_{\lambda+\epsilon}(\tilde V) - R_{\lambda-\epsilon}(\tilde V)) f\, d\lambda,\\
        \Pi_- f &= \lim_{n\to\infty} \tilde E(i[-b_n, 0]) f = \lim_{n\to\infty}\lim_{\epsilon\to 0^+}\int_{i[-b_n, 0]} (R_{\lambda+\epsilon}(\tilde V) - R_{\lambda-\epsilon}(\tilde V)) f\, d\lambda.
    \end{aligned}
    \label{eq:pi-plus-minus}
\end{equation}
Meanwhile, we have
\begin{align*}
    (\tilde E(i[-b_n, 0]) \circ J) f
        &= \tilde E(i[-b_n, 0]) f^* \\
        &= \lim_{\epsilon \to 0^+} \int_{i[-b_n, 0]}\left(\frac{1}{(\lambda-\epsilon) - \tilde V} - \frac{1}{(\lambda+\epsilon) - \tilde V}\right) f^* \, d\lambda \\
        &= \left(\lim_{\epsilon \to 0^+} \int_{i[-b_n, 0]}\left(\frac{1}{(-\lambda-\epsilon) - \tilde V} - \frac{1}{(-\lambda+\epsilon) - \tilde V} \right) f \, d\lambda\right)^* \\
        &= \left(\lim_{\epsilon \to 0^+} \int_{i[0, b_n]}\left(\frac{1}{(\lambda-\epsilon) - \tilde V} - \frac{1}{(\lambda+\epsilon) - \tilde V} \right) f \, d\lambda \right) ^*\\
        &= (J \circ \tilde E(i[0,b_n]))f.
\end{align*}
With the above and~\eqref{eq:pi-plus-minus}, we obtain
\begin{equation*}
    (\Pi_- \circ J) f = \lim_{n\to\infty} (\tilde E(i[-b_n, 0]) \circ J) f =  \lim_{n\to\infty} (J \circ \tilde E(i[0, b_n])) f = (J \circ \Pi_+) f,
\end{equation*}
proving the lemma.

\section{Finite-rank approximation}
\label{sec:finite-rank}

Invariably, usage of the regularized generators $V_{z,\tau}$ from \eqref{eq:vzt} in practical applications requires access to spectrally consistent finite-rank approximations of these operators. In this section, we address the problem of building such finite-rank approximations in a manner that is amenable to data-driven approximation. In \cref{sec:numimplement}, we will take the finite-rank objects built here and translate them into those computable in a data-driven manner. For readers content with the construction in \cref{sec:theory}, one may find the overview in \cref{alg:numerical} sufficient.

\subsection{Spectral approximation of compact operators and practical obstructions}

Our objective is to build a sequence of operators on $H$ that (i) are resolvents of finite-rank, skew-adjoint operators; and (ii) converge strongly to $R_{z,\tau}$ from~\eqref{eq:rzt}. The following result (e.g., \cite[sections~3.6 and~5.1]{Chatelin11}) will provide a working definition of spectrally convergent approximations of compact operators that we will employ in our construction.

\begin{theorem}
    Let $A: \mathbb E \to \mathbb E$ be a compact operator on a Banach space $\mathbb E$ and $A_1, A_2, \ldots$ a sequence of compact operators on $\mathbb E$ that converges compactly to $A$; that is, $A_n \sto A$ and for every bounded sequence $f_1, f_2, \ldots \in \mathbb E$ the sequence $(A - A_n) f_n$ has compact closure. Then, the following hold for every nonzero eigenvalue $\lambda$ of $A$ and every open neighborhood $O \subseteq \mathbb C$ such that $\sigma(A) \cap O = \{ \lambda \}$:
    \begin{enumerate}
        \item There exists $n_* \in \mathbb N$ such that for all $n > n_*$ the set $\sigma(A_n) \cap O$ contains at most $m$ eigenvalues of $A_n$, where $m$ is the multiplicity of $\lambda$. Moreover, the multiplicities of these eigenvalues sum to $m$, and every $\lambda_n \in \sigma(A_n) \cap O$ converges to $\lambda$ as $n\to\infty$.
        \item The spectral projections to the eigenspaces corresponding to  $\sigma(A_n) \cap O$ converge strongly to the spectral projection to the eigenspace of $A$ corresponding to $\lambda$.
    \end{enumerate}
    \label{thm:spec_compact}
\end{theorem}

As motivation for our approach, we begin by noting that a tentative way of building approximations of $R_{z,\tau}$ would be to approximate the compact operators $S^+_{z,\tau}$ and $S^-_{z,\tau}$ separately by projection onto finite-dimensional subspaces $H_{+,L} \subset H_+$ and $H_{-,L} \subset H_-$; that is, employ $\hat S^+_{z,\tau,L} := \Pi_{+, L} S^+_{z,\tau} \Pi_{+, L}$ and $\hat S^-_{z,\tau,L} := \Pi_{-,L} S^-_{z,\tau} \Pi_{-,L}$ as approximations of $S^+_{z,\tau}$ and $S^-_{z,\tau}$, respectively, where $\Pi_{\pm, L} : H \to H$ are orthogonal projections mapping onto $H_\pm$. With these operators, one could define $\hat R^+_{z,\tau,L} = a_z^{-1}(\hat S^+_{z,\tau,L})$ and $\hat R^-_{z,\tau,L} = J \circ \hat R^+_{z,\tau,L} \circ J$ analogously to the construction of $R^+_{z,\tau}$ and $R^-_{z,\tau}$ from \cref{sec:mainalg}, respectively, and assemble a finite-rank resolvent
\begin{equation}
    \hat R_{z,\tau,L} = \hat R^-_{z,\tau,L} \oplus R_z^0 \oplus \hat R^+_{z,\tau,L},
    \label{eq:rztl_tentative}
\end{equation}
acting on $H_{-,L} \oplus H_0 \oplus H_{+,L}$, with a corresponding finite-rank, skew-adjoint approximate generator $\hat V_{z,\tau,L} = \beta_z(\hat R_{z,\tau,L})$. Note that for well-definition of $\hat R_{z,\tau,L}$ via~\eqref{eq:rztl_tentative} it is important that $\hat R^+_{z,\tau,L}$ and $\hat R^-_{z,\tau,L}$ have $H_+$ and $H_-$ as mutually orthogonal invariant subspaces, respectively.  Arranging for the subspaces $H_{+,L}$ and $H_{-,L}$ to form increasing families towards $H_+$ and $H_-$, respectively, would then imply that $\hat R_{z,\tau,L}$ converges strongly to $R_{z,\tau}$ as $L \to \infty$.

Despite its apparent simplicity, a shortcoming of this approach is that it is difficult to construct appropriate bases for the approximation spaces $H_{+,L}$ and $H_{-,L}$. In particular, aside from special cases involving integral operators $G_\tau$ that commute with $\Pi_{\pm}$, there are no kernel integral operator techniques known to us that can produce orthonormal bases of $H_\pm$ via eigendecomposition (analogously to the eigenbases of $H$ associated with $G_\tau$). In the absence of such methods, we must content ourselves with approximations of $S^+_{z,\tau}$ and $S^-_{z,\tau}$ that do not preserve the invariance of the positive- and negative-frequency subspaces $H_\pm$ under these operators. Building the resolvent and finite-rank generator from these approximations will require solution of an additional eigenvalue problem, as we describe in the following subsection.

\subsection{Finite-rank approximation of the compactified resolvent and associated generator}
\label{sec:finite-rank_generator}

Throughout this section, we restrict the compactified resolvent $R_{z,\tau}$ from~\eqref{eq:rzt} to the codimension-1 subspace $\tilde H \subset H$ of zero mean functions. This restriction is without loss, since we can trivially recover $R_{z,\tau}$ on $H$ by adding to the restricted operator $R_{z,\tau} |_{\tilde H}$ the rank-1 operator $R_z^0(V)$ from~\eqref{eq:rz0}.

Viewing, then, $R_{z,\tau}$ as an operator on $\tilde H$, it follows from the fact that $\lVert G_\tau f\rVert_{\tilde H} < \lVert f \rVert_{\tilde H}$ for any nonzero $f \in \tilde H$ that the spectra of $S^+_{z,\tau}$ and $S^-_{z,\tau}$ are strict, closed subsets of $[0, z^{-1}]$ that do not contain $z^{-1}$. Thus, using $a_z : C_z^+ \to [0, z^{-1}] $ from \cref{sec:mainalg} and similarly defining $\bar a_z : C_z^- \to [0, z^{-1}]$ as $\bar a_z(\lambda) = \lvert \lambda \rvert$, we have that $\sigma(R_{z,\tau}) = a^{-1}_z(\sigma(S^+_{z,\tau})) \cup \bar a^{-1}_z(\sigma(S^-_{z,\tau}))$  is a subset of a closed arc in $C_z$ that does not contain $z^{-1}$; see Fig.~\ref{fig:gammaschem} for a schematic illustration.

\begin{figure}
    \centering
    \includegraphics[width=0.7\textwidth]{"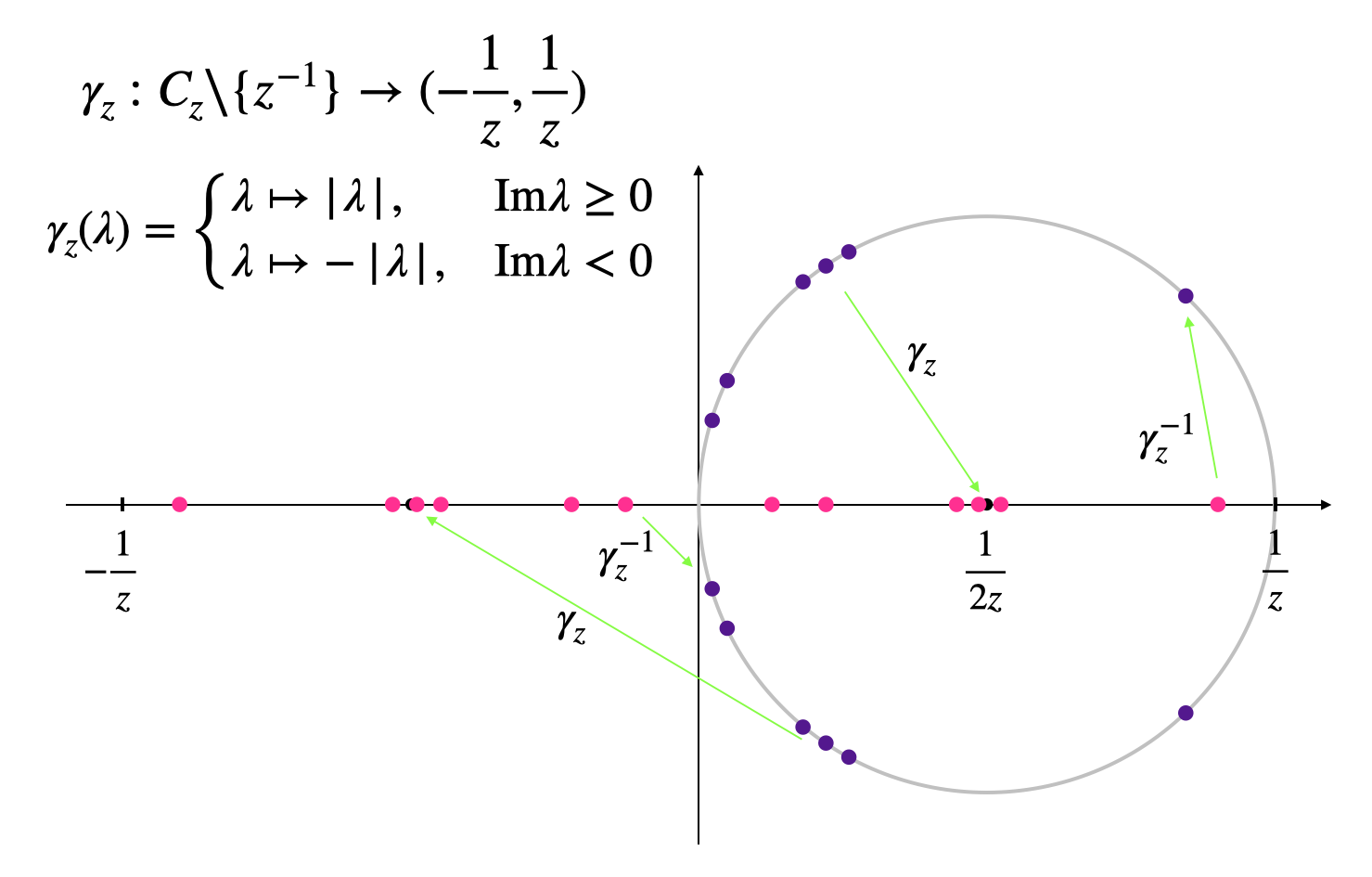"}
    \caption{Schematic of the bijective correspondence of the spectrum of the compactified resolvent $R_{z, \tau}$ (illustrated with purple dots), which lies in the circular arc $C_z \setminus \{ z^{-1} \}$, and the spectrum of the operator $S_{z,\tau}^+ - S_{z,\tau}^-$ (illustrated with magenta dots), which lies in the interval $(-z^{-1}, z^{- 1})$. This correspondence is realized through the function $\gamma_z$ as described in \cref{sec:finite-rank_generator}.}
    \label{fig:gammaschem}
\end{figure}

Let $\tilde C_z = C_z \setminus \{ z^{-1} \}$ and consider the map $\gamma_z : \tilde C_z \to (-z^{-1}, z^{-1})$ defined as
\begin{displaymath}
    \gamma_z(\lambda) =
    \begin{cases}
        \abs{\lambda}, \quad \im\lambda \geq 0,\\
        - \abs{\lambda}, \quad \im\lambda < 0.
    \end{cases}
\end{displaymath}
One readily verifies that $\gamma_z$ is a continuous, injective mapping that maps $\sigma(R_{z,\tau})$ into the closed interval $I_{z,\tau} = [-e_{1,\tau}, e_{1,\tau}]$ (i.e., $I_{z,\tau}$ is the smallest interval containing $\sigma(-S^-_{z,\tau}) \cup \sigma(S^+_{z,\tau})$). Moreover, we have $\gamma_z(R_{z,\tau}) = S^+_{z,\tau} - S^-_{z,\tau}$ and thus
\begin{displaymath}
    R_{z,\tau} = \gamma_z^{-1}(S^+_{z,\tau} - S^-_{z,\tau}).
\end{displaymath}
Note the explicit formula for the inverse of $\gamma_z$,
\begin{displaymath}
    \gamma_z^{-1}(x) =
    \begin{cases}
        a_z^{-1}(x), \quad x \in [0, z^{-1}),\\
        (a_z^{-1}(-x))^*, \quad x \in (-z^{-1}, 0),
    \end{cases}
\end{displaymath}
where $a_z^{-1}$ was defined in~\eqref{eq:ginv}. By continuity of $\gamma_z^{-1}$, we can build approximations of $R_{z,\tau}$ that are spectrally consistent in the sense of \cref{thm:spec_compact} by first approximating $S_{z,\tau} \equiv S^+_{z,\tau} - S^-_{z,\tau}$ by finite-rank operators with spectra contained in closed subintervals of $(-z^{-1}, z^{-1})$, and then applying $\gamma_z^{-1}$ to these operators, as we now describe.

Using the kernel eigenbasis $ \{ \phi_j \}_{j=1}^\infty$ of $\tilde H$ from \cref{sec:kernel}, define the $L$-dimensional approximation spaces $H_L = \spn \{ \phi_1, \ldots, \phi_L \}$ and the corresponding orthogonal projections $\Pi_L : \tilde H \to \tilde H$ with $\ran \Pi_L = H_L$. Define also the finite-rank, projected resolvents $\tilde R_{z,L}^+ : \tilde H \to \tilde H$ as $\tilde R_{z,L}^+ = \Pi_L R_z^+(V) \Pi_L$, and the corresponding positive operators $\tilde S_{z,L}^+ := \lvert \tilde R_{z,L}^+\rvert$. We then have:

\begin{lemma}
    For any $L \in \mathbb N$, the spectrum of $\tilde S_{z,L}^+$ is contained in the interval $[0, z^{-1}]$.
    \label{lem:szl_spec}
\end{lemma}

\begin{proof}
    Since $\tilde S_{z,L}^+$ is self-adjoint, we have $\lVert \tilde S_{z,L}^+\rVert = \sqrt{\lVert (\tilde S_{z,L}^+)^2\rVert} = \sqrt{\lVert \tilde R_{z,L}^{+*} \tilde R_{z,L}^+\rVert}$. Moreover, for any $f \in \tilde H$, we have
    \begin{align*}
        \langle f, \tilde R_{z,L}^{+*} \tilde R_{z,L}^+ f\rangle &= \langle \Pi_L R_z^+(V) \Pi_L f, \Pi_L R_z^+(V) \Pi_L f \rangle \\
                                                            &\leq \langle R_z^+(V) \Pi_L f, R_z^+(V) \Pi_L f\rangle = \langle \Pi_L f, (S_z^+)^2 \Pi_L\rangle \\
                                                            &\leq \lVert S_z^+\rVert^4 \lVert f\rVert_H^2,
    \end{align*}
    and therefore $\lVert \tilde R_{z,L}^{+*} \tilde R_{z,L}^+\rVert \leq \lVert S_z^+\rVert^2$. We thus conclude that $\lVert \tilde S_{z,L}^+\rVert \leq \lVert S_z^+\rVert$ and $\sigma(\tilde S_{z,L}^+) \subseteq \sigma(S_z^+) \subset [0, z^{-1}]$ since $S_{z,L}^+$ is positive.
\end{proof}

Analogously to $S_z^+$ from~\eqref{eq:Sz}, we can obtain $\tilde S_{z,L}^+$ from a polar decomposition of $\tilde R_{z,L}^+$,
\begin{displaymath}
    \tilde R_{z,L}^+ = W_{z,L} \tilde S_{z,L}^+,
\end{displaymath}
where $W_{z,L} \in B(\tilde H)$ is a partial isometry. Moreover, we introduce operators $\tilde R_{z,L}^-, \tilde S_{z,L}^- \in B(\tilde H)$ associated with the negative-frequency subspace similarly. Using $\tilde S_{z,L}^+$ and $\tilde S_{z,L}^-$, we define the self-adjoint operators $S^+_{z,\tau,L}: \tilde H \to \tilde H$ and $S^-_{z,\tau,L}: \tilde H \to \tilde H$ as
\begin{equation}
        S^+_{z,\tau,L} = \sqrt{\tilde S^+_{z,L}} G_\tau \sqrt{\tilde S^+_{z,L}}, \quad S^-_{z,\tau,L} = \sqrt{\tilde S^-_{z,L}} G_\tau \sqrt{\tilde S^-_{z,L}}.
    \label{eq:sztl}
\end{equation}
By \cref{lem:szl_spec} and the fact that $\Pi_L$ are orthogonal projections, the spectra of $S^+_{z,\tau,L}$ and $-S^-_{z,\tau,L}$ are subsets of the positive and negative halves of $I_{z,\tau}$, respectively. However, note that $\ran S^+_{z,\tau,L}$ and $\ran S^-_{z,\tau,L}$  may not be linearly independent subspaces of $H_L$.

We would like to assemble a single operator that approximates $S_{z,\tau}$ from $S^+_{z,\tau,L}$ and $S^-_{z,\tau,L}$, while avoiding issues due to possible colinearity of $\ran S^+_{z,\tau,L}$ and $\ran S^-_{z,\tau,L}$ (e.g., creation of a non-trivial nullspace of $S_{z,\tau,L}^+ - S_{z,\tau,L}^-$ with formally infinite corresponding eigenfrequencies). With that in mind, we further reduce the rank of $S^+_{z,\tau,L}$ and $S^-_{z,\tau,L}$ and subtract the resulting operators to yield our final approximation of $S_{z,\tau}$. In what follows, $\Psi_\tau^{+,(M)} : \tilde H \to \tilde H$ will denote the orthogonal projection onto the $M$-dimensional subspace of $\tilde H$ spanned by the leading $M$ eigenfunctions of $S^+_{z,\tau}$, and $S^{+, (M)}_{z,\tau} : \tilde H \to \tilde H$ will be the rank-$M$ operator $S^{+,(M)}_{z,\tau} = \Psi_\tau^{+,(M)} S^+_{z,\tau} \Psi_\tau^{+,(M)}$. We similarly define $S^{-,(M)}_{z,\tau} = \Psi_\tau^{-,(M)} S^-_{z,\tau} \Psi_\tau^{-,(M)}$, where $\Psi_\tau^{-,(M)}$ is the orthogonal projection onto $\spn \{ \psi_{1,\tau}^*, \ldots,\psi_{M,\tau}^* \}$, and set $S^{(M)}_{z,\tau} = S^{+,(M)}_{z,\tau} - S^{-,(M)}_{z,\tau}$. For later convenience, we also define $\Psi_\tau^{(M)} = \Psi_\tau^{-,(M)} + \Psi_\tau^{+,(M)}$.

Let us write down an eigendecomposition
\begin{displaymath}
    S^+_{z,\tau,L} \xi_{j,\tau,L} = e_{j,\tau,L} \xi_{j,\tau,L}, \quad j \in \mathbb N,
\end{displaymath}
where the eigenvalues $e_{j,\tau,L}$ are ordered in decreasing order (note that $e_{j,\tau,L} = 0$ for all $j > L$), and the corresponding eigenfunctions $\xi_{j,\tau,L}$ are chosen to be orthonormal in $\tilde H$. By \eqref{eq:rzt-}, $S^-_{z,\tau,L}$ admits the eigendecomposition
\begin{displaymath}
    S^-_{z,\tau,L} \xi^*_{j,\tau,L} = e_{j,\tau,L} \xi^*_{j,\tau,L}.
\end{displaymath}
Fixing a parameter $M \leq L$, we define the reduced-rank operators $S^{+,(M)}_{z,\tau,L}$ and $S^{-,(M)}_{z,\tau,L}$ as
\begin{equation}
    S^{+,(M)}_{z,\tau,L} = \Xi_{\tau,L}^{(M)} S^+_{z,\tau,L} \Xi_{\tau,L}^{(M)}, \quad S^{-,(M)}_{z,\tau,L} = \bar \Xi_{\tau,L}^{(M)} S^-_{z,\tau,L} \bar \Xi_{\tau,L}^{(M)},
    \label{eq:sztlm_pm}
\end{equation}
where $\Xi_{\tau,L}^{(M)} : \tilde H \to \tilde H$ and $\bar \Xi_{\tau,L}^{(M)}: \tilde H \to \tilde H$ are the orthogonal projections onto $\spn \{ \xi_{1,\tau,L}, \ldots, \xi_{M,\tau,L} \}$ and $\spn \{ \xi_{1,\tau,L}^*, \ldots, \xi_{M,\tau,L}^* \}$, respectively.
We then introduce the operators
\begin{equation}
    S^{(M)}_{z,\tau,L} = S_{z,\tau,L}^{+,(M)} - S_{z,\tau,L}^{-,(M)}.
    \label{eq:sztlm}
\end{equation}
These operators will provide our finite-rank approximation of $S_{z,\tau}$. This approximation converges spectrally in the sense of the following proposition.

\begin{proposition}
    With notation as above, let $M \in \mathbb N$ be such that $e_{M+1,\tau} \neq e_{M,\tau}$; that is, $\Psi_\tau^{(M)}$ projects onto an $2M$-dimensional union of eigenspaces of $S^{(M)}_{z,\tau}$. Then, there exists $L_* \in \mathbb N$ such that for all $L > L_*$ the following hold.
    \begin{enumerate}
        \item $S_{z,\tau,L}^{(M)}$ has rank $2M$.
        \item There exists a closed subinterval $\tilde I_{z,\tau}$ of $(-z^{-1}, z^{-1})$ that contains $\sigma(S_{z,\tau,L}^{(M)})$ and $\sigma(-S^-_{z,\tau}) \cup \sigma(S^+_{z,\tau})$.
        \item As $L\to\infty$, $S^{(M)}_{z,\tau,L}$ converges in operator norm, and thus spectrally in the sense of \cref{thm:spec_compact}, to $S_{z,\tau}^{(M)}$.
        \item As $M \to \infty$, $S_{z,\tau}^{(M)}$ converges in operator norm, and thus spectrally in the sense of \cref{thm:spec_compact}, to $S_{z,\tau}$.
    \end{enumerate}
    \label{prop:spec_s}
\end{proposition}

\begin{proof}
    See \cref{sec:proof_spec_s}.
\end{proof}

By \cref{prop:spec_s}(iii), for fixed $M$ and sufficiently large $L$, $S_{z,\tau,L}^{(M)}$ has exactly $2M$ nonzero eigenvalues that lie in the interval $\tilde I_{z,\tau}$. We will denote these eigenvalues as $e^{(M)}_{j,\tau,L}$ with $j \in \{ -M, \ldots, -1, 1, \ldots, M \}$, and order them as
\begin{displaymath}
    e_{1,\tau,L}^{(M)} \geq \cdots \geq e_{M,\tau,L}^{(M)} > 0 > e_{-M,\tau,L}^{(M)} \geq \cdots \geq e_{-1,\tau,L}^{(M)}.
\end{displaymath}
We let $\psi_{j,\tau,L}^{(M)} \in \tilde H$ be an orthonormal set of eigenfunctions corresponding to $e_{j,\tau,L}^{(M)}$. From~\eqref{eq:sztlm}, we see that the eigenvalues satisfy $e^{(M)}_{-j,\tau,L} = - e^{(M)}_{j,\tau,L}$, and the corresponding eigenfunctions can be chosen so as to satisfy $\psi^{(M)}_{-j,\tau,L} = \psi^{(M)*}_{j,\tau,L}$. Let $\Psi_{\tau,L}^{(M)} : \tilde H \to \tilde H$ be the orthogonal projection onto $\ran S_{z,\tau,L}^{(M)}$ (i.e., the $2M$-dimensional span of the $\psi_{j,\tau,L}^{(M)}$). By \cref{prop:spec_s}(iv), we have that $\Psi_{\tau,L}^{(M)}$ converges strongly to $\Psi_\tau^{(M)}$ as $L \to \infty$.

Next, define the finite-rank operator $\tilde R_{z,\tau,L}^{(M)} = \gamma_z^{-1}(S_{z,\tau,L}^{(M)})$. Note that the well-definition of $\tilde R_{z,\tau,L}^{(M)}$ relies on the fact that the spectrum of $S_{z,\tau,L}^{(M)}$ is a subset of $\tilde I_{z,\tau}$. In particular, $\tilde R_{z,\tau,L}^{(M)}$ has exactly $2M$ nonzero eigenvalues
\begin{equation}
    \vartheta_{j,\tau,L}^{(M)} := \gamma_z^{-1}(e_{j,\tau,L}^{(M)}),
    \label{eq:lambda_approx}
\end{equation}
which come in complex-conjugate pairs, $\vartheta_{j,\tau,L}^{(M) *} = \vartheta_{-j,\tau,L}^{(M)}$, and have $\psi_{j,\tau,L}^{(M)}$ as corresponding eigenfunctions. For these eigenvalues, we define the eigenfrequencies (cf.\ \eqref{eq:vzt})
\begin{equation}
    \omega_{j,\tau,L}^{(M)} = \frac{1}{i} \beta_z(\vartheta_{j,\tau,L}^{(M)}).
    \label{eq:omega_approx}
\end{equation}
By continuity of $\gamma_z^{-1}$ on the closed interval $\tilde I_{z,\tau}$ and \cref{prop:spec_s}(iii), we have that $\tilde R_{z,\tau,L}^{(M)}$ converges in operator norm to $\tilde R_{z,\tau}^{(M)} := \Psi_\tau^{(M)} R_{z,\tau} \Psi_\tau^{(M)}$ as $L \to\infty$.

Letting $\Psi_{\tau}^{(M),\perp} = I - \Psi_\tau^{(M)}$ and $ \Psi_{\tau,L}^{(M),\perp} = I - \Psi_{\tau,L}^{(M)}$ be the complementary projectors to $\Psi_\tau^{(M)}$ and $\Psi_{\tau,L}^{(M)}$, respectively, we define the operators
\begin{displaymath}
    R_{z,\tau,L}^{(M)} = \tilde R_{z,\tau,L}^{(M)} + z^{-1} \Psi_{\tau,L}^{(M),\perp}, \quad R_{z,\tau}^{(M)} = \tilde R_{z,\tau}^{(M)} + z^{-1} \Psi_{\tau}^{(M),\perp}.
\end{displaymath}
These operators have discrete spectra contained in $C_z \setminus \{ 0 \}$ and are resolvents of skew-adjoint operators
\begin{displaymath}
    V_{z,\tau,L}^{(M)} = \beta_z(R_{z,\tau,L}^{(M)}), \quad V_{z,\tau}^{(M)} = \beta_z(R_{z,\tau}^{(M)}) \equiv \Psi_\tau^{(M)} V_{z,\tau} \Psi_\tau^{(M)},
\end{displaymath}
respectively, that have rank $2M$. In particular, the nonzero eigenfrequencies of $V_{z,\tau,L}^{(M)}$ are given by $\omega_{j,\tau,L}^{(M)}$ from~\eqref{eq:omega_approx}, and have $\psi_{j,\tau,L}^{(M)}$ as corresponding orthonormal eigenfunctions. The nonzero eigenfrequencies of $V_{z,\tau}^{(M)}$ are identical to the first $2M$ eigenfrequencies $\omega_{j,\tau}$ of $V_{z,\tau}$ ranked in order of increasing modulus.

We will employ the resolvent $R_{z,\tau,L}^{(M)}$ as an approximation of the compactified resolvent $R_{z,\tau}$ that is accessible from the approximation spaces $\tilde H_L$. By the convergences $\tilde R_{z,\tau,L}^{(M)} \nto \tilde R_{z,\tau}^{(M)}$ and $\Psi_{\tau,L}^{(M)} \sto \Psi_\tau^{(M)}$, we have $R_{z,\tau,L}^{(M)} \sto R_{z,\tau}^{(M)}$ as $L\to\infty$. Moreover, from $\Psi_\tau^{(M)} \sto I$ (which implies $\tilde R_{z,\tau}^{(M)} \nto R_{z,\tau} $) and $\Psi_{\tau}^{(M),\perp} \sto 0$, we get $R_{z,\tau}^{(M)} \sto R_{z,\tau}$ as $M\to\infty$. We therefore conclude that our approximation of $V_{z,\tau}$ by the finite-rank operators $V_{z,\tau,L}^{(M)}$ converges in strong resolvent sense in the iterated limit of $M\to\infty$ after $L\to\infty$.

\subsection{Proof of \cref{prop:spec_s}}
\label{sec:proof_spec_s}

We begin by recalling the following Hilbert space result.

\begin{lemma}
    Let $A: \mathbb H \to \mathbb H$ be a compact, self-adjoint operator on a Hilbert space $\mathbb H$ and $B_1, B_2, \ldots$ a uniformly bounded sequence of self-adjoint operators on $\mathbb H$ that converges strongly to the identity. Then, the sequence $A_1, A_2, \ldots$ with $A_n = B_n A B_n$ converges to $A$ in operator norm.
    \label{lem:compact_conv}
\end{lemma}

\begin{proof}
    Since $A$ is compact and $B_n$ is a uniformly bounded sequence converging strongly to the identity, the sequence $B_n A$ converges to $A$ in operator norm; that is,
    \begin{equation}
        \lim_{n\to\infty} \lVert (I - B_n) A\rVert = 0.
        \label{eq:AB_conv}
    \end{equation}
    Letting $\beta = \sup_n \lVert B_n\rVert$, we have
    \begin{align*}
        \lVert A - A_n \rVert &= \lVert A - B_n A B_n \rVert = \lVert A - B_n A + B_n A - B_n A B_n A \rVert \\
                              &\leq \lVert (I - B_n) A \rVert + \lVert B_n A (I - B_n)\rVert  \leq \lVert (I - B_n) A \rVert + \beta \lVert A(I - B_n)\rVert \\
                              &= (1 + \beta) \lVert (I - B_n) A \rVert,
    \end{align*}
    where we have used self-adjointness of $A$ and $B_n$ to obtain the last equality. The claim of the lemma follows from~\eqref{eq:AB_conv}.
\end{proof}

We use \cref{lem:compact_conv} to prove claims~(iii) and~(iv) of \cref{prop:spec_s}.

\subsubsection*{Proof of claims~(iii) and~(iv).}

Consider first the operators $S_{z,\tau,L}^+$ from~\eqref{eq:sztl}. Since $S_{z,\tau}^+$ is self-adjoint and compact, and $\Pi_L$ is a sequence of orthogonal projections converging strongly to the identity on $\tilde H$, it follows from \cref{lem:compact_conv} that, as $L\to\infty$, $S_{z,\tau,L}^+$ converges to $S_{z,\tau}^+$ in operator norm. It is known that operator norm convergence implies compact convergence \cite[section~3.3]{Chatelin11}. Thus, $S_{z,\tau,L}^+$ converges spectrally to $S_{z,\tau}^+$ in the sense of \cref{thm:spec_compact}. In particular, for any $M$ such that $\Psi_\tau^{+,(M)}$ projects onto a union of eigenspaces of $S_{z,\tau}^+$, we have that the projections $\Xi_{\tau,L}^{(M)}$ converge strongly to $\Psi_\tau^{+,(M)}$. Together, the norm convergence of $S_{z,\tau,L}^+$ to $S_{z,\tau}^+$ and the strong convergence of $\Xi_{\tau,L}^{(M)}$ to $\Psi_\tau^{+,(M)}$ imply that, as $L\to \infty$, $S^{+,(M)}_{z,\tau,L}$ converges to $S_{z,\tau}^{+,(M)}$ in norm.

Repeating these arguments for $S_{z,\tau,L}^-$, we can deduce that $S_{z,\tau,L}^{-,(M)}$ converges to $S_{z,\tau}^{-,(M)}$ in norm. We therefore conclude that $S^{(M)}_{z,\tau,L}$ converges to $S_{z,\tau}^{(M)}$ in norm, and thus spectrally in the sense of \cref{thm:spec_compact}. This proves claim~(iii). The norm convergence of $S_{z,\tau}^{(M)} = \Psi_\tau^{(M)} S_{z,\tau} \Psi_\tau^{(M)}$ in claim~(iv) follows similarly from the fact that $S_{z,\tau}$ is self-adjoint and compact and $\Psi_\tau^{(M)}$ is a sequence of orthogonal projections converging strongly to the identity on $\tilde H$ as $M\to\infty$.

We now continue with the proof of the remaining claims.

\subsubsection*{Proof of claims~(i) and~(ii).}

First, note that the nonzero eigenvalues $e_{j,\tau}^{(M)}$ of $S_{z,\tau}^{(M)}$ satisfy $\lvert e_{j,\tau}^{(M)}\rvert \leq e_{1,\tau}$ and $\lvert e_{j,\tau}^{(M)}\rvert \geq e_{M,\tau}$ by construction.

Fix any $e_\text{max} \in (e_{1,\tau}, z^{-1})$, $e_\text{min} \in (e_{M+1,\tau}, e_{M,\tau})$, and define $\tilde I_{z,\tau} = [-e_\text{max}, e_\text{max}]$. By claim~(iii) and \cref{thm:spec_compact}(i), there exists $L_* \in \mathbb N$ such that for every $L > L_*$, the multiplicities of the eigenvalues of $S_{z,\tau,L}^{(M)}$ in the set $(-e_\text{max}, -e_\text{min}) \cup (e_\text{min}, e_\text{max})$ sum up to $2M$, i.e., $\rank S_{z,\tau,L}^{(M)} \geq 2M$. Since $\rank S_{z,\tau,L}^{(M)} \leq 2M$ by construction, we have $\rank S_{z,\tau,L}^{(M)} = 2M$, proving claim~(ii). Moreover, we have  $\sigma(S_{z,\tau,L}^{(M)}) \subset (-e_\text{max}, -e_\text{min}) \cup \{ 0 \} \cup (e_\text{min}, e_\text{max})$ and $\sigma(S_{z,\tau}) \subset I_{z,\tau} \subset \tilde I_{z,\tau} $. This verifies claim~(i) and completes our proof of \cref{prop:spec_s}.

\section{Numerical implementation}
\label{sec:numimplement}

In this section, we describe a numerical procedure for approximating the compactified resolvent and associated approximate generator from \cref{sec:theory,sec:finite-rank} using data sampled along dynamical trajectories. The first step in the procedure is to build a data-driven dictionary of observables that asymptotically approximates an orthonormal basis of the infinite-dimensional Hilbert space $H$ using eigenfunctions of kernel integral operators. This step is identical to the procedure employed in \cite{DasEtAl21}. In \cref{sec:basis}, we give a cursory overview of the basis construction following a description of the training dataset in \cref{sec:training}. In \cref{sec:op-approx}, we describe a scheme for data-driven operator approximation that employs this basis. Using this scheme and the finite-rank approximations developed in \cref{sec:finite-rank}, we build an approximation of the projection $\Pi_+$ to the positive-frequency subspace and an approximation of the integral formula~\eqref{eq:resolventintegral} for the resolvent via numerical quadrature. These steps are discussed in \cref{sec:filtering,sec:resolvent-quad}, respectively. In \cref{sec:num-compactification}, we describe data-driven analogs of the resolvent compactification and rank reduction schemes from \cref{sec:theory,sec:finite-rank} leading to the approximate resolvent $R_{z,\tau,L}^{(M)}$ and finite-rank generator $V_{z,\tau,L}^{(M)}$. Finally, in \cref{sec:num-eig} we build data-driven approximations of $R_{z,\tau,L}^{(M)}$ and $V_{z,\tau,L}^{(M)}$, and compute the corresponding eigenfrequency--eigenfunction pairs. A high-level summary of the entire procedure is displayed in \cref{alg:numerical}.

\begin{algorithm}
    \caption{Numerical operator approximation for selected resolvent parameter $z >0$, regularization parameter $\tau > 0$, number of basis functions $L \in \mathbb N$, rank parameter $M \leq L$, number of resolvent quadrature nodes $Q \in \mathbb N$, and number of samples $N \in \mathbb N$. The algorithm computes the $L\times L$ matrix representation $\bm V_{z,\tau,N}^{(M)}$ of the finite-rank generator $V_{z,\tau,L,N}^{(M)}$ from~\eqref{eq:res-final-approx}, along with its nonzero eigenfrequencies $\omega_{j,\tau,L,N}^{(M)}$ from~\eqref{eq:eig_gamma_omega} and the corresponding eigenfunctions $\psi_{j,\tau,L,N}^{(M)}$ from~\eqref{eq:psi_data_driven}. $\overline{(\cdot)}$ denotes elementwise complex-conjugation of matrices.}
  \label{alg:numerical}
  \begin{algorithmic}[1]
      \STATE{Compute kernel eigenvectors $ \{ \phi_{j,N}\}_{j=1}^L$ and corresponding eigenvalues $ \{ \lambda_{j,N}^{\tau} \}_{j=1}^L$ from~\eqref{eq:eigGN}. Set $\bm \Lambda_{\tau,N} = \diag(\lambda_{1,N}^{\tau}, \ldots, \lambda_{L,N}^{\tau}) \in \mathbb R^{L\times L}$.}
      \STATE{Filter basis to have only positive Fourier frequencies using~\eqref{eq:Pi_DFT_time_series}: $\phi_{j,N}^+ = \hat \Pi_{+,N} \phi_{j,N} $}.
      \STATE{Compute projected resolvent matrix $\bm{R}_{z, N}^+ \in \mathbb C^{L\times L}$ using the quadrature formula \eqref{eq:res-mat} or~\eqref{eq:res-mat2} with $Q$ nodes and projected shift operator matrix $\bm U_N^+$ obtained from $\phi_{j,N}^{+}$.}
      \STATE{Compute polar decomposition $\bm{R}_{z, N}^{+} = \bm{W}_{z,N} \bm{S}^+_{z,N}$, where $\bm{W}_{z,N}$ is unitary and $\bm{S}^+_{z, N}$ is positive.}
  \STATE{Compute matrix square root $\sqrt{\bm S^+_{z,N}}$}.
  \STATE{Apply regularization: $\bm{S}^+_{z,\tau,N} = \sqrt{\bm{S}^+_{z,N}} \bm\Lambda_{\tau,N}\sqrt{\bm{S}^+_{z,N}}$.}
  \STATE{Compute eigendecomposition of $\bm{S}^+_{z,\tau,N}$, with eigenvalues $e_{l,\tau,L,N}$ and orthonormal eigenvectors $\bm c_{l,\tau,N} \in \mathbb C^L$. }
  \STATE{Set $\bm C_{\tau,N}^{(M)} = (\bm c_{1,\tau,N}, \ldots, \bm c_{M,\tau,N}) \in \mathbb C^{L\times M}$  and $\bm E_{z,\tau,N}^{(M)} = \diag(e_{1,\tau,L,N}, \ldots, e_{M,\tau,L,N}) \in \mathbb R^{M\times M}$. Form the reduced-rank matrices $\bm S_{z,\tau,N}^{+,(M)} = \bm C_{\tau,N}^{(M)} \bm E_{z,\tau,N}^{(M)} \bm C_{\tau,N}^{(M)*}$, $\bm S_{z,\tau,N}^{-,(M)} = \overline{\bm S_{z,\tau,N}^{+,(M)}}$, and $\bm S_{z,\tau,N}^{(M)} = \bm S_{z,\tau,N}^{+,(M)} - \bm S_{z,\tau,N}^{-,(M)}$.}
  \STATE{Compute eigendecomposition of $\bm S_{z,\tau,N}^{(M)}$ with eigenvalues $\vartheta_{l,\tau,L,N}^{(M)}$ and orthonormal eigenvectors $\bm q_{l,\tau,N}^{(M)} \in \mathbb C^L$. Form eigenfunctions $\psi_{l,\tau,L,N}^{(M)}$ with expansion coefficients $\bm q_{l,\tau,N}^{(M)}$ with respect to the $\phi_{j,N}$ basis.}
  \STATE{Set $\bm \Omega_{\tau,N}^{(M)} = \diag(\omega_{-M,\tau,L,N}^{(M)}, \ldots, \omega_{-1,\tau,L,N}^{(M)}, \omega_{1,\tau,L,N}^{(M)}, \ldots, \omega_{M,\tau,L,N}^{(M)}) \in \mathbb R^{L\times L}$, $\bm Q^{(M)}_{\tau,N} = (\bm q_{-M,\tau,N}^{(M)}, \ldots, \bm q_{-1,\tau,N}^{(M)}, \bm q_{1,\tau,N}^{(M)}, \ldots, \bm q_{M,\tau,N}^{(M)}) \in \mathbb C^{L\times 2M}$. Form the matrix $V_{z,\tau,N}^{(M)} = i \bm Q^{(M)}_{\tau,N} \bm\Omega_{\tau,N}^{(M)} \bm Q^{(M)*}_{\tau,N}$.}
    \RETURN{Generator matrix $\bm V_{z,\tau,N}^{(M)}$ and eigenpairs $ \{ (\omega_{j,\tau,L,N}^{(M)}, \psi_{j,\tau,L,N}^{(M)})\}$}.
  \end{algorithmic}
\end{algorithm}

\subsection{Training data}
\label{sec:training}

We make the following standing assumptions on the training data.
\begin{enumerate}[label=(A\arabic*)]
    \item \label[assump]{assump:a1} We have access to time-ordered samples $y_0, \ldots, y_{N-1} $ in a data space $Y$ obtained from a continuous map $F: \mathcal M \to Y$ that is injective on the forward-invariant set $M \subseteq \mathcal M$; that is, $y_n = F(x_n)$ with $x_n \in \mathcal M$. The samples are taken at a fixed sampling interval $\Delta t > 0$ along a dynamical trajectory $x_0, \ldots, x_{N-1} \in \mathcal M$ with $x_n = \Phi^{n\,\Delta t}(x_0)$.
    \item The invariant measure $\mu$ is ergodic under the discrete-time map $\Phi^{\Delta t}$.
    \item \label[assump]{assump:a3} The initial point $x_0$ (and thus the entire orbit $x_0,x_1,\ldots$) lies in $M$, and moreover it lies in the basin of $\mu$. This means that
        \begin{equation}
            \label{eq:mu-av}
            \lim_{N\to\infty}\frac{1}{N} \sum_{n=0}^{N-1} f(x_n) = \int_X f \, d\mu, \quad \forall f \in C(M).
        \end{equation}
    \setcounter{assump}{\value{enumi}}
\end{enumerate}

Note that we do not require knowledge of the state space trajectory $x_0, \ldots, x_{N-1}$. Moreover, while~\eqref{eq:mu-av} holds for $\mu$-a.e.\ $x_0 \in \mathcal M$ by ergodicity, we do not assume that the $x_n $ lie in the support of $\mu$ (which may be a null set with respect to an ambient measure on $\mathcal M$, such as an attractor of a dissipative system). Invariant measures for which~\eqref{eq:mu-av} holds for initial conditions $x_0$ drawn from a positive-measure set with respect to an ambient measure (e.g., Lebesgue measure) are called observable, or physical; e.g., \cite{Young02,Blank17}.

Henceforth, we will let $\mu_N = \sum_{n=0}^{N-1} \delta_{x_n} / N$ denote the sampling measure supported on the trajectory $x_0, \ldots, x_{N-1}$. Equation~\eqref{eq:mu-av} can then be understood as weak-$^*$ convergence of $\mu_N$ to $\mu$ as $N\to\infty$ at fixed $\Delta t$. We also let $\hat H_N = L^2(\mu_N)$ be the $L^2$ space associated with $\mu_N$, equipped with the inner product $\langle f, g \rangle_N = \sum_{n=0}^{N-1} f^*(x_n) g(x_n)/N$, where the elements of $\hat H_N$ are equivalence classes of complex-valued functions on $\mathcal M$ with common values on the trajectory $x_0, \ldots, x_{N-1}$. For simplicity, we shall assume that the points $x_0, \ldots, x_{N-1}$ are distinct---under \crefrange{assump:a1}{assump:a3} this holds aside from trivial cases such as $\mu$ being supported on a singleton set containing a fixed point. The Hilbert space $\hat H_N$ is then $N$-dimensional, and there is a bijective correspondence between elements $f \in \hat{H}_N $ and $N$-dimensional column vectors $\bm f = (f(x_0), \ldots, f(x_{N-1}) )^\top \in \mathbb C^N$. Note that if $f$ is a unit vector with $\lVert f\rVert_{\hat H_N} = 1$ then $\bm f$ has 2-norm $\lVert \bm f\rVert_2 = \sqrt{N}$. We let $\iota_N : C(M) \to \hat H_N$ be the restriction map on continuous functions into $\hat H_N$, i.e., $\iota_N f = \hat f$ where $\hat f(x_n) = f(x_n)$ for all $n \in \{ 0, \ldots, N-1 \}$ (cf.\ $\iota: C(M) \to H$ from \cref{sec:dynamics}). We also let $\tilde H_N$ be the $(N-1)$-dimensional subspace of $\hat H_N$ consisting of zero-mean functions with respect to the sampling measure, i.e., $\tilde H_N = \{ f \in \hat H_N : \langle \bm 1, f \rangle_N \equiv \int_{\mathcal M} f \, d\mu_N = 0 \} $ (cf.\ $\tilde H \subset H$ from \cref{sec:dynamics}). In accordance with the notation laid out in \cref{sec:dynamics}, $\lVert \hat A\rVert$ will denote the operator norm of $\hat A \in B(\hat H_N)$, which is equivalent in this case to any norm on $N\times N$ complex matrices.

\begin{remark}[single vs.\ multiple trajectories]
    While in \crefrange{assump:a1}{assump:a3} we consider that the samples are taken on a single dynamical trajectory, the methods described below can be generalized to training with data from ensembles of shorter trajectories. The basic requirements are that the corresponding sampling measures converge to $\mu$ in the sense of~\eqref{eq:mu-av} and the individual trajectories in the ensemble are sufficiently long for the frequency filtering and quadrature methods in \cref{sec:filtering,sec:resolvent-quad} to be applied.
\end{remark}

\subsection{Basis construction}
\label{sec:basis}

Recall from \cref{sec:kernel} that the integral operators $G_\tau$ used for resolvent compactification are induced from Markovian kernels $p_\tau: \mathcal M \times \mathcal M \to \mathbb R_+$. For the purposes of data-driven approximation, we assume that $p_\tau$ is the pullback of a kernel $q_\tau : Y \times Y \to \mathbb R_+ $ on data space; that is, $p_\tau(x,x') = q_\tau(F(x),F(x'))$. We also assume that available to us is a family of continuous kernels $q_{\tau,N} : Y \times Y \to \mathbb R_+$ and their pullbacks $p_{\tau,N} : \mathcal M \times \mathcal M \to \mathbb R_+$, $p_{\tau,N}(x,x') = q_{\tau,N}(F(x),F(x'))$, such that $p_{\tau,N}$ is symmetric and Markovian with respect to $\mu_N$, i.e.,
\begin{equation*}
    p_{\tau,N}(x,x') = p_{\tau,N}(x',x), \quad p_{\tau,N}(x, x') \geq 0, \quad \int_M p_{\tau,N}(x,\cdot)\,d\mu_N = 1, \quad \forall x, x' \in \mathcal M.
\end{equation*}
We require that, as $N\to\infty$, $p_{\tau,N}$ converges to $p_\tau$, uniformly on $M \times M$. The data-dependent kernels $p_{\tau,N}$ can be constructed analogously to $p_\tau$ by normalization of a data-independent kernel $k: \mathcal M \times \mathcal M \to \mathbb R_+$; see \ref{app:markov}.

Associated with each $p_{\tau,N}$ is a self-adjoint Markov operator $G_{\tau,N} : \hat H_N \to \hat H_N$ defined as $G_{\tau,N} f  = \int_M p_{\tau,N}(\cdot,x) f(x)\,d\mu_N(x)$. Concretely, we represent $G_{\tau,N}$ by the $N\times N$ kernel matrix $\bm G_\tau = [G_{ij}]$ with entries $G_{ij} = p_{\tau,N}(x_i,x_j) / N$. This matrix is a bistochastic matrix (i.e., $\bm G^\top_\tau = \bm G_\tau$ and $\sum_{j=0}^{N-1} G_{ij} = 1$ for all $i \in \{ 0, \ldots, N-1 \}$), and it has the property that if $f \in \hat H_N$ is represented by the column vector $\bm f = (f(x_0), \ldots, f(x_{N-1}))^\top \in \mathbb C^N$, then $\bm g = \bm G_\tau \bm f$ with $\bm g = (g_0, \ldots, g_{N-1})^\top$ is the column matrix representation of $g = G_{\tau,N} f$, i.e., $g_i = g(x_i)$.

Correspondingly, an eigendecomposition of $\bm G_\tau$,
\begin{equation*}
    \bm G_\tau \bm \phi_j = \lambda^\tau_{j,N} \bm \phi_j, \quad \bm \phi_j = (\phi_{0,j,N}, \ldots, \phi_{N-1,j,N})^\top,
\end{equation*}

where the eigenvalues are ordered as $1=\lambda^\tau_{0,N} \geq \lambda^\tau_{1,N} \geq \cdots \geq \lambda^\tau_{N-1,N}$ and the eigenvectors are chosen such that $\bm \phi_0 = (1, \ldots, 1)$ and $\bm \phi_i^\top \bm \phi_j = N \delta_{ij}$, provides an orthonormal basis $ \{ \phi_{0,N},\ldots,\phi_{N-1,N}\}$ of $\hat H_N$ with $\phi_{j, N}(x_n) = \phi_{n, j, N}$, consisting of eigenvectors of $G_{\tau,N}$,
\begin{equation}
    G_{\tau,N} \phi_{j,N} = \lambda^\tau_{j,N} \phi_{j,N}.
    \label{eq:eigGN}
\end{equation}
Similarly to the eigenvectors of $G_\tau$ corresponding to nonzero eigenvalues (see \cref{sec:kernel}), the eigenvectors $\phi_{j,N}$ corresponding to $\lambda_{j,N}^\tau \neq 0$ have continuous representatives $\varphi_{j,N} \in C(M)$ given by
\begin{equation}
    \label{eq:varphi}
    \varphi_{j,N}(x) = \frac{1}{\lambda_{j,N}^\tau} \int_M p_{\tau,N}(x,x')\phi_{j,N}(x')\,d\mu_N(x') = \frac{1}{\lambda_{j,N}^\tau N} \sum_{n=0}^{N-1} p_{\tau,N}(x,x_n) \phi_{n, j, N}.
\end{equation}
Note that since $p_{\tau,N}(x,x_n) = q_{\tau,N}(F(x), y_n)$ and $q_{\tau,N}(\cdot,y_n)$ are known functions, the eigenfunctions $\varphi_{j,N}$ can be evaluated at any $x \in M$ given the corresponding value $y = F(x)$ in data space.

The following convergence result for the eigenbasis of $G_{\tau,N}$ is based on spectral approximation techniques for kernel integral operators on spaces of continuous functions \cite{VonLuxburgEtAl08}.
\begin{lemma}
    \label{lem:basis}
    Under the assumptions on the kernel and training data stated in \cref{sec:kernel,sec:training,sec:basis} the following hold.
    \begin{enumerate}
        \item For every (nonzero) eigenvalue $\lambda^\tau_j$ of $G_\tau$,  $\lim_{N\to\infty} \lambda_{j,N}^\tau = \lambda^\tau_j$, including multiplicities in the sense of \cref{thm:spec_compact}(i).
        \item For every eigenfunction $\phi_j \in H$ of $G_\tau$ corresponding to nonzero eigenvalue $\lambda^\tau_j$, there exists a sequence of eigenfunctions $\phi_{j,N} \in \hat H_N$ of $G_{\tau,N}$ such that $\lim_{N\to\infty} \lVert \varphi_{j,N} - \varphi_j \rVert_{C(M)} = 0$, where $\varphi_j, \varphi_{j,N} \in C(M)$ are the continuous representatives of $\phi_j$  and $\phi_{j,N}$, respectively.
    \end{enumerate}
\end{lemma}
\begin{proof}
    See \cite[Appendix~A]{Giannakis21a}.
\end{proof}

\Cref{lem:basis} in conjunction with ergodicity of $G_\tau$ (\cref{prop:K5}) imply that for sufficiently large $N$, $\lambda_{0,N}^\tau$ is a simple eigenvalue. Henceforth, we will assume that this is the case. A sufficient condition for $\lambda_{0,N}^\tau$ to be simple is that $\bm G_\tau$ has strictly positive elements.

\subsection{Operator approximation}
\label{sec:op-approx}

We use the results of \cref{sec:basis} to perform data-driven approximation of operators on $\tilde H$.

Similarly to \cref{sec:finite-rank_generator}, for an integer parameter $L$ consider the $L$-dimensional subspace $H_L \subset \tilde H$ spanned by the leading $L$ nonconstant eigenfunctions of $G_\tau$, i.e., $H_L = \spn \{\phi_1, \ldots, \phi_L \}$ with corresponding orthogonal projection $\Pi_L : \tilde H \to \tilde H$. Henceforth, we shall assume that $L$ is chosen such that $\lambda^\tau_L \neq \lambda^\tau_{L+1}$ so that $H_L$ is a union of eigenspaces of $G_\tau$. Given a bounded operator $A \in B(\tilde H)$, we define its compression $A_L \in B(\tilde H)$ as $A_L = \Pi_L A \Pi_L$. Since $\{\phi_1, \ldots \}$ is an orthonormal basis of $\tilde H$, as $L\to\infty$, the operators $A_L$ converge to $A$ in the strong operator topology of $B(\tilde H)$. Note that we can equivalently view $A_L$ as an operator from $H_L$ into itself. In particular, $A_L$ is completely characterized through its $L\times L$ matrix representation $\bm A = [A_{ij}]_{i,j=1}^L$ with $A_{ij} = \langle \phi_i, A \phi_j \rangle$.

To make an analogous construction in the data-driven setting, we define $\hat{H}_{L,N} = \spn\{\phi_{1,N}, \ldots, \phi_{L,N}\} \subseteq \tilde H_N$ via the eigenvectors $\phi_{j,N}$ of $G_{\tau,N}$, and use the orthogonal projections $\Pi_{L,N} : \hat H_N \to \hat H_N$ with $\ran\Pi_{L,N} = \hat{H}_{L,N}$ to define compressions $\hat{A}_{L,N} := \Pi_{L,N} \hat A_N \Pi_{L,N}$ of operators $\hat A_N \in B(\hat H_N)$. The operators $\hat{A}_{L,N}$ are represented by $L\times L$ matrices $\bm A_N = [\hat A_{ij}]_{i,j=1}^L$ with entries $A_{ij} = \langle \phi_{i,N}, \hat A_N \phi_{j,N}\rangle_N$. Note that we do not ask that $\hat A_N$ preserves the subspace of zero-mean functions $\tilde H_N \subset \hat H_N$. Note also that $\bm A_N$ can be computed numerically so long as the action of $\hat A_N$ on the basis elements $\phi_{j,N}$ is known.

Next, consider a bounded operator $A \in B(\tilde H)$ and a sequence of operators $\hat A_1, \hat A_2, \ldots$ on $\hat H_N$ that approximates $A$ in the following sense:
\begin{enumerate}[label=(A\arabic*)]
    \setcounter{enumi}{\value{assump}}
    \item \label[assump]{assump:op-approx1} The family $\hat A_N$ is uniformly bounded, i.e., $\sup_N \lVert \hat A_N\rVert < \infty $.
    \item \label[assump]{assump:op-approx2} There is an operator $\tilde A : \tilde C(M) \to \tilde C(M)$ on zero-mean continuous functions such that
        \begin{enumerate}
            \item $ \iota \tilde A = A \iota $.
            \item For every $ f \in \tilde C(M)$, $\lim_{N\to\infty}\lVert (\iota_N \tilde A - \hat A_N \iota_N) f\rVert_{\hat H_N} = 0$.
        \end{enumerate}
\end{enumerate}
Given such an approximation, and a choice of dimension parameter $L$, the $L\times L$ matrix representations of the compressions $A_{L,N}$ consistently approximate the corresponding matrix representation of $A_L$.

\begin{lemma}
    \label{lem:op-approx}
    Let $\bm A$ and $\bm A_N$ be the $L\times L$ matrix representations of $\hat{A}_L$ and $\hat{A}_{L,N}$, respectively, constructed as above. Then under \cref{assump:op-approx1,assump:op-approx2}, and the assumptions of \cref{lem:basis}, we have $\lim_{N\to \infty} \bm A_N = \bm A$ in any matrix norm.
\end{lemma}

\begin{proof}
    See Supplementary Information (SI) of ref.~\cite{FreemanEtAl23}.
\end{proof}

For our purposes, the main utility of \cref{lem:op-approx} is that it allows us to consistently approximate Koopman operators on $\tilde H$ by shift operators on $\tilde H_N$ \cite{BerryEtAl15,Giannakis19}. For that, let $\hat U_N : \hat H_N \to \hat H_N$ be the circular left shift operator defined as
\begin{equation}
    \label{eq:shift-op}
    \hat U_N f(x_n) =
    \begin{cases}
        f(x_{n+1}), & 0 \leq n \leq N-2,\\
        f(x_{0}), & n = N -1.
    \end{cases}
\end{equation}
As $\hat U_N$ is unitary, the family $\{\hat U_N \}_{N\geq 1}$ is uniformly bounded, $\lVert \hat U_N\rVert = 1$, and satisfies \cref{assump:op-approx1}. Moreover, letting $\tilde U^t : \tilde C(M) \to \tilde C(M)$ be the time-$t$ Koopman operator on zero-mean continuous functions, $\tilde U^t f = f \circ \Phi^t$, we have $\iota \tilde U^t = U^t \iota$ since $\mu$ is an invariant measure of the flow $\Phi^t$. One also verifies using \cref{assump:a3} that $\lim_{N\to\infty}\lVert (\iota_N \tilde U^{\Delta t} - \hat U_N \iota_N) f\rVert_{\hat H_N}$ for any $f \in \tilde C(M)$, and thus that \cref{assump:op-approx2} holds.

By \cref{lem:op-approx}, we conclude that the $L\times L$ projected shift operator matrices $\bm U_N = [\hat U_{ij}]_{i,j=1}^L$ with $\hat U_{ij} = \langle \phi_{i,N}, \hat U_N \phi_{j,N}\rangle_N$ consistently approximate the projected Koopman operator matrix $\bm U = [U_{ij}]_{i,j=1}^L$ with $U_{ij} = \langle \phi_i, U^{\Delta t} \phi_j\rangle$, i.e., $\lim_{N\to\infty} \bm U_N = \bm U$. Note that the matrix elements $\hat U_{ij}$ can be computed directly by applying~\eqref{eq:shift-op} to the kernel eigenvectors $\phi_{j,N}$; explicitly,
\begin{equation}
    \hat U_{ij} = \frac{1}{N} \sum_{n=0}^{N-2} \phi_{n, i, N} \phi_{n + 1, j,N} + \frac{1}{N} \phi_{N - 1, i, N} \phi_{0, j, N}.
    \label{eq:shift_op_ij}
\end{equation}
In the following two subsections we will use our data-driven approximation of the Koopman operator to derive an approximation of (i) the projection $\Pi_+$ onto the positive-frequency subspace (\cref{sec:filtering}); and (ii) the resolvent $R_z(V)$ of the generator (\cref{sec:resolvent-quad}).

\begin{remark}[circular vs.\ non-circular shift operators]
    In addition to $\hat U_N$ from~\eqref{eq:shift-op}, consistent approximations of the Koopman operator can also be obtained using non-circular shift operators, e.g., $\hat U_N' : \hat H_N \to \hat H_N$, where
    \begin{equation}
        \hat U_N' f(x_n) =
        \begin{cases}
            f(x_{n+1}), & 0 \leq n \leq N-2,\\
            0, & n = N -1.
        \end{cases}
        \label{eq:shift-op-noncirc}
    \end{equation}
    The corresponding $L \times L$ shift operator matrices $\bm U_N' = [\hat U'_{ij}]_{i,j=1}^L$, $\hat U_{ij}' = \langle \phi_{i,N}, \hat U_N' \phi_{j,N}\rangle_N$, satisfy $\lim_{N\to\infty} \bm U_N' = \bm U$; however, the $\hat U_N'$ are not unitary operators. Unitarity of $\hat U_N$ will turn out to be useful when we perform data-driven approximation of the resolvent $R_z(V)$ in \cref{sec:resolvent-quad} (see, in particular, \eqref{eq:res-mat}). Therefore, in this work we opt to approximate the Koopman operator using the circular shift operator from~\eqref{eq:shift-op}.
\end{remark}

\subsection{Positive-frequency filtering}
\label{sec:filtering}

Let $ \{ A, \tilde A, \hat A_1, \hat A_2, \ldots\}$ with $A: \tilde H \to \tilde H$, $\tilde A: \tilde C(M) \to \tilde C(M)$, and $\hat A_N : \hat H_N \to \hat H_N$ be a family of operators satisfying \cref{assump:op-approx1,assump:op-approx2}. In this subsection, we build a data-driven approximation of the projected operator
\begin{equation}
    A_+ := \Pi_+ A \Pi_+
    \label{eq:a-pos}
\end{equation}
onto the positive-frequency subspace $H_+$. We will arrive at \eqref{eq:a-pos-data-driven}, which employs the discrete Fourier transform (DFT) to filter out negative frequencies.

For a time unit $S>0$ and natural numbers $m$ and $n$ define the frequency and time domains
\begin{equation}
    \Omega_m = \left[-\frac{2\pi m}{S}, \frac{2 \pi m}{S}\right], \quad T_{m,n} = \left[-\frac{n^2S}{m(2n+1)}, \frac{n^2S}{m(2n+1)}\right],
    \label{eq:freq-time}
\end{equation}
respectively. Fix also a family  $\{h_m : \mathbb R \to \mathbb R_+\}_{m \in \mathbb N}$ of smooth, compactly supported functions with $\supp h_m \subseteq [0, 2\pi m/S]$ that converge pointwise to the characteristic function $\chi_{[0,\infty)}$, i.e., $\lim_{m\to\infty}h_m(\omega) = \chi_{[0,\infty)}(\omega)$ for every $\omega\in \mathbb R$. With this family, define the operators $P_{m,n} : H \to H$ as
\begin{equation}
    P_{m,n} f = \int_{\mathbb R^2} h_m(\omega) \chi_{T_{m,n}}(t) e^{-i\omega t} U^t f \, d(\omega, t).
    \label{eq:Pmn}
\end{equation}
We then have the following approximation lemma.

\begin{lemma}
    \label{lem:pos-freq-approx}
    The operators $P_{m,n}$ converge strongly to $\Pi_+$ in the sense of the iterated limit
    \begin{displaymath}
        \Pi_+ f = \lim_{m\to\infty}\lim_{n\to\infty} P_{m,n} f, \quad \forall f \in H.
    \end{displaymath}
\end{lemma}

\begin{proof}

Define $\xi_{m,n} \in C_0(\mathbb R)$ as
\begin{displaymath}
    \xi_{m,n}(\omega) = \frac{1}{2\pi} \int_{-nS/m}^{nS/m} e^{-i\omega t} \, dt = \frac{1}{2\pi} \int_{\mathbb R} e^{-i\omega t} \chi_{T_{m,n}}(t)\, dt.
\end{displaymath}
Since $h_m$ is smooth and compactly supported, for every $\omega' \in \mathbb R$ we have $h_m(\omega') = \lim_{n\to\infty} h_{m,n}(\omega')$, where $h_{m,n} \in C_0(\mathbb R)$ is defined by
\begin{displaymath}
    h_{m,n}(\omega') = \int_{\mathbb R}  \xi_{m,n}(\omega-\omega') h_m(\omega)\, d\omega = \int_{\mathbb R^2} h_m(\omega)  \chi_{T_{m,n}}(t) e^{-i\omega t} e^{i\omega' t}\, d(\omega, t).
\end{displaymath}
Since $h_m$ converges to $\chi_{[0,\infty)}$ pointwise, for every $f \in H$ we have
\begin{displaymath}
    \Pi_+ f = \chi_{i[0,\infty)}(V) f = \chi_{[0,\infty]}(V/i) f = \lim_{m\to\infty} h_m(V / i) f = \lim_{m\to\infty} \lim_{n\to\infty} h_{m,n}(V/i) f,
\end{displaymath}
where
\begin{displaymath}
    h_{m,n}(V/i) f = \int_{\mathbb R^2} h_m(\omega)  \chi_{T_{m,n}}(t) e^{-i\omega t} e^{V t} f \, d(\omega, t) = \int_{\mathbb R^2} h_m(\omega)  \chi_{T_{m,n}}(t) e^{-i\omega t} U^t f \, d(\omega, t).
\end{displaymath}
The claim of the lemma follows.
\end{proof}

\begin{remark}
    The particular choice of frequency--time domain in~\eqref{eq:freq-time} is made in order to make contact between our approach and the DFT. The convergence results in \cref{lem:pos-freq-approx}, as well as \cref{prop:pos-freq-approx} stated below, remain valid for other domain choices so long as $\Omega_m$ increases to $\mathbb R$ as $m \to \infty$ and $T_{m,n}$ increases to $\mathbb R$ as $n \to\infty$ at fixed $m$.
\end{remark}

\subsubsection{Numerical quadrature}

We approximate $P_{m,n}$ by a family of operators $ \{ P_{m,n,K} \}_{K \in \mathbb N}$ obtained by replacing the integral over the frequency--time domain $\Omega_m \times T_{m,n}\subseteq \mathbb R^2$ by quadrature. Defining $\tilde N = (2n+1) K^2/ n^2$ and the gridpoints $ \{(\omega_k, t_j)\}_{k,j=-K}^K$ where

\begin{equation*}
    \omega_k = \frac{\len(\Omega_m) k}{2K} = \frac{2\pi mk}{S K}, \quad t_j = \frac{\len(T_{m,n}) j}{2K} = \frac{n^2 S j}{m(2n + 1)K}, \quad \omega_k t_j = \frac{2\pi k j }{\tilde N},
\end{equation*}
we set
\begin{equation}
    P_{m,n,K} := \frac{1}{\tilde N} \sum_{k,j=-K}^K h_m(\omega_k) \chi_{T_{m,n}}(t_j) e^{-2\pi i k j / \tilde N} U^{t_j}.
    \label{eq:PmnK}
\end{equation}

By continuity of the cross-correlation function $C_{fg}(t) = \langle g, U^t f\rangle$ for any $f, g \in H$, it follows that $\lim_{K\to\infty}\langle g, P_{m,n,K} f \rangle = \langle g, P_{m,n} f\rangle$, and thus that $P_{m,n,K}$ converges to $P_{m,n}$ strongly. Combining this fact with \cref{lem:pos-freq-approx}, we can express the projected operator $A_+$ from~\eqref{eq:a-pos} as the iterated strong limit
\begin{displaymath}
    A_+ f = \lim_{m\to\infty}\lim_{n\to\infty}\lim_{K\to\infty} A_{m,n,K} f, \quad \forall f \in H,
\end{displaymath}
where $A_{m,n,K} = P_{m,n,K}^* A P_{m,n,K}$.

Consider now the operators $\tilde P_{m,n,K} : \tilde C(M) \to \tilde C(M)$ $\tilde P_{m,n,K}' : \tilde C(M) \to \tilde C(M)$ and $\hat P_{m,n,K,N} : \hat H_N \to \hat H_N$ defined as
\begin{align}
    \nonumber\tilde P_{m,n,K} &= \frac{1}{\tilde N} \sum_{k,j=-K}^K h_m(\omega_k) e^{-2\pi i k j / \tilde N} \tilde U^{t_j}, \\
    \nonumber\tilde P'_{m,n,K} &= \frac{1}{\tilde N} \sum_{k,j=-K}^K h_m(\omega_k) \chi_{T_{m,n}}(t_j) e^{2\pi i k j / \tilde N} \tilde U^{-t_j},\\
    \label{eq:PmnKN}\hat P_{m,n,K,N} &= \frac{1}{\tilde N} \sum_{k,j=-K}^K h_m(\omega_k) e^{-2\pi i k j / \tilde N} \hat U^j_N,
\end{align}
where in the definition of $\hat P_{m,n,K,N}$ we used the shift operator $\hat U_N$ based on the sampling interval $\Delta t = nS/(mK)$. One readily verifies that for each $m,n,K \in \mathbb N$, the family of operators $ \{ P_{m,n,K}, \tilde P_{m,n,K}, \hat P_{m,n,K,1}, \hat P_{m,n,K,2}, \ldots\}$ satisfies \cref{assump:op-approx1,assump:op-approx2}. Similarly, these assumptions are satisfied by the family $ \{ P^*_{m,n,K}, \tilde P'_{m,n,K}, \hat P^*_{m,n,K,1}, \hat P^*_{m,n,K,2}, \ldots\}$.
We thus conclude that $ \{ A^+_{m,n,K}, \tilde A^+_{m,n,K}, \hat A^+_{m,n,K,1}, \hat A^+_{m,n,K,2},\ldots \}$ with
\begin{gather*}
    A^+_{m,n,K} = P^*_{m,n,K} A \tilde P_{m,n,K}, \quad \tilde A^+_{m,n,K} = \tilde P_{m,n,K}^* \tilde A \tilde P_{m,n,K}, \\
    \hat A^+_{m,n,K,N} =  \hat P^*_{m,n,K,N} \hat A_N \hat P_{m,n,K,N}
\end{gather*}
also satisfies \cref{assump:op-approx1,assump:op-approx2}.

Define now the functions $ \{ \phi^+_j \in H_+ \}_{j\in \mathbb N}$ and $\{\phi^+_{j,m,n,K,N} \in \hat H_N\}_{j=1}^{N-1}$ as $\phi^+_j = \Pi_+ \phi_j$ and $\phi^+_{j,m,n,K,N} = \hat P_{m,n,K,N} \phi_{j,N}$. The following is a corollary of \cref{lem:op-approx} that summarizes our approximation of positive-frequency projected operators.

\begin{proposition}
    Given $m,n,K,L,N \in \mathbb N$ and with notation as above, let $\bm A^+ = [A^+_{ij}]_{i,j=1}^L$ and $\bm A_{m,n,K,N}^+ = [\hat A^+_{ij}]_{i,j=1}^L$ be the $L \times L$ matrix representations of $A_L^+ := \Pi_L A^+ \Pi_L$ and $ A^+_{m,n,K,L,N} := \Pi_{L,N} \hat A_{m,n,K,N} \Pi_{L,N}$, i.e.,
    \begin{align*}
        A^+_{ij} &= \langle \phi_i, A^+ \phi_j\rangle = \langle \phi^+_i, A \phi^+_j\rangle,\\
        \hat A^+_{ij} &= \langle \phi_{i,N}, \hat A_{m,n,K,N}^+ \phi_{j,N}\rangle_N = \langle \phi^+_{i,m,n,K,N}, \hat A_{m,n,K,N} \phi^+_{j,m,n,K,N}\rangle_N.
    \end{align*}
    Then, $\bm A_{m,n,K,N}^+$ converges to $\bm A^+$ in the iterated limit
    \begin{displaymath}
        \lim_{m\to\infty}\lim_{n\to\infty}\lim_{K\to\infty}\lim_{N\to\infty} \bm A_{m,n,K,N}^+ = \bm A^+.
    \end{displaymath}
    \label{prop:pos-freq-approx}
\end{proposition}

\begin{remark}
    Aside from~\eqref{eq:PmnKN}, other schemes could be employed to build the approximate projections $\hat{P}_{m,n,K,N}$, e.g., high-order quadrature, or convolution-based methods \cite{ColbrookEtAl21}. Such methods may be able to attain a sufficiently fast rate of convergence with respect to $K$ that allows combining the iterated limits of $n\to\infty$ after $K\to\infty$ to a single $K\to\infty$ limit over an increasing integration domain $T_{m,n_K}$. At the same time, however, the use of high-order quadrature imposes regularity conditions on the cross-correlation function $C_{fg}$ that may not hold for arbitrary $f, g \in H$ (effectively restricting the class of basis functions for which \cref{prop:pos-freq-approx} holds). For instance, if $g$ does not lie in the domain of the generator $V$ then $C_{fg}$ is not continuously differentiable, which creates obstacles to approximation of $\langle f, P_{m,n} g \rangle$ via high-order quadrature.
\end{remark}

When convenient, we will use the abbreviated notations $\phi_{j,N}^+ \equiv \phi_{j,m,n,K,N}^+$ for the positive-frequency projected basis functions and $\bm A_N^+ \equiv \bm A_{m,n,K,N}^+$ for the corresponding projected operator matrices. Convergence of $\bm A^+_N$ to $\bm A^+$ will be understood in the sense of \cref{prop:pos-freq-approx}.

\subsubsection{Discrete Fourier transform}

In the experiments of \cref{sec:examples}, we assume that we are given $N$ samples at a fixed sampling interval $\Delta t$ (see \cref{sec:training}), and we take $N$ to be odd for simplicity. We then use the special case of $P_{m,n,K,N}$ with $S = 2\,\Delta t$, $m=1$, and $n=K=(N-1)/2$. With this choice, the frequency domain $\Omega_1 = [-\pi/ \Delta t, \pi/\Delta t]$ is consistent with the Nyquist--Shannon theorem for the sampling interval $\Delta t$. Moreover, we have $\tilde N = N$ and the operator $ P_{1,K,K,N}$ becomes
\begin{displaymath}
    P_{1,K,K,N} = \frac{1}{N} \sum_{k,j=-K}^K h_1(\omega_k) e^{-2\pi i k j / N} \hat U_N^j.
\end{displaymath}
Choosing, further, the window function $h_1$ such that $h_1(\omega_j) = 1$ for all $j = 1, \ldots, K -1 $, our final approximation for the positive-frequency projection is given by $\hat \Pi_{+,N} : \hat H_N \to \hat H_N$, where
\begin{equation}
    \hat \Pi_{+,N} = \frac{1}{N} \sum_{k=1}^K \sum_{j=-K}^K e^{-2\pi i k j / N} \hat U_N^j.
   \label{eq:Pi_DFT}
\end{equation}
This leads to the projected operator
\begin{equation}
    \hat A_{+,N} := \hat \Pi_{+,N}^* \hat A_N \hat \Pi_{+,N},
    \label{eq:a-pos-data-driven}
\end{equation}
which provides (through its matrix representation), a data-driven approximation of $A_+$ from~\eqref{eq:a-pos}. Since, for any $f \in \hat H_N$,
\begin{equation}
    \hat \Pi_{+,N} f(x_l) = \frac{1}{N} \sum_{k=1}^K \sum_{j=-K}^K e^{-2\pi i k j / N} f(x_{j+l}),
    \label{eq:Pi_DFT_time_series}
\end{equation}
$\hat \Pi_{+,N} f$ can be directly computed by applying a fast Fourier transform (FFT) to the $N$-dimensional vector representing $f$, zeroing-out the non-positive frequencies, and then applying an inverse FFT.

Aside from~\eqref{eq:Pi_DFT_time_series}, DFT-based computation of $P_{m,n,K,N} f$ is possible for other values of $n, K > (N-1)/ 2$, using zero-padding of the vector representing $f$ as appropriate. Working with frequency window functions other than $h_1(\omega_j) =1$ used in~\eqref{eq:Pi_DFT} would also be amenable to usage of DFT. Here, we do not consider such possibilities further since we found that~\eqref{eq:Pi_DFT_time_series} yields satisfactory results, but doing so would be an interesting direction for future work.



\subsection{Approximation of the resolvent}
\label{sec:resolvent-quad}

Recall that the finite-rank approximation scheme discussed in \cref{sec:finite-rank_generator} employs finite-rank approximations $\tilde R_{z,L}^+$ and $\tilde R_{z,L}^-$ of the positive- and negative-frequency projected resolvents $R_z^+$ and $R_z^-$, respectively. In order to build data-driven analogs of these approximations, we take advantage of the integral representation of the resolvent \cref{eq:resolventintegral}, reproduced here for convenience:
\begin{equation*}
    R_z(V) = \int_0^\infty e^{-zt} U^t \, dt, \quad \re z > 0.
\end{equation*}
As with our approximation of the projection $\Pi_+$ in \cref{sec:filtering}, given a time unit $S'>0$, we introduce an increasing sequence $T_1,T_2,\ldots$ of time domains, $T_\ell = [0, \ell S']$, and approximate $R_z(V) \in B(\tilde H)$ by the operators $R_{z,\ell} \in B(\tilde H)$, where
\begin{equation}
    R_{z,\ell} = \int_0^{\ell S'} e^{-zt} U^t \, dt.
    \label{eq:Rzl}
\end{equation}
Due to the exponential decay of the $e^{-zt}$ term in the integrand, as $\ell\to\infty$, $R_{z,\ell}$ converges to $R_z(V)$ in operator norm.

We further approximate $R_{z,\ell}$ by a sequence of operators $\{ R_{z,\ell,1}, R_{z,\ell,2}, \ldots \} $, where $R_{z,\ell,Q} : H \to H $ is obtained from a quadrature rule with nodes $t_0, \ldots, t_Q \in T_\ell$ and corresponding weights $w_0, \ldots, w_Q \in \mathbb R_+$, viz.
\begin{equation}
    R_{z,\ell,Q} = \sum_{q=0}^Q w_q e^{-z t_q} U^{t_q}.
    \label{eq:RlQ}
\end{equation}
In the experiments of \cref{sec:examples}, we use the composite Simpson rule where the nodes are equispaced, $t_q = (q - 1) \, \Delta t$ with $\Delta t = \ell S' / Q$, and the weights have the values $w_0 = w_Q = \Delta t/3 $, $w_{2q+1} = 2\,\Delta t/3$, and $w_{2q} = 4\,\Delta t/3$. The resulting approximations from~\eqref{eq:RlQ} converge in the strong operator topology, i.e., $\lim_{Q\to\infty} R_{z,\ell,Q} f = R_{z,\ell} f$ for every $f \in \tilde H$. Explicit rates of convergence can be obtained under additional assumptions on $f$. For instance, if $f$ lies in a finite-dimensional Koopman-invariant subspace of $\tilde H$, then $\mathbb R \ni t \mapsto U^t f$ is a $C^\infty$ function in the norm of $\tilde H$ and the error $\lVert (R_{z,\ell,Q} - R_{z,\ell})f\rVert_H$ is $O(Q^{-3})$. Of course, we may have to contend with slower rates of convergence as such assumptions may not hold in practice. A natural alternative to a composite Simpson scheme given the exponentially decaying integrand in~\eqref{eq:Rzl} is Gauss--Laguerre quadrature.

Using $R_{z,\ell,Q}$, we obtain a strongly convergent approximation of the positive-frequency projected resolvent as
\begin{equation}
    R_z^+(V) f = \lim_{\ell \to \infty} \lim_{Q \to \infty} R_{z,\ell,Q}^+ f, \quad \forall f \in \tilde H,
    \label{eq:res-conv}
\end{equation}
where $R_{z,\ell,Q}^+ = \Pi_+ R_{z,\ell,Q}\Pi_+$.

Next, passing to the data-driven setting, we approximate $R_{z,\ell,Q}$ by $\hat R_{z,\ell,Q,N} : \hat H_N \to \hat H_N$,
\begin{equation}
    \hat R_{z,\ell,Q,N} = \sum_{q=0}^Q w_q e^{-z q \, \Delta t} \hat U^q_N,
    \label{eq:RlQN}
\end{equation}
where $\hat U_N$ is the shift operator on $\hat H_N$ from \eqref{eq:shift-op}. In this last approximation we have tacitly identified the interval $\Delta t$ used to build $R_{z,\ell,Q}$ in~\eqref{eq:RlQ} with the sampling interval of the data. It is also natural to set $S' = \Delta t$ so that $\ell$ and $Q$ represent the number of timesteps within the integration domain $T_\ell$.

Since $R_{z,\ell,Q}$ and $\hat R_{z,\ell,Q,N}$ are given by finite linear combinations of Koopman and shift operators, respectively (i.e., they satisfy \cref{assump:op-approx1,assump:op-approx2}), the projected operators $R_{z,\ell,Q}^+$ can be approximated by positive-frequency projections of $\hat R_{z,\ell,Q,N}$ using \cref{prop:pos-freq-approx} in conjunction with~\eqref{eq:res-conv}. Specifically, given $L \leq N - 1$ and for each $q \in \{0, \ldots, Q\}$, we can compute the $L \times L$ positive-frequency projected shift operator matrix $\bm U^{+,q}_N = [\hat U^{+,q}_{ij}]_{i,j=1}^{L-1}$ with elements $\hat U_{ij}^{+,q} = \langle \phi_{i,N}^+, \hat U_N^q \phi_{j,N}^+\rangle_N$ (cf.~\eqref{eq:shift_op_ij}), i.e.,
\begin{equation*}
    \hat{U}^{+,q}_{ij} = \frac{1}{N}\sum_{n = 0}^{N - 1 - q} \phi_{n, i, N}^{+} \phi_{n + q, j,N}^{+} + \frac{1}{N}\sum_{n=N-q}^{N-1}\phi_{n, i, N}^{+} \phi_{n-N-q, j, N}^{+},
\end{equation*}
and the matrix
\begin{equation}
    \bm R_{z,N}^+ = \sum_{q=0}^Q w_q e^{-zq\,\Delta t} \bm U_N^{+,q}.
    \label{eq:res-mat}
\end{equation}
The matrix $\bm R^+_{z,N}$ represents the projected operator
\begin{equation*}
    \hat  R^+_{z,\ell,Q,L,N} = \Pi_{L,N} \hat  R^+_{z,\ell,Q,N} \Pi_{L,N}, \quad \hat R^+_{z,\ell,Q,N} = \hat\Pi_{+,N} \hat R_{z,\ell,Q,N} \hat\Pi_{+,N},
\end{equation*}
in the $ \{ \phi_{j,N} \}$ basis of $\hat H_{L,N}$. Its computation completes step~3 of \cref{alg:numerical}. By \cref{prop:pos-freq-approx} and~\eqref{eq:res-conv}, $\bm R_{z,N}^+$ converges to the $L\times L$ matrix representation $\bm R^+_z = [R^+_{ij}]_{i,j=1}^L$ of $\tilde R^+_{z,L} = \Pi_L R^+_z(V) \Pi_L$, with $R^+_{ij} = \langle \phi_i^+, R_z(V) \phi_j^+\rangle$ in the iterated limit $\lim_{\ell\to\infty}\lim_{Q\to\infty}\lim_{m\to\infty}\lim_{n\to\infty}\lim_{K\to\infty}\lim_{N\to\infty}$. (Recall that we have suppressed the dependence of the approximate positive-frequency projection $\hat \Pi_{+,N}$ on the parameters $m$, $n$, and $K$ appearing in the iterated limit.)

In the ensuing subsections, we will use the abbreviated notation $R_{z,L,N}^+ \equiv \hat  R_{z,L,\ell,Q,N}^+$ and interpret convergence of $R_{z,L,N}^+$ to $R^+_{z,L}$, as well as convergence of related data-driven operator approximations, in that limit. We also note that a variant of~\eqref{eq:res-mat} is
\begin{equation}
    \tilde{\bm R}_{z,N}^+ = \sum_{q=0}^Q w_q e^{-zq\,\Delta t} (\bm U_N^+)^q, \quad \bm U^+_N \equiv \bm U^{+,1}_N,
    \label{eq:res-mat2}
\end{equation}
which applies iteratively the 1-step shift matrix $\bm U^+_N$ rather than using the multi-step matrices $\bm U_N^{+,q}$. This approximation also consistently recovers $\bm R_z^+$ and has the computational benefit of having to form (and store) only a single shift operator matrix (as opposed to $Q$ such matrices used in~\eqref{eq:res-mat}). For this reason, in the experiments of \cref{sec:examples} we implement \cref{alg:numerical} using $\tilde{\bm R}_z^+$. We alert the reader to the fact that the numerical stability of \eqref{eq:res-mat2} depends on the fact that $\hat U_N$ is a unitary operator. In numerical implementations using on non-circular shift operators (e.g., $\hat U_N'$ from~\eqref{eq:shift-op-noncirc}), we found that thresholding the singular values of the operator matrix to 1 is important to ensure stable behavior of the iterative approximation~\eqref{eq:res-mat2} as $Q$ increases.

\subsection{Compactification}
\label{sec:num-compactification}

Having computed the matrix representation $\bm{R}_{z, N}^{+} \in \mathbb C^{L\times L}$, steps~4--6 of \cref{alg:numerical} approximate the compact operator $S_{z,\tau}^+$ and its negative-frequency counterpart, $S_{z,\tau}^-$. Leveraging the results of \cref{sec:finite-rank,sec:op-approx}, we build these approximation using data-driven approximations of the finite-rank operators $S_{z,\tau,L}^+$ and $S_{z,\tau,L}^-$, respectively, defined in~\eqref{eq:sztl}. As one can verify using~\cref{lem:conj} in conjunction with the fact that the basis functions $\phi_j$ are real, the matrix elements of $S_{z,\tau,L}^+$ and $S_{z,\tau, L}^-$ are complex conjugates of one another, i.e., $\langle \phi_i, S_{z,\tau,L}^-\phi_j\rangle = \langle \phi_i, S_{z,\tau,L}^+ \phi_j\rangle^*$. In light of that, it suffices to consider approximations of $S_{z,\tau,L}^+$ and obtain from them approximations of $S_{z,\tau,L}^-$ by complex conjugation.

First, we approximate $\tilde S_{z,L}^+ = |R_{z,L}^+|$ from~\eqref{eq:sztl} by $\tilde S_{z,L,N}^+ := \lvert R_{z,L,N}^+ \rvert$. The latter operator is represented in the $ \{ \phi_{j,N} \}$ basis of $\hat H_{L,N}$ by the $L\times L$ matrix $\bm S_{z,N}^+ = \lvert \bm R_{z,N}^+ \rvert$, which we compute by a polar decomposition of $\bm R^+_{z,N}$; that is, $\bm R_{z,N}^+ = \bm W_{z,N} \lvert \bm R_{z,N}^+ \rvert$ where $\bm W_{z,N}$ is a unitary matrix. We then approximate $\sqrt{\tilde S_{z,L}^+}$ by $\sqrt{\tilde S_{z,L,N}^+}$ which is represented by the matrix square root  $\sqrt{\bm S_{z,N}^+}$. Since the modulus and square root functions are continuous on $\mathbb C$ and $[0,\infty)$, respectively, it follows from the convergence of $R_{z,N}^+$ to $R_z^+$ (see \cref{lem:pos-freq-approx}) that $\sqrt{\bm S_{z,N}^+}$ converges to the matrix representation $\sqrt{ \bm S_z^+ } = [\langle \phi_i, \sqrt{S_z^+}\phi_j\rangle]_{i,j=1}^L$ of $\sqrt{\tilde S_{z,L}^+}$.

The pre/post multiplication of $\sqrt{\tilde S_{z,L}^+}$ by $G_\tau$ to effect compactification (see~\eqref{eq:szt}) is approximated by pre/post multiplication of $\sqrt{\tilde S_{z,L,N}^+}$ by $G_{\tau, N} $, giving
\begin{equation*}
    S_{z,\tau,L,N}^{+} := \sqrt{\tilde S_{z,L,N}^+} G_{\tau,N} \sqrt{\tilde S_{z,L,N}^+}
\end{equation*}
as an operator on $\hat H_{L,N}$. The $L\times L$ matrix representation of this operator in the $ \{ \phi_{j,N} \}$ basis of $\hat H_{L,N}$ is given by
\begin{displaymath}
    \bm S_{z,\tau,N}^+ = \sqrt{\bm S_{z,N}^+} \bm \Lambda_{\tau,N} \sqrt{\bm S_{z,N}^+},
\end{displaymath}
where $\bm \Lambda_{\tau,N} := \diag(\lambda_{1,N}^{\tau}, \ldots, \lambda_{L,N}^{\tau})$ and $\lambda^\tau_{j,N}$ are eigenvalues of $G_{\tau,N}$ (see \ref{app:markov_semigroup}). By \cref{lem:basis} and the convergence of $\sqrt{\bm S_{z,N}^+}$ to $\sqrt{\bm S_z^+}$, it follows that $\bm S_{z,\tau,N}^+$ converges to the matrix representation $\bm S_{z,\tau}^+ = [\langle \phi_i, S_{z,\tau}^+ \phi_j\rangle]_{i,j=1}^L$ of $S_{z,\tau,L}^+$.

Defining $\bm S_{z,\tau,N}^- = \overline{\bm S_{z,\tau,N}^+}$, where $\overline{(\cdot)}$ denotes elementwise complex conjugation, we have that $\bm S_{z,\tau,N}^-$ is the matrix representation of $S_{z,\tau,L,N}^{-} := \sqrt{\tilde S_{z,L,N}^-} G_{\tau,N} \sqrt{\tilde S_{z,L,N}^-}$, where $\tilde S_{z,L,N}^- := \lvert R_{z,L,N}^- \rvert$. The matrices $\bm S_{z,\tau,N}^-$ converge to the matrix representation $\bm S_{z,\tau}^- = [\langle \phi_i, S_{z,\tau}^- \phi_j\rangle]_{i,j=1}^L$ of $S_{z,\tau,L}^-$.

\subsection{Eigendecomposition}
\label{sec:num-eig}

The rest of \cref{alg:numerical} (steps 7--11) involves (i) approximating the reduced-rank operators $S_{z,\tau,L}^{+,(M)}$ and $S_{z,\tau,L}^{-,(M)}$ from the eigendecompositions of $S_{z,\tau,L,N}^+$ and $S_{z,\tau,L,N}^+$, respectively; (ii) forming the corresponding approximation of $S_{z,\tau,L}^{(M)}$ in~\eqref{eq:sztlm}; and (iii) computing another eigendecomposition for that operator to yield our approximations of the resolvent $R_{z,\tau,L}^{(M)}$, the finite-rank generator $V_{z,\tau,L}^{(M)}$, and their associated eigenfrequencies and eigenfunctions.

We begin by computing an eigendecomposition of $\bm{S}_{z, \tau, N}^+$,
\begin{displaymath}
    \bm{S}_{z,\tau, N}^+ \bm c_{l,\tau,N}  = e_{l,\tau,L,N} \bm c_{l,\tau,N},
\end{displaymath}
where the eigenvalues $e_{l,\tau,L,N}$ are ordered in decreasing order, $z^{-1} > e_{1,\tau,L,N} \geq e_{2,\tau,L,N} \geq \cdots \geq e_{L,\tau,L,N} \geq 0$, and the corresponding eigenvectors $\bm c_{l,\tau,N} \in \mathbb C^L$ are chosen to be orthonormal with respect to the standard Euclidean inner product, $\bm c_{k,\tau,N} \cdot \bm c_{l,\tau,N} = \delta_{kl}$. The eigenvectors $\bm c_{l,\tau,N} = (c_{1l}, \ldots, c_{Ll})^\top$ give the expansion coefficients of eigenfunctions $\xi_{l,\tau,L,N}$ of $S^+_{z,\tau,L,N}$ in the $\{ \phi_{j,N} \}$ basis of $\hat H_{L,N}$; that is,
\begin{displaymath}
    S_{z,\tau,L,N}^+ \xi_{l,\tau,L,N} = e_{l,\tau,L,N} \xi_{l,\tau,L,N},
\end{displaymath}
where $\xi_{l,\tau,L,N} = \sum_{j=1}^L c_{jl} \phi_{j,N}$ and $\langle \xi_{k,\tau,L,N}, \xi_{l,\tau,L,N}\rangle_N = \delta_{kl}$. By convergence of $\bm S_{z,\tau,N}^+$ to $\bm S_{z,\tau}^+$, the eigenvalues $e_{l,\tau,L,N}$ of $S^+_{z,\tau,L,N}$ converge to the eigenvalues $e_{l,\tau,L}$ of $S^+_{z,\tau,L}$ in the sense of \cref{thm:spec_compact}(i), and for every eigenfunction $\xi_{l,\tau,L} \in H_L$ of $S^+_{z,\tau,L}$ there exists a sequence of eigenfunctions $\xi_{l,\tau,L,N} \in \hat H_{L,N}$ converging to it in the sense of uniform convergence of their corresponding representatives in $C(M)$.

Next, let $\bm C_{\tau,N}^{(M)} = (\bm c_{1,\tau,N}, \ldots, \bm c_{L,\tau,N})$ be the $L\times M$ matrix whose columns are the leading $M$ eigenvectors $\bm c_{l,\tau,N}$, and $\bm E_{\tau,N}^{(M)} = \diag(e_{1,\tau,L,N}, \ldots, e_{M,\tau,L,N})$ the diagonal matrix whose diagonal entries are the corresponding eigenvalues. Note that $\bm \Xi^{(M)}_{\tau,N} = \bm C_{\tau,N}^{(M)} \bm C_{\tau,N}^{(M)*}$ is the matrix representation of the orthogonal projection $\Xi_{\tau,L,N}^{(M)} : \hat H_{L,N} \to \hat H_{L,N}$ onto the $M$-dimensional subspace of $\hat H_{L,N}$ spanned by $\xi_{1,\tau,L,N}, \ldots, \xi_{M,\tau,L,N}$. Define the matrix $\bm S_{z,\tau,N}^{+,(M)} = \bm C_{\tau,L,N}^{(M)} \bm E_{\tau,N}^{(M)} \bm C_{\tau,N}^{(M)*}$ that represents the reduced-rank operator $S_{z,\tau,L,N}^{+,(M)} := \Xi_{\tau,L,N}^{(M)} S_{z,\tau,L,N}^+ \Xi_{\tau,L,N}^{(M)}$. The elementwise complex conjugate $\overline{\bm S_{z,\tau,N}^{+,(M)}} =: \bm S_{z,\tau,N}^{-,(M)}$ is the matrix representation of the negative-frequency reduced-rank operator $S_{z,\tau,L,N}^{-,(M)} := \bar\Xi_{\tau,L,N}^{(M)} S_{z,\tau,L,N}^- \bar \Xi_{\tau,L,N}^{(M)}$. Correspondingly, $\bm S_{z,\tau,N}^{(M)} = \bm S_{z,\tau,N}^{+,(M)} - \bm S_{z,\tau,N}^{-,(M)}$ represents the operator $S_{z,\tau,L,N}^{(M)} = S_{z,\tau,L,N}^{+,(M)} - S_{z,\tau,L,N}^{-,(M)}$.

By the previous results, the matrices $\bm S_{z,\tau,N}^{+,(M)}$, $\bm S_{z,\tau,N}^{-,(M)}$, and $\bm S_{z,\tau,N}^{(M)}$ converge to the matrix representations of the operators $S_{z,\tau,L}^{+,(M)}$, $S_{z,\tau,L}^{-,(M)}$, and $S_{z,\tau,L}^{(M)}$, respectively, from~\eqref{eq:sztlm_pm} and~\eqref{eq:sztlm}. The operator $S_{z,\tau,L}^{(M)}$ has rank at most $2M$. By the convergence of its matrix representation to that of $S_{z,\tau,L}^{(M)}$ and \cref{prop:spec_s}, for sufficiently large $N$ and $L$, $\rank S_{z,\tau,L,N}^{(M)}$ will be equal to $2M$, and $\sigma(S_{z,\tau,L,N}^{(M)})$ will lie in a (closed) proper subinterval of $[-z^{-1}, z^{-1}]$. Henceforth, we will assume that this is the case.

We now compute an eigendecomposition of $\bm S_{z,\tau,N}^{(M)}$,
\begin{displaymath}
    \bm{S}_{z,\tau, N}^{(M)} \bm q_{j,\tau,N}^{(M)}  = e_{j,\tau,L,N}^{(M)} \bm q_{j,\tau,N}^{(M)},
\end{displaymath}
indexing, as in \cref{sec:finite-rank_generator}, the nonzero eigenvalues by $j\in \{ -M, \ldots, -1, 1, \ldots, M \}$ and ordering them as
\begin{displaymath}
    e_{-M,\tau,L,N}^{(M)} \leq \cdots \leq e_{-1,\tau,L,N}^{(M)} < 0 < e_{1,\tau,L,N}^{(M)} \leq \cdots \leq e_{M,\tau,L,N}^{(M)}.
\end{displaymath}
We choose the corresponding eigenvectors $\bm q_{j,\tau,N}^{(M)} = (q_{1j}, \ldots, q_{Lj})^\top$ to be orthonormal on $\mathbb C^L$, $\bm q_{j,\tau,N}^{(M)} \cdot \bm q_{k,\tau,N}^{(M)} = \delta_{jk}$, and define associated eigenfunctions $\psi_{j,L,N}^{(M)}$ as
\begin{equation}
    \psi_{j,L,N}^{(M)} = \sum_{i=1}^L q_{ij} \phi_{i,N}.
    \label{eq:psi_data_driven}
\end{equation}
We also let $\bm Q_{\tau,N}^{(M)} = (\bm q_{-M,\tau,N}^{(M)}, \ldots, \bm q_{-1,\tau,N}^{(M)}, \bm q_{1,\tau,N}^{(M)}, \ldots, \bm q_{M,\tau,N}^{(M)})$ be the $L \times (2M)$ matrix containing the eigenvectors $\bm q_{j,\tau,N}^{(M)}$ in its columns. The matrix $\bm \Psi_{\tau,N}^{(M)} = \bm Q_{\tau,N}^{(M)} \bm Q_{\tau,N}^{(M)*}$ is then a projection matrix that represents the orthogonal projection $\Psi_{L,N}^{(M)} : \hat H_{L,N} \to \hat H_{L,N}$ onto $\spn \{ \psi_{j,L,N}^{(M)} \}_{j\in \{ -M, \ldots, -1, 1, \ldots, M \}}$ with respect to the $ \{\phi_{j,N}\}$ basis of $\hat H_{L,N}$. We denote the complementary projection operator to $\Psi_{L,N}^{(M)}$ as $\Psi_{L,N}^{(M),\perp} \equiv I - \Psi_{L,N}^{(M)}$. This operator is represented by the projection matrix $\bm \Psi_{\tau,N}^{(M), \perp} := \bm I - \bm \Psi_{\tau,N}^{(M)}$ where $\bm I$ is the $L\times L$ identity matrix. By the convergence of the matrix representation of $S_{z,\tau,L,N}^{(M)}$ to the matrix representation of $S_{z,\tau,L}^{(M)}$, the eigenvalues $e_{j,\tau,L,N}^{(M)}$ of $S_{z,\tau,L,N}^{(M)}$ converge to the eigenvalues $e_{j,\tau,L}^{(M)}$ of $S_{z,\tau,L}^{(M)}$ in the sense of \cref{thm:spec_compact}(i). Moreover, by \cref{lem:basis}, for every corresponding eigenfunctions $\psi_{j,\tau,L}^{(M)}$ of $S_{j,\tau,L}^{(M)}$, there is a family of eigenfunctions $\psi_{j,\tau,L,N}^{(M)}$ that converges to it in the sense of uniform convergence in $C(M)$ of their continuous representatives.

Next, similarly to~\eqref{eq:lambda_approx} and~\eqref{eq:omega_approx}, respectively, we define resolvent eigenvalues $\vartheta_{j,\tau,L,N}^{(M)} \in C_z$ and eigenfrequencies $\omega_{j,\tau,L,N}^{(M)} \in i \mathbb R$ as
\begin{equation}
    \vartheta_{j,\tau,L,N}^{(M)} = \gamma_z^{-1}(e_{j,\tau,L,N}^{(M)}), \quad
    \omega_{j,\tau,L,N}^{(M)} = \frac{1}{i} \beta_z(\vartheta_{j,\tau,L,N}^{(M)}).
    \label{eq:eig_gamma_omega}
\end{equation}
We assemble these eigenvalues and eigenfrequencies into $(2M)\times (2M)$ diagonal matrices
\begin{align*}
    \bm \Theta_{\tau,N}^{(M)} &= \diag(\vartheta_{-M,\tau,L,N}^{(M)}, \ldots, \vartheta_{-1,\tau,L,N}^{(M)}, \vartheta_{1,\tau,L,N}^{(M)}, \ldots, \vartheta_{M,\tau,L,N}^{(M)}),\\
    \bm \Omega_{\tau,N}^{(M)} &= \diag(\omega_{-M,\tau,L,N}^{(M)}, \ldots, \omega_{-1,\tau,L,N}^{(M)}, \omega_{1,\tau,L,N}^{(M)}, \ldots, \omega_{M,\tau,L,N}^{(M)}),
\end{align*}
and define
\begin{equation}
    \bm R_{z,\tau,N}^{(M)} = \bm Q_{\tau,N}^{(M)} \bm \Theta_{\tau,N}^{(M)} \bm Q_{\tau,N}^{(M)*} + z^{-1} \bm \Psi_{\tau,N}^{(M),\perp}, \quad \bm V_{z,\tau,N}^{(M)} = i \bm Q_{\tau,N}^{(M)} \bm\Omega_{\tau,N}^{(M)} \bm Q_{\tau,N}^{(M)*}.
    \label{eq:res-final-approx}
\end{equation}
The matrix $\bm R_{z,\tau,N}^{(M)}$ represents (with respect to the $ \{ \psi_{j,N} \}$ basis of $\hat H_{L,N}$) the resolvent $R_{z,\tau,L,N}^{(M)} : \hat H_{L,N} \to \hat H_{L,N}$ of a skew-adjoint operator $V_{z,\tau,L,N}^{(M)}$ on $\hat H_{L,N}$ of rank $2M$ that is represented by the (skew-adjoint) matrix $\bm V_{z,\tau,N}^{(M)}$.

By construction, the operators $R_{z,\tau,L,N}^{(M)}$ and $V_{z,\tau,L,N}^{(M)}$ have the eigendecompositions
\begin{displaymath}
    R_{z,\tau,L,N}^{(M)} \psi_{j,\tau,L,N}^{(M)} = \vartheta_{j,\tau,L,N}^{(M)} \psi_{j,\tau,L,N}^{(M)}, \quad V_{z,\tau,L,N}^{(M)} \psi_{j,\tau,L,N}^{(M)} = i \omega_{j,\tau,L,N}^{(M)} \psi_{j,\tau,L,N}^{(M)},
\end{displaymath}
respectively. Moreover, by continuity of the functions $\gamma_z^{-1}$ and $\beta_z$, the eigenvalues $\vartheta_{j,\tau,L,N}^{(M)}$ and eigenfrequencies $\omega_{j,\tau,L,N}^{(M)}$ from~\eqref{eq:eig_gamma_omega} converge to $\vartheta_{j,\tau,L}^{(M)}$ and $\omega_{j,\tau,L}^{(M)}$ from~\eqref{eq:lambda_approx} and~\eqref{eq:omega_approx}, respectively.

To summarize, the asymptotic convergence of our scheme is as follows:
\begin{enumerate}
    \item At fixed regularization parameter $\tau$, rank parameter $M$, and approximation space dimension parameter $L$, the $L\times L$ matrix representations of $R_{z,\tau,L,N}^{(M)}$ and $V_{z,\tau,L,N}^{(M)}$ on $\tilde H_{L,N}$ converge to the $L \times L$ matrix representations of $R_{z,\tau,L}^{(M)} = \tilde \Pi_L R_{z,\tau}^{(M)} \tilde \Pi_L$ and $V_{z,\tau,L}^{(M)} = \tilde \Pi_L V_{z,\tau}^{(M)} \tilde \Pi_L$ on $\tilde H_L$, respectively, in the sense of \cref{prop:pos-freq-approx} and~\eqref{eq:res-conv} (i.e., in the iterated limit $\lim_{m\to\infty}\lim_{n\to\infty}\lim_{K\to\infty}\lim_{\ell\to\infty}\lim_{Q\to\infty}\lim_{N\to\infty}$). As a result, the eigenvalues $\gamma_{l,\tau,L,N}^{(M)}$ and $\omega_{l,\tau,L,N}^{(M)}$ of $R_{z,\tau,L,N}^{(M)}$ and $V_{z,\tau,L,N}^{(M)}$ converge to eigenvalues $\gamma_{l,\tau,L}^{(M)}$ and $\omega_{l,\tau,L}^{(M)}$ of $R_{z,\tau,L}^{(M)}$ and $V_{z,\tau,L}^{(M)}$, respectively. The corresponding eigenfunctions $\psi_{l,\tau,L,N}^{(M)}$ converge to eigenfunctions $\psi_{l,\tau,L}^{(M)}$ of $R_{z,\tau,L}^{(M)}$ and $V_{z,\tau,L}^{(M)}$ in the sense of convergence of the corresponding continuous representatives in $C(M)$ (cf.\ \cref{lem:basis}).
    \item At fixed regularization parameter $\tau$ and rank parameter $M$, $R_{z,\tau,L}^{(M)}$ converges as $L\to\infty$ strongly to the projected resolvent $R_{z,\tau}^{(M)}$. As a result, the corresponding skew-adjoint operators $V_{z,\tau,L}^{(M)}$ converge in strong resolvent sense, and thus spectrally in the sense of \cref{thm:spec-conv}, to $V_{z,\tau}^{(M)}$.
    \item At fixed regularization parameter $\tau$, $R_{z,\tau}^{(M)}$ converges as $M\to\infty$ strongly to the compactified resolvent $R_{z,\tau}$. As a result, $V_{z,\tau}^{(M)}$ converges in strong resolvent sense, and thus spectrally in the sense of \cref{thm:spec-conv}, to $V_{z,\tau}$.
    \item As $\tau\to 0^+$ the compactified resolvents $R_{z,\tau}$ converge to $R_z(V)$ strongly. As a result, $V_{z,\tau}$ converges to the Koopman generator $V$ in strong resolvent sense, and thus spectrally in the sense of \cref{thm:spec-conv}.
\end{enumerate}

\subsection{Practical computing considerations}

There are several computational choices one makes when implementing \cref{alg:numerical}, namely the choice of: basis functions $\phi_{j,N}$, the number $L$ of such functions used, the rank parameter $M$, the choice of resolvent parameter $z$, the choice of regularization parameter $\tau$, and the numerical integration time $T_\ell = Q \,\Delta t$ used for resolvent computation.

\paragraph{Choice of basis functions.} The asymptotic convergence of \cref{alg:numerical} requires fairly mild assumptions on the kernel $k$ (i.e., continuity, Markov normalizability, and strict positivity of the associated integral operator), which can be met using many of the popular classes of kernels in the literature \cite{Genton01}. In particular, our approach of using orthonormal basis functions derived by eigendecomposition of integral operators addresses an important problem in data-driven Koopman and transfer operator techniques, namely the choice of a well-conditioned dictionary adapted to the function space on which the Koopman/transfer operator acts \cite{WilliamsEtAl15,KlusEtAl16}.

That being said, at any given $L,M$ and number of samples $N$, the performance of \cref{alg:numerical} invariably depends on the choice of $k$. In broad terms, our approach is subject to usual bias--variance tradeoffs encountered in data-driven techniques---increasing $L$ and/or $M$ reduces the bias associated with the approximation of $R_{z,\tau}$ by the finite-rank operator $R_{z,\tau,L}^{(M)}$, but at the same time, increasing $L$ and/or $M$ at fixed $N$ increases the risk of sampling errors (``variance'') in the approximation of the matrix representation of $R_{z,\tau,L}^{(M)}$ by the matrix representation of $R_{z,\tau,L,N}^{(M)}$. Thus, for efficient and statistically robust spectral computations, it is advantageous that the leading basis vectors $\phi_j$ span subspaces of $H$ which are approximately invariant under the action of the Koopman operator.

In specific cases, the kernel $k$ can be explicitly chosen such that the $\phi_j$ span unions of Koopman eigenspaces (e.g., a translation-invariant kernel for systems on tori with pure point spectra; see \cref{sec:lineartor} and \ref{app:markov_semigroup}). In such cases, accurate approximations of Koopman eigenfunctions can be obtained with small numbers of basis functions. In general, however, the more complex are the dynamics and/or observation modality $F$, the more basis functions are required. A possible empirical approach for increasing the efficiency of the basis functions in representing approximately Koopman-invariant subspaces is to lift the observation map to take values in a higher-dimensional space using delay embeddings. It can be shown that this approach leads to kernel integral operators $G_\tau$ that asymptotically commute with the Koopman operator, and as a result the eigenspaces of $G_\tau$ with sufficiently isolated corresponding eigenvalues become approximately Koopman-invariant \cite{DasGiannakis19,Giannakis21a}. This approach may be particularly useful in high-dimensional applications (e.g., climate dynamics \cite{FroylandEtAl21}).

\paragraph{Balance between resolvent parameter $z$ and total integration time $T_\ell$.} The choice of $z$ and $T_\ell$ will affect the accuracy of the numerical quadrature~\eqref{eq:RlQN}. As $z$ increases, the rapidly decreasing term $e^{-zt}$ will require finely spaced quadrature nodes for accuracy. Eventually, for $t$ greater than a $z$-dependent value $t_\epsilon$, $e^{-zt}$ will decay to zero to machine precision; increasing $T_\ell$ beyond $t_\epsilon$ will not improve convergence of the numerical quadrature, and may in fact be detrimental due to injection of numerical noise to small matrix elements of $R_{z,\ell,Q,N}$ from~\eqref{eq:RlQN}. At the same time, for a given $z$, $T_\ell$ must be sufficiently large so that the truncated resolvent $R_{z,\ell}$ from \eqref{eq:Rzl} is a good approximation to $R_z(V)$ when projected onto the finite-dimensional approximation space $H_L$. We recommend that $z$ is chosen so that $z^{-1}$ is on the order of expected frequencies of the system.

\paragraph{Choice of regularization parameter $\tau$.} Choosing $\tau$ involves a similar tradeoff to the choice of number of basis functions $L$. That is, the smaller $\tau$ is the closer is the approximate generator $V_{z,\tau}$ to the true generator $V$ is in a spectral sense, but, in general, the smaller $\tau$ is accurate approximation of $V_{z,\tau}$ requires larger numbers of basis functions and/or training samples. As one might expect, the sensitivity of the eigenvalues and eigenfunctions of $V_{z,\tau}$ on $\tau$ is higher in systems with non-diagonalizable Koopman operators (e.g., the L63 system studied in \cref{sec:L63}), where the eigenspaces of $V_{z,\tau}$ along convergent sequences of eigenvalues may not have well-defined limits as $\tau$ decreases to 0. On the other hand, in systems with non-trivial Koopman eigenfunctions, like the torus rotation described in \cref{sec:lineartor}, the results of spectral calculations appear to be more insensitive to the choice of $\tau$.

As a practical guideline, $\tau$ can be tuned by minimization of the pseudospectral bound $\varepsilon_{T_c}$ from \cref{eq:eig-epsilon} for the leading eigenfunctions of $V_{z,\tau}$ over candidate $\tau$ values. Fortunately, varying $\tau$ takes relatively few computation resources as the use of $G_{\tau/2,N}$ (\cref{alg:numerical}, step 5) happens after the expensive basis function computation (the first step of \cref{alg:numerical}).

\paragraph{Choice of rank parameter $M$.} In the numerical experiments detailed in \cref{sec:examples}, we found that we could obtain sufficient results when the rank parameter $M$ was chosen to be no larger than $L / 3$, however smaller values of $M$ can improve results, particularly when $L$ is large. For the second example (skew-product flow on the torus; \cref{sec:skewtor}), we found that $M = 266 \approx L/3$ yields less than optimal results, and instead use $M \approx L / 20$. (Using values of $M = L / 6$ and $M = L / 10$ gave very similar results in the sense of eigenfrequency accuracy and the magnitude $\epsilon_{T_c}$.) Ultimately, we recommend making an initial choice of $M = L / 6$, with the possibility of tuning this parameter through minimizing $\epsilon_{T_c}$ and maintaining $M$ large enough such that the eigendecomposition can capture relevant dynamics of the system.

\section{Examples}
\label{sec:examples}

In this section, we apply the technique described in \cref{sec:numimplement} to three ergodic systems of increasing complexity: a linear rotation on the torus, a skew rotation on the torus, and the Lorenz 63 (L63) system \cite{Lorenz63} on $\mathbb R^3$. The linear and skew torus rotations have pure point spectra which we can compute analytically. The L63 system is mixing with an associated continuous spectrum of the Koopman operator.

In each example, we numerically generate state space trajectory data $x_0, \ldots, x_{N-1} \in X$ at a fixed interval $\Delta t$. We embed the state space trajectory in data space $Y = \mathbb R^d$ via an embedding $F: X \to Y$ to produce time-ordered samples $y_0, \ldots, y_{N-1}$ with $y_n = F(x_n)$. Using the samples $y_n$, we compute kernel eigenvectors $\phi_{j,N}$ and associated eigenvalues via the approach described in \cref{sec:basis} and \ref{app:markov}. We then execute \cref{alg:numerical} to compute eigenfunctions $\psi_{j,\tau,L,N}^{(M)}$ of the approximate generator $V_{z,\tau,L,N}^{(M)}$ and their corresponding eigenfrequencies $\omega_{j,\tau,L,N}^{(M)}$. For the remainder of this section, if there is no risk of confusion  we suppress $\tau$, $L$, $M$, and $N$ subscripts/superscripts from our notation for eigenfunctions and eigenfrequencies, i.e., $\phi_{j,N} \equiv \phi_j$, $\psi_{j,\tau,L,N}^{(M)} \equiv \psi_j$, and $\omega_{j,\tau,L,N}^{(M)} \equiv \omega_j$.

Our main experimental objectives are to verify that our approach (i) yields  consistent results with analytical solutions for systems with known spectra (linear and skew torus rotations); and (ii) identifies observables that behave as approximate Koopman eigenfunctions, i.e., have strong cyclicity and slow correlation decay, under mixing dynamics. Given these objectives, we order our computed eigenpairs $(\omega_1, \psi_1), (\omega_2, \psi_2), \ldots$ using the pseudospectral criteria described in \cref{sec:pseudospec}; that is, $\psi_1$ is the non-constant eigenfunction with the smallest pseudospectral bound $\varepsilon_{T_c}(\omega_j, \psi_j)$ from~\eqref{eq:eig-epsilon}, $\psi_2$ has the second-smallest $\varepsilon_{T_c}(\omega_j, \psi_j)$ value, and so on.

Representative numerical eigenfrequencies computed for the three systems under study using datasets of different lengths, along with corresponding values of $\varepsilon_{T_c}$ and error relative to analytically known eigenfrequencies (when available) are listed in \cref{tab:num_results}.

\begin{table}[]
    \centering
    \resizebox{\textwidth}{!}{%
    \begin{tabular}{llllllll}
        \toprule
    \textbf{System} &
      \textbf{Observation Map} &
      \textbf{$N$} &
      \textbf{$T_c$} &
      \textbf{Computed $\omega_j$} &
      \textbf{Known $\omega_j$} &
      \textbf{Relative Error} &
      \textbf{$\varepsilon_{T_c}(\omega_j, \psi_j)$} \\ \hline
     &
       &
      \cellcolor[HTML]{EFEFEF} &
      \cellcolor[HTML]{EFEFEF} &
      \cellcolor[HTML]{EFEFEF}1.0001 &
      \cellcolor[HTML]{EFEFEF}1.0000 &
      \cellcolor[HTML]{EFEFEF}0.01\% &
      \cellcolor[HTML]{EFEFEF}0.01175 \\
     &
       &
      \cellcolor[HTML]{EFEFEF} &
      \cellcolor[HTML]{EFEFEF} &
      \cellcolor[HTML]{EFEFEF}5.4784 &
      \cellcolor[HTML]{EFEFEF}5.4772 &
      \cellcolor[HTML]{EFEFEF}0.02\% &
      \cellcolor[HTML]{EFEFEF}0.01176 \\
     &
       &
      \multirow{-3}{*}{\cellcolor[HTML]{EFEFEF}40962} &
      \multirow{-3}{*}{\cellcolor[HTML]{EFEFEF}49.10} &
      \cellcolor[HTML]{EFEFEF}6.4787 &
      \cellcolor[HTML]{EFEFEF}6.4772 &
      \cellcolor[HTML]{EFEFEF}0.02\% &
      \cellcolor[HTML]{EFEFEF}0.01133 \\
     &
       &
       &
       &
      1.0001 &
      1.0000 &
      0.01\% &
      0.12167 \\
     &
       &
       &
       &
      5.478 &
      5.4772 &
      0.01\% &
      0.12179 \\
     &
      \multirow{-6}{*}{$\bbR^4$} &
      \multirow{-3}{*}{4098} &
      \multirow{-3}{*}{49.10} &
      6.4774 &
      6.4772 &
      0.00\% &
      0.12160 \\
     &
       &
      \cellcolor[HTML]{EFEFEF} &
      \cellcolor[HTML]{EFEFEF} &
      \cellcolor[HTML]{EFEFEF}1.0002 &
      \cellcolor[HTML]{EFEFEF}1.0000 &
      \cellcolor[HTML]{EFEFEF}0.02\% &
      \cellcolor[HTML]{EFEFEF}0.01203 \\
     &
       &
      \cellcolor[HTML]{EFEFEF} &
      \cellcolor[HTML]{EFEFEF} &
      \cellcolor[HTML]{EFEFEF}5.4811 &
      \cellcolor[HTML]{EFEFEF}5.4772 &
      \cellcolor[HTML]{EFEFEF}0.07\% &
      \cellcolor[HTML]{EFEFEF}0.01176 \\
     &
       &
      \multirow{-3}{*}{\cellcolor[HTML]{EFEFEF}40962} &
      \multirow{-3}{*}{\cellcolor[HTML]{EFEFEF}49.10} &
      \cellcolor[HTML]{EFEFEF}6.4819 &
      \cellcolor[HTML]{EFEFEF}6.4772 &
      \cellcolor[HTML]{EFEFEF}0.07\% &
      \cellcolor[HTML]{EFEFEF}0.01163 \\
     &
       &
       &
       &
      1.0002 &
      1.0000 &
      0.02\% &
      0.12344 \\
     &
       &
       &
       &
      5.4823 &
      5.4772 &
      0.09\% &
      0.16230 \\
    \multirow{-12}{*}{Linear rotation} &
      \multirow{-6}{*}{$\bbR^3$} &
      \multirow{-3}{*}{4098} &
      \multirow{-3}{*}{49.10} &
      6.481 &
      6.4772 &
      0.06\% &
      0.13050 \\ \hline
     &
       &
      \cellcolor[HTML]{EFEFEF} &
      \cellcolor[HTML]{EFEFEF} &
      \cellcolor[HTML]{EFEFEF}1.0001 &
      \cellcolor[HTML]{EFEFEF}1.0000 &
      \cellcolor[HTML]{EFEFEF}0.01\% &
      \cellcolor[HTML]{EFEFEF}0.00974 \\
     &
       &
      \cellcolor[HTML]{EFEFEF} &
      \cellcolor[HTML]{EFEFEF} &
      \cellcolor[HTML]{EFEFEF}5.4882 &
      \cellcolor[HTML]{EFEFEF}5.4772 &
      \cellcolor[HTML]{EFEFEF}0.20\% &
      \cellcolor[HTML]{EFEFEF}0.01244 \\
     &
       &
      \multirow{-3}{*}{\cellcolor[HTML]{EFEFEF}40962} &
      \multirow{-3}{*}{\cellcolor[HTML]{EFEFEF}19.64} &
      \cellcolor[HTML]{EFEFEF}6.5149 &
      \cellcolor[HTML]{EFEFEF}6.4772 &
      \cellcolor[HTML]{EFEFEF}0.58\% &
      \cellcolor[HTML]{EFEFEF}0.02215 \\
     &
       &
       &
       &
      1.0002 &
      1.0000 &
      0.02\% &
      0.98219 \\
     &
       &
       &
       &
      5.4837 &
      5.4772 &
      0.12\% &
      0.10718 \\
     &
      \multirow{-6}{*}{$\bbR^4$} &
      \multirow{-3}{*}{4098} &
      \multirow{-3}{*}{19.64} &
      6.5075 &
      6.4772 &
      0.47\% &
      0.11348 \\
     &
       &
      \cellcolor[HTML]{EFEFEF} &
      \cellcolor[HTML]{EFEFEF} &
      \cellcolor[HTML]{EFEFEF}1.0019 &
      \cellcolor[HTML]{EFEFEF}1.0000 &
      \cellcolor[HTML]{EFEFEF}0.19\% &
      \cellcolor[HTML]{EFEFEF}0.00951 \\
     &
       &
      \cellcolor[HTML]{EFEFEF} &
      \cellcolor[HTML]{EFEFEF} &
      \cellcolor[HTML]{EFEFEF}5.4899 &
      \cellcolor[HTML]{EFEFEF}5.4772 &
      \cellcolor[HTML]{EFEFEF}0.23\% &
      \cellcolor[HTML]{EFEFEF}0.10952 \\
     &
       &
      \multirow{-3}{*}{\cellcolor[HTML]{EFEFEF}40962} &
      \multirow{-3}{*}{\cellcolor[HTML]{EFEFEF}19.64} &
      \cellcolor[HTML]{EFEFEF}6.6956 &
      \cellcolor[HTML]{EFEFEF}6.4772 &
      \cellcolor[HTML]{EFEFEF}3.37\% &
      \cellcolor[HTML]{EFEFEF}0.30054 \\
     &
       &
       &
       &
      1.0025 &
      1.0000 &
      0.25\% &
      0.09899 \\
     &
       &
       &
       &
      5.4677 &
      5.4772 &
      0.17\% &
      0.20315 \\
    \multirow{-12}{*}{Skew rotation} &
      \multirow{-6}{*}{$\bbR^3$} &
      \multirow{-3}{*}{4098} &
      \multirow{-3}{*}{19.64} &
      6.6596 &
      6.4772 &
      2.82\% &
      0.32353 \\ \hline
     &
       &
      \cellcolor[HTML]{EFEFEF} &
      \cellcolor[HTML]{EFEFEF} &
      \cellcolor[HTML]{EFEFEF}7.4639 &
      \cellcolor[HTML]{EFEFEF}NA &
      \cellcolor[HTML]{EFEFEF}NA &
      \cellcolor[HTML]{EFEFEF}0.15949 \\
     &
       &
      \cellcolor[HTML]{EFEFEF} &
      \cellcolor[HTML]{EFEFEF} &
      \cellcolor[HTML]{EFEFEF}14.2152 &
      \cellcolor[HTML]{EFEFEF}NA &
      \cellcolor[HTML]{EFEFEF}NA &
      \cellcolor[HTML]{EFEFEF}0.35803 \\
    \multirow{-3}{*}{l63} &
      \multirow{-3}{*}{Id} &
      \multirow{-3}{*}{\cellcolor[HTML]{EFEFEF}64000} &
      \multirow{-3}{*}{\cellcolor[HTML]{EFEFEF}2.00} &
      \cellcolor[HTML]{EFEFEF}4.7587 &
      \cellcolor[HTML]{EFEFEF}NA &
      \cellcolor[HTML]{EFEFEF}NA &
      \cellcolor[HTML]{EFEFEF}0.78442 \\
      \bottomrule
    \end{tabular}
    }
    \caption{Representative eigenfrequencies $\omega_j$ for the linear rotation, skew rotation, and L63 system. We print numerically computed eigenfrequencies, analytically known eigenfrequencies (for the torus rotation and skew rotation), and the corresponding relative error. The table also shows the value of the pseudospectral bound $\varepsilon_{T_c}(\omega_j \psi_j)$ from \eqref{eq:eig-epsilon} for the corresponding eigenfunctions $\psi_j$ and the time parameter $T_c$ for which $\varepsilon_{T_c}(\psi_j)$ was computed. The quantity $N$ denotes the number of training data points for each experiment.}
    \label{tab:num_results}
    \end{table}

\subsection{Linear rotation on the torus}
\label{sec:lineartor}

A linear rotation on the torus is the simplest system studied and is defined by the flow $\Phi^t : \mathbb T^2 \to \mathbb T^2$,
\[\Phi^t(\theta_1, \theta_2) = (\theta_1 + \alpha_1 t, \theta_2 + \alpha_2 t) \mod 2\pi, \]
where $\alpha_1, \alpha_2 \in \mathbb R$ are frequency parameters, and $(\theta_1, \theta_2) \in [0, 2\pi)^2$ are standard angle coordinates.
For any $\alpha_1, \alpha_2 \in \mathbb R$, the flow $\Phi^t$ preserves the normalized Haar measure $\mu$ on $\mathbb T^2$, and when $\alpha_1$ and $\alpha_2$ are incommensurate $\mu$ is ergodic (in that case, $\mu$ is the unique Borel invariant probability measure under $\Phi^t$). The eigenfrequency ($\omega_j$) and eigenfunction ($\psi_j$) pairs for the Koopman generator $V$ on $L^2(\mu)$ can be written as
\begin{equation}
    \omega_j = \alpha_1 m_j +  \alpha_2 n_j, \quad \psi_j(\theta_1, \theta_2) = e^{i (m_j\theta_1 + n_j\theta_2)}
    \label{eq:omega-phi_rot}
\end{equation}
for some integers $m_j, n_j$.
Thus, when $\Phi^t$ is ergodic, the set of eigenfrequencies has the structure of an additive abelian group generated by $\alpha_1$ and $\alpha_2$, and forms a dense subset of the real line. The eigenfunctions $\psi_j$ have the form of Fourier functions, i.e., they correspond to characters of $\mathbb T^2$ viewed as an abelian group.

In the following experiments, we set $\alpha_1 = 1$ and $\alpha_2 = \sqrt{30}$. We additionally split this case into two experiments: one where the torus $\mathbb{T}^2$ is embedded into $\bbR^4$ and one where the embedding is  into $\bbR^3$. The observation map $F_4: \mathbb{T}^2 \to \bbR^4$ is the standard ``flat'' embedding of the torus,
\begin{equation}
    F_4(\theta_1, \theta_2) = (\cos\theta_1, \sin\theta_1, \cos\theta_2, \sin\theta_2).
    \label{eq:F4}
\end{equation}
The observation map $F_3: \mathbb{T}^2 \to \bbR^3$ is given by
\begin{equation}
    F_3(\theta_1, \theta_2) = ((1 + R \cos\theta_1)\cos\theta_2, (1 + R \cos\theta_1)\sin\theta_2, R \sin\theta_1 ),
    \label{eq:F3}
\end{equation}
where $R$ is a radius parameter set to $0.5$. While the Koopman eigenfunctions of the system do not change with the observation map, one could anticipate that computing these from data sampled in $\mathbb R^4$ will be easier since in this case the integral operator $G_\tau$ commutes with the Koopman operator $U^t$ (by translation invariance of the kernel; see \ref{app:markov_semigroup}), which means that Koopman eigenfunctions can be represented as finite linear combinations of basis functions $\phi_j$.

\paragraph{Generated data}
Data was generated with a time step of $\Delta t = 0.0491 \approx 2\pi / 28$ for 2011 time units ($N = \text{40,962}$ samples), and we compute the approximate Koopman eigenfunctions with the following parameters for both observation maps: $z = 1$ (resolvent parameter), $\tau = 10^{-4}$ (regularization parameter), $L = 101$ (number of basis functions $\phi_j$), $M = 33$ (rank parameter), and $T_\ell = 50$ (resolvent integration time). We compute the pseudospectral bound $\varepsilon_{T_c}(\omega_j, \psi_j)$ with the time parameter $T_c =49.1$. To assess the robustness of our results, we also compute eigenfunctions using a time series of $N= \text{4,098}$ samples rather than 40,962.

\paragraph{Numerical results}

\Cref{fig:torusex_4d,fig:torusex_3d} show representative numerical eigenfunctions $\psi_j$ computed from the dataset with $N=\text{40,962}$ samples using the $F_4$ and $F_3$ observation maps, respectively. These eigenfunctions were selected on the basis of their low corresponding $\varepsilon_{T_c}(\omega_j, \psi_j)$ values and the proximity of the corresponding eigenfrequencies $\omega_j$ with the theoretical values $\omega_j$ from~\eqref{eq:omega-phi_rot} equal to $\alpha_1 = 1$ ($m_j=1$, $n_j=0$), $\alpha_2 = \sqrt{30} \approx 5.48$ ($m_j=0$, $n_j=1$), and $\alpha_1 + \alpha_2 \approx 6.48$ ($m_j=n_j=1$). In the case of the $F_4$ experiments, the plotted eigenfunctions are $\psi_{56}$, $\psi_{58}$, and $\psi_{54}$; in the case of the $F_3$ experiments, they are $\psi_{16}$, $\psi_8$, and $\psi_4$. The numerically computed eigenfrequencies match the true eigenfrequencies between two and five significant digits (see \cref{tab:num_results}). Note that, in general, there is no correspondence between our ordering of the eigenfunctions based on $\varepsilon_{T_c}(\psi_j)$ and an ordering based on the corresponding eigenfrequencies. In the case of the $F_4$ observation map, the ordering places the generating eigenfrequencies near the end of list of eigenfrequencies at the 56th and 58th places of the $2M = 66$ nontrivial eigenfrequencies. In this case, the $\epsilon_{T_c}$ ordering becomes somewhat arbitrary as the eigenfunctions are well resolved, with $0.0082 \leq \epsilon_{T_c} \leq 0.0118$ for the first 60 eigenfrequencies. Eigenfrequency errors are generally more significant in the $F_3$ experiments---this is expected from the fact that in this case the embedding yields a manifold with curvature, and the corresponding kernel integral operator $G_\tau$ does not commute with the Koopman operator.

The bottom rows of \cref{fig:torusex_4d,fig:torusex_3d} show scatterplots of the real parts of $\psi_j$ on the sampled states $x_n$ on the torus. These scatterplots have the structure of plane waves and are in good agreement with the theoretical eigenfunctions in \eqref{eq:omega-phi_rot}; see \cref{fig:knowneigf}(top) for comparison. The middle row shows traceplots of the time series $\psi_j(x_0), \psi_j(x_1), \ldots$ in the complex plane. Theoretically, the time series should trace the unit circle in the complex plane. The numerical results in \cref{fig:torusex_4d,fig:torusex_3d} demonstrate that this is the case to a very good approximation. Some discrepancy from the unit circle is visible in the $\psi_{16}$ results in \cref{fig:torusex_3d}, which is likely due to the fact that the basis functions $\phi_j$ derived from the $F_3$ embedding cannot represent general Koopman eigenfunctions through finite basis expansions. The top rows of \cref{fig:torusex_4d,fig:torusex_3d} show time series of the real parts of $\psi_j(x_n)$. These time series have the structure of sinusoidal waves with frequencies close to the corresponding eigenfrequencies, as expected theoretically.

In separate calculations, we have computed eigenfunctions and eigenfrequencies using the $N=4098$-sample datasets (with either $F_4$ or $F_3$ embeddings). The reduction of the number of samples was found to impart a modest loss of accuracy compared to the $N=\text{40,962}$ datasets (evidenced by the higher values of $\varepsilon_{T_c}$ in \cref{tab:num_results}), but overall the computed eigenfunctions are in good qualitative agreement with those depicted in \cref{fig:torusex_4d,fig:torusex_3d}.

\begin{figure}
  \centering
  \includegraphics[width=\textwidth]{"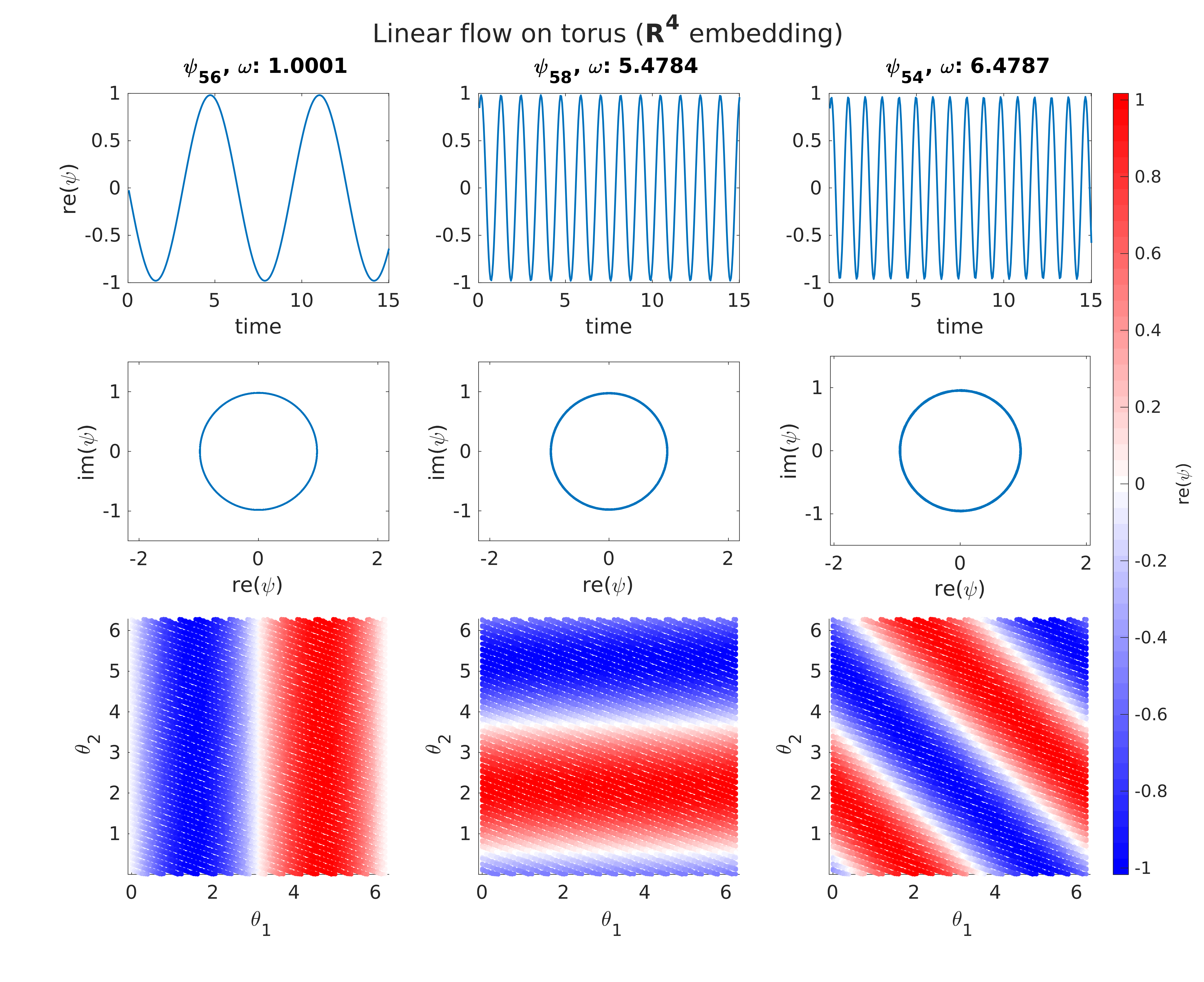"}
  \caption{Eigenfunctions of the approximate Koopman generator for the linear torus rotation computed using trajectory data in $\bbR^4$ from the embedding $F_4$ in~\eqref{eq:F4}. Each column shows a different eigenfunction $\psi_j$, with eigenfrequency $\omega_j$ given in the title. The index $j$ represents the eigenfunction order which corresponds to a sorting of smallest $\epsilon_{T_c}(\psi_j)$ as computed in \eqref{eq:eig-epsilon}. Top row: Real part of $\psi_j$ versus time along a portion of the dynamical trajectory. Middle row: Evolution of $\psi_j$ along a portion of the dynamical trajectory plotted in the complex plane. Bottom row: Scatterplot of the real part of $\psi_j$ on the 2-torus parameterized by the angle coordinates $\theta_1, \theta_2 \in [0,2\pi)$. Notice that eigenfunction $\psi_{54}$, shown in the right column, has an eigenfrequency that is approximately an integer linear combination of the eigenfrequencies of $\psi_{56}$ and $\psi_{58}$, shown in the left and middle columns, respectively.}
  \label{fig:torusex_4d}
\end{figure}

\begin{figure}
  \centering
  \includegraphics[width=\textwidth]{"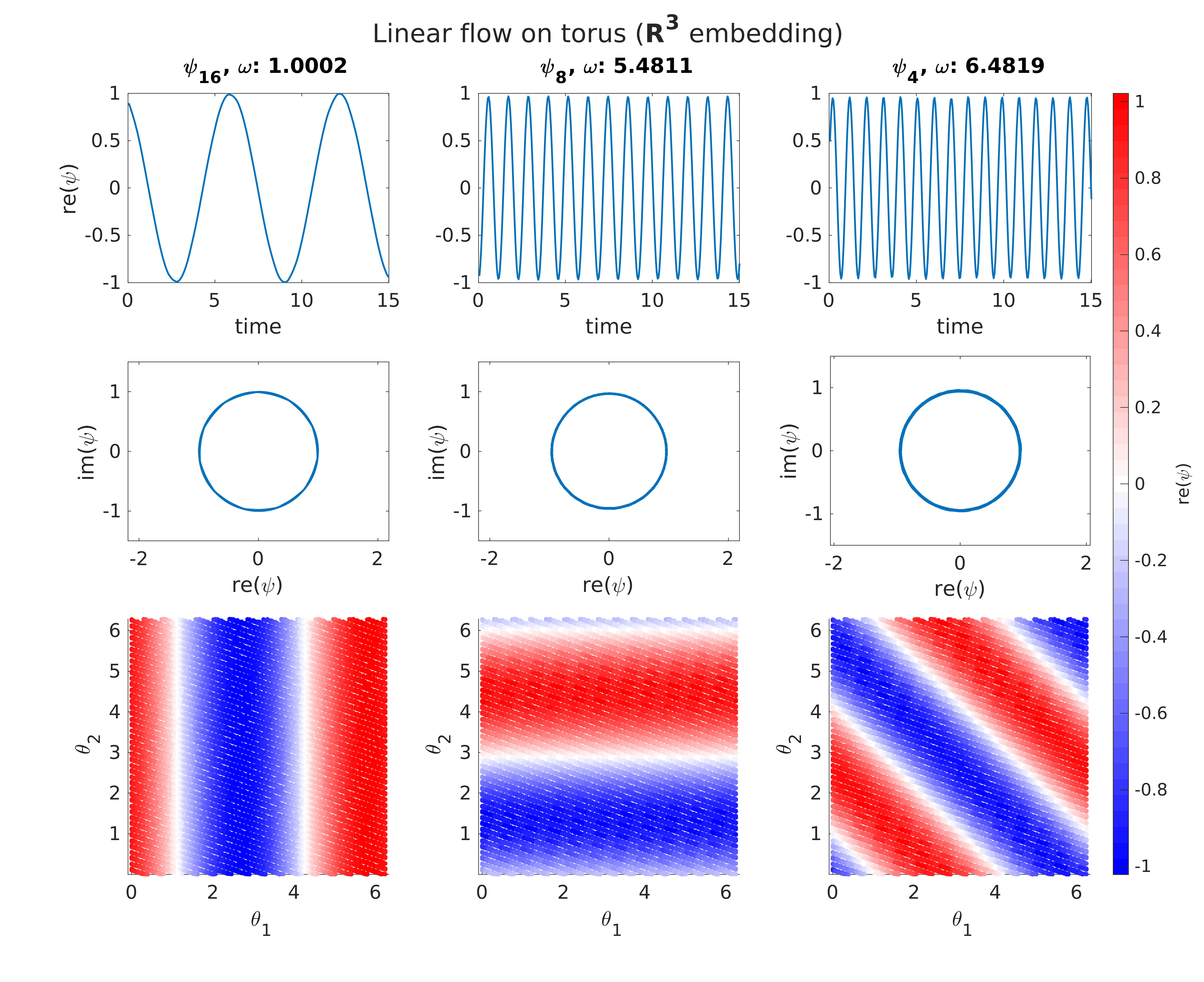"}
  \caption{Eigenfunctions of the approximate generator for the linear torus rotation, with trajectory data given in $\bbR^3$. Figure organized as described in \cref{fig:torusex_4d}. As one would expect, the eigenfrequencies and eigenfunctions associated with this system are nearly identical to those found in the $\bbR^4$ embedding of the same system.}
  \label{fig:torusex_3d}
\end{figure}

\subsection{Skew rotation on the torus}
\label{sec:skewtor}

We consider the skew-product flow $\Phi^t : \mathbb T^2 \to \mathbb T^2$ defined by the system of differential equations
\begin{equation}
    \frac{d\theta_1}{dt} = \alpha_1, \quad \frac{d\theta_2}{dt} = \alpha_2 ( 1 + \beta \sin\theta_1).
    \label{eq:skewrot}
\end{equation}
As with the linear rotation example in \cref{sec:lineartor}, $\alpha_1$ and $\alpha_2$ are real frequency parameters, and the system is measure-preserving and ergodic for the Haar probability measure when $\alpha_1$ and $\alpha_2$ are incommensurate. The parameter $\beta \in [0, 1)$ controls the strength of the influence of the $\theta_1$ evolution to the $\theta_2$ dynamics.

One can verify by solving \eqref{eq:skewrot} that the Koopman generator of this system admits the eigenfrequencies $\omega_j$ and corresponding eigenfunctions $\psi_j$ given by
\[ \omega_j = \alpha_1 m_j + \alpha_2 n_j, \quad \psi_j(\theta_1, \theta_2) = e^{i m_j \theta_1} e^{i n_j (\theta_2 + \beta\frac{\alpha_2}{\alpha_1} \cos(\alpha_1 \theta_1))}\]
with $m_j, n_j \in \mathbb Z$. Thus, the eigenfrequencies of this system are the same as those of the linear rotation in \cref{sec:lineartor}, whereas the eigenfunctions have the structure of distorted Fourier functions with the amount of distortion depending on $\beta$. See \cref{fig:knowneigf} in \ref{app:eigf} for plots of selected eigenfunctions. One can also verify that $ \{ \psi_j \}$ forms an orthonormal basis of $L^2(\mu)$.

In fact, the system~\eqref{eq:skewrot} is topologically conjugate to the linear torus rotation. We can define this conjugacy explicitly by setting $z_1$ and $z_2$ to the eigenfunctions $\psi_j$ corresponding to the generating eigenfrequencies $\alpha_1$ and $\alpha_2$, respectively, and defining $ h: \mathbb T^2 \to \mathbb C^2$ as the map $h(\theta_1, \theta_2) = (z_1(\theta_1), z_2(\theta_2))$. It can then be shown that $h$ is a continuous injective map with image $\mathbb T^2 \subset \mathbb C^2$ that satisfies the conjugacy relationship $h \circ \Phi^t = \Phi^t_\text{rot} \circ h$. Here, $\Phi_\text{rot}$ stands for the linear rotation from \cref{sec:lineartor} with the same generating frequencies as the skew rotation $\Phi^t$ generated by~\eqref{eq:skewrot}. More generally, the existence of $h$ follows from the Halmos--von Neumann theorem \cite{HalmosVonNeumann42}, which states that any two spectrally isomorphic systems with pure point spectra are measure-theoretically isomorphic. In this case, the isomorphism is realized by the continuous function $h$ so it becomes a topological conjugacy.

\paragraph{Generated data} We set $\beta = 0.5$ and use the same frequency parameters $\alpha_1 = 1$ and $\alpha_2 = \sqrt{30}$ as in \cref{sec:lineartor}. Similarly to \cref{sec:lineartor}, the data used was generated with a time step of $\Delta t = 0.0491 \approx 2 \pi / 28$ for $N = \text{40,962}$ or 4098 samples. The eigenfunctions of the approximate Koopman generator are computed with parameters $z = 1$, $L = 800$, $\tau = 10^{-4}$, $M = 40$, and $T_\ell = 50$. We order the eigendecomposition using $\varepsilon_{T_c}$ as in the previous example, with $T_c = 19.64$.

\paragraph{Numerical results}
We have computed eigenfunctions from data embedded in either $\mathbb R^4$ or $\mathbb R^3$ using the observation maps $F_4$ and $F_3$ from~\eqref{eq:F4} and~\eqref{eq:F3}, respectively.  The experiments with different numbers of samples and observation maps yielded broadly similar results, so in what follows we focus on results from the $F_4$ observation map with $N=\text{40,962}$ samples. See \cref{tab:num_results} for summary results from all four experiments.

In \cref{fig:skewex_4d}, we plot three representative eigenfunctions $\psi_j$ in a similar manner as the linear rotation eigenfunctions shown in \cref{fig:torusex_4d,fig:torusex_3d}. The three plotted eigenfunctions, $\psi_{22}$, $\psi_2$, and $\psi_{16}$, have eigenfrequencies approximately equal to $\alpha_1$, $\alpha_2$, and $\alpha_1 + \alpha_2$, respectively, similarly to \cref{fig:torusex_4d,fig:torusex_3d}.
These eigenfunctions match analytical expectations: the computed eigenfunctions plotted on the trajectory of the flow are similar to the true eigenfunctions shown in \cref{fig:knowneigf}(bottom) and oscillate at approximately the correct frequency.

\begin{figure}[htbp]
  \centering
  \includegraphics[width=\textwidth]{"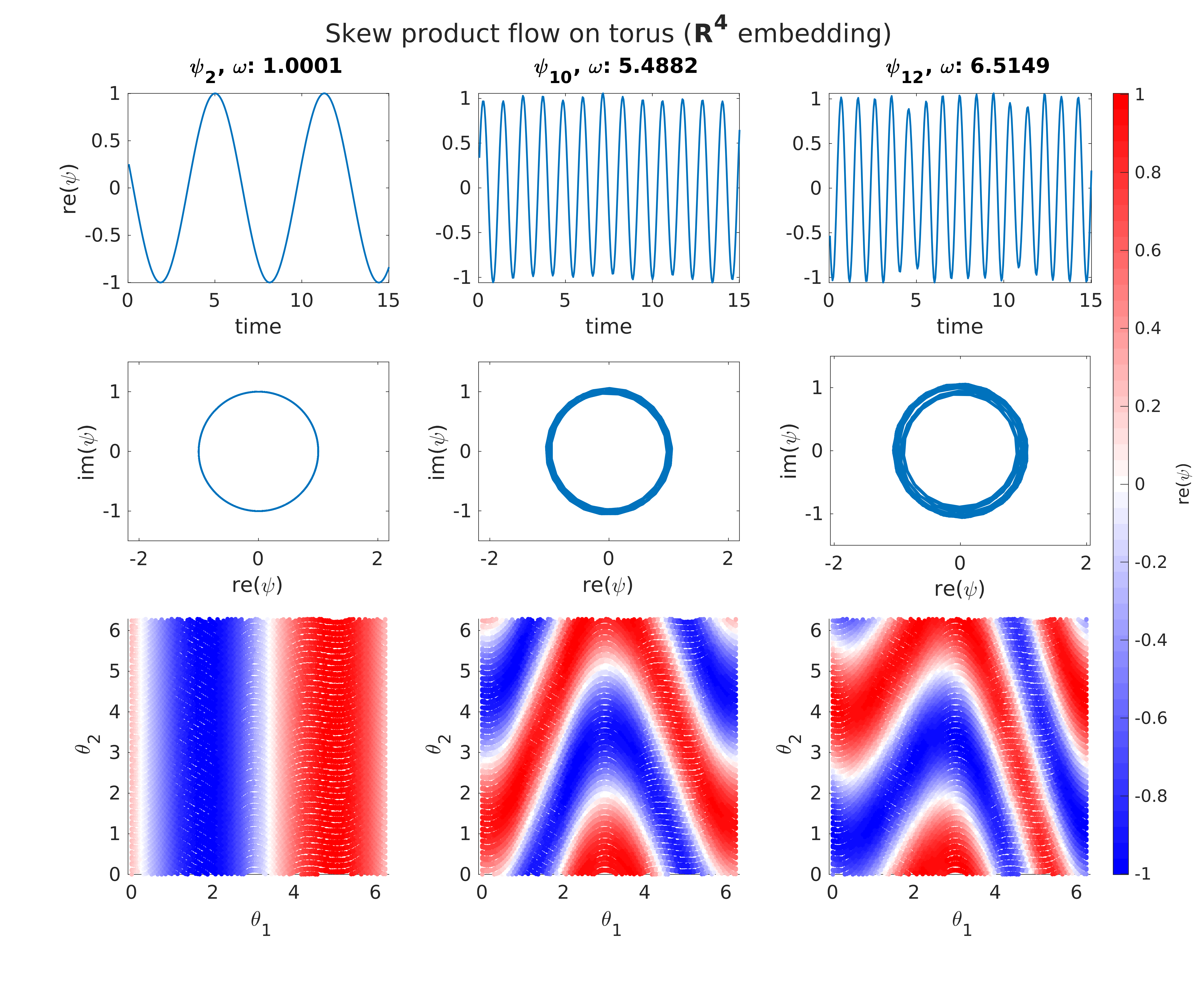"}
  \caption{Eigenfunctions of the approximate Koopman generator for the skew torus rotation, with trajectory data given in $\bbR^4$. Organized as in \cref{fig:torusex_4d}.}
  \label{fig:skewex_4d}
\end{figure}

\subsection{Lorenz 63 system}
\label{sec:L63}
The L63 system is generated by the smooth vector field $ \vec V : \mathbb R^3 \to \mathbb R^3$ defined as
\[ \vec V(x,y,z) =  (-\sigma(y - x), x(\rho - z) - y, xy - \beta z).\]
We use the standard parameter values $\beta = 8/3$, $\rho = 28$, $\sigma = 10$, and the identity for the observation map $F: \mathbb R^3 \to \mathbb R^3$. For this choice of parameters, the L63 system is known to have a compact attractor $X \subset \bbR^3$ with fractal dimension $\approx 2.06$ that supports a unique SRB measure $\mu$ (which is physical) \cite{Tucker99}. The system is also known to be mixing with respect to $\mu$ \cite{LuzzattoEtAl05}, which implies that the associated unitary Koopman group has no nonzero eigenfrequencies. The $H_p$ subspace is then a one-dimensional space of  constant functions, while $H_c$ contains all zero-mean functions in $L^2(\mu)$ (i.e., $H_c = \tilde H$; see \cref{sec:spectral-ergodic}). The L63 system with the standard parameter values is dissipative, $\divr \vec V < 0$, and can be shown to possess compact absorbing balls containing $X$ \cite{LawEtAl14}. Setting the forward invariant set $M$ to such an absorbing ball, all assumptions for data-driven approximation stated in \cref{sec:training} are rigorously satisfied. Yet, unlike the examples in \cref{sec:lineartor,sec:skewtor}, it is not obvious how the discrete spectra of the approximate Koopman generator $V_{z,\tau, L}^{(M)}$ should behave.

\paragraph{Generated data}
Data was generated with a time step of $\Delta t = 0.01$ for $4096$ time units.  We compute eigenfunctions of the approximate Koopman generator using the parameters $z = 1$, $\tau = 2 \times  10^{-6}$, $L = 2000$, $M = 333$, and $T_\ell = 50$. We use the pseudospectral bound $\varepsilon_{T_c}(\omega_k, \psi_j)$ computed with $T_c = 2$ to order the eigenpairs $(\omega_j,\psi_j)$.

\paragraph{Numerical results}
Representative eigenfrequencies $\omega_j$ and their corresponding eigenfunctions $\psi_j$ are displayed in \cref{tab:num_results,fig:l63ex}. Despite the mixing nature of the dynamics, the eigenfunctions appear to oscillate regularly and behave as Koopman eigenfunctions for sufficiently short times (see \cref{tab:num_results}). Additionally, many of the eigenfunctions remain predictable well beyond the Lyapunov time. The positive Lyapunov exponent of the system, approximately 0.90566, corresponds to a Lyapunov time of approximately $1.104$ time units \cite{viswanath1998lyapunov}. In contrast, many of the computed eigenfunctions remain correlated, thus predictable, for up to 5 time units. Some of these eigenfunctions -- namely $\psi_2$ and $\psi_4$ plotted in \cref{fig:l63ex} -- remain correlated for up to 20 time units. Notable also is that results similar to $\psi_2$ (where the similarity is most evident in the scatterplot of $\re\psi_2$ on the attractor) have been found in the L63 system by other methods, including \cite{DasEtAl21, FroylandEtAl21, Giannakis21a, KordaEtAl20}. The period associated with $\psi_2$, $2\pi/\omega_1 \approx 0.84$ time units, appears to capture the timescale of approximate periodic orbits around the fixed points centered at the ``holes'' in the L63 attractor lobes. In contrast, eigenfunction $\psi_{30}$, which is also shown in \cref{fig:l63ex}, appears to be of a qualitatively nature and is associated with a longer period of $2 \pi / \omega_{30} \approx 1.32$ time units. A possible interpretation of this eigenfunction is that it represents mixing between the two attractor lobes.

\begin{figure}[htbp]
  \centering
  \includegraphics[width=\textwidth]{"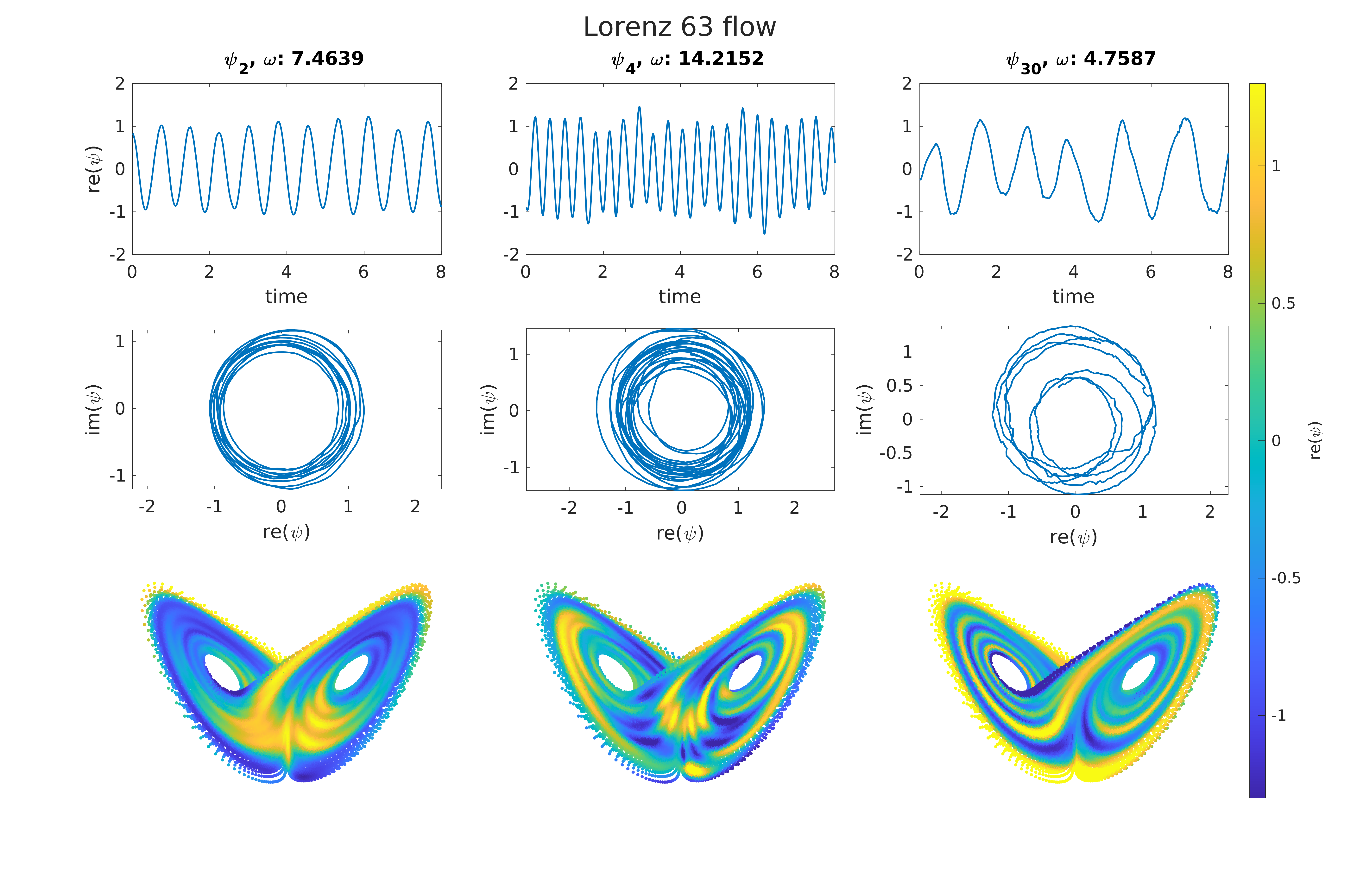"}
  \caption{Eigenfunctions of the approximate generator for the Lorenz 63 system, organized as described in \cref{fig:torusex_4d}. Eigenfunction $\psi_2$, plotted in the center column, appears to be an approximate harmonic of the eigenfunction in the left column, as $2\omega_2 \approx \omega_4$ and the eigenfunctions behave as such on the attractor. The rightmost eigenfunction, $\psi_{30}$, appears to be less smooth, but still has qualitatively good behavior (oscillates regularly at about the frequency of $\omega_{30}$).}
  \label{fig:l63ex}
\end{figure}

\subsection{Approximation of correlations}

To further test the convergence of our approximations, we examine if we can reproduce the time-autocorrelation functions $C_f(t)$ from~\eqref{eq:autocorr}. Defining, for $f \in \tilde H$ and for $t \in \mathbb R$, the Borel probability measure $\nu_f: \mathcal B(i \mathbb R) \to [0, 1]$ and bounded Borel-measurable function $h_t : i \mathbb R \to \mathbb C$ as $\nu_f(\Theta) = \langle f, E(\Theta) f\rangle$ and $h_t(i\omega) = e^{i\omega t}$ , we have $C_f(t) = \int_{i \mathbb R} h_t \, d\nu_f$ (here, $E: \mathcal B(i \mathbb R) \to B(H)$ is the spectral measure of the generator $V$; see \cref{sec:dynamics}). Similarly, the time-autocorrelation function $C_{f,\tau} : \mathbb R \to \mathbb C$ under $V_{z,\tau}$ can be expressed as
\begin{equation}
    \label{eq:autocorr_tau}
    C_{f,\tau}(t) = \langle f, U^t_\tau f\rangle = \int_{i \mathbb R} h_t \, d\nu_{f,\tau},
\end{equation}
where $\nu_{f,\tau} : \mathcal B(i \mathbb R) \to [0, 1]$ is the Borel probability measure $\nu_{f,\tau}(\Theta) = \langle f, E_\tau(\Theta)\rangle$ induced by the (discrete) spectral measure $E_\tau: \mathcal B(i \mathbb R) \to B(H)$ of $V_{z,\tau}$,
\begin{equation}
    E_\tau(\Theta) = \sum_{l: i\omega_{l,\tau} \in \Theta} \langle \psi_{l,\tau}, \cdot \rangle \psi_{l,\tau}.
    \label{eq:pvm_tau}
\end{equation}
Thus, testing for approximation of $C_f(t)$ by $C_{f,\tau}(t)$ amounts to testing for approximation of the spectral measure of the generator for specific observables ($f$) and test functions $(h_t)$. Combining \eqref{eq:autocorr_tau} and \eqref{eq:pvm_tau} leads to the formula
\begin{equation}
    C_{f,\tau}(t) = \sum_{l\in \mathbb Z} e^{i\omega_{l,\tau}t} \lvert c_{l,\tau}\rvert^2, \quad c_{l,\tau} = \langle \psi_{l,\tau}, f\rangle,
    \label{eq:autocorr_tau2}
\end{equation}
giving the autocorrelation function in terms of the expansion coefficients $c_{l,\tau}$ of $f$ in the $\psi_{l, \tau}$ basis of $H$ and the corresponding eigenvalues $\omega_{l,\tau}$.

In the numerical setting of this section, we consider the empirical normalized anomaly correlation function computed for times $t_j = j\, \Delta t$, $j \in \mathbb N_0$, from samples of an observable $f: \mathcal M \to \mathbb R$ on the training data,
\begin{displaymath}
    \hat C_f(t_j) = \frac{1}{N \lVert \hat f'\rVert_{\hat H_N}^2} \sum_{n=0}^{N-1} \hat f'(x_n) \hat f'(x_{n+j}), \quad \hat f' = f - \frac{1}{N}\sum_{n=0}^{N-1} f(x_n).
\end{displaymath}
If $f$ represents a unit vector in $\tilde H$ then $\hat C_f(t_j)$ converges to $C_f(t_j)$ as $N\to \infty $ by ergodicity. We further approximate $C_{f,\tau}(t_j)$ by the autocorrelation function $\hat C_{f,\tau}(t_j)$ associated with the spectral measure of $V_{z,\tau,L,N}^{(M)}$, computed using an analogous formula to~\eqref{eq:autocorr_tau2} applied to the unit vector $\hat f' / \lVert \hat f'\rVert_{\hat H_N} \in \hat H_N$ using the eigenfrequencies $\omega_{l,\tau,L,N}^{(M)}$ and the corresponding eigenfunctions $\psi_{l,\tau,L,N}^{(M)}$.


In \cref{fig:correlations}, we compare the empirical autocorrelation $\hat C_f$ with the approximated autocorrelation $\hat C_{f,\tau}$ with $f$ set to selected components $F^{(j)}: X \to \mathbb R$ of the embedding maps $F = (F^{(1)}, \ldots, F^{(d)}) : X \to \mathbb R^d$ for the linear rotation, skew-product rotation, and L63 examples from \cref{sec:lineartor,sec:skewtor,sec:L63}, respectively.

In all examples considered, $\hat C_{f,\tau}$ tracks $\hat C_f$ relatively well (at least for short times), with the case of the linear torus flow exhibiting the most accurate reconstruction (perhaps unsurprisingly due to the simplicity of the dynamics and observation map). A noticeable discrepancy between $\hat C_{f,\tau}(t_j)$ and $\hat C_f(t_j)$ is an apparent buildup of phase error as $t_j$ increases for the $F^{(3)}$ component of the embedding map for the skew-product rotation and L63 system. In the case of the $F^{(1)}$ and $F^{(2)}$ components of the L63 system, we see that $\hat C_{f,\tau}$ captures the initial decay of correlations, but on longer times exhibits oscillations which are not seen in the true system. This re-emergence of correlations is likely due to the fact that $\hat C_{f,\tau}$ is given by a superposition of at most $2M$ periodic components (see \cref{alg:numerical}), and thus cannot decay to 0 as $t_j \to \infty$. Nonetheless, for any $T>0$, $\hat C_{f,\tau}(t_j)$ can be made arbitrarily close to the true autocorrelation $C_f(t_j)$ for all $t_j \in [0, T]$ for sufficiently large $N, M, L$ and sufficiently small $\tau$. 


\begin{figure}

    \includegraphics[width=0.9\textwidth]{"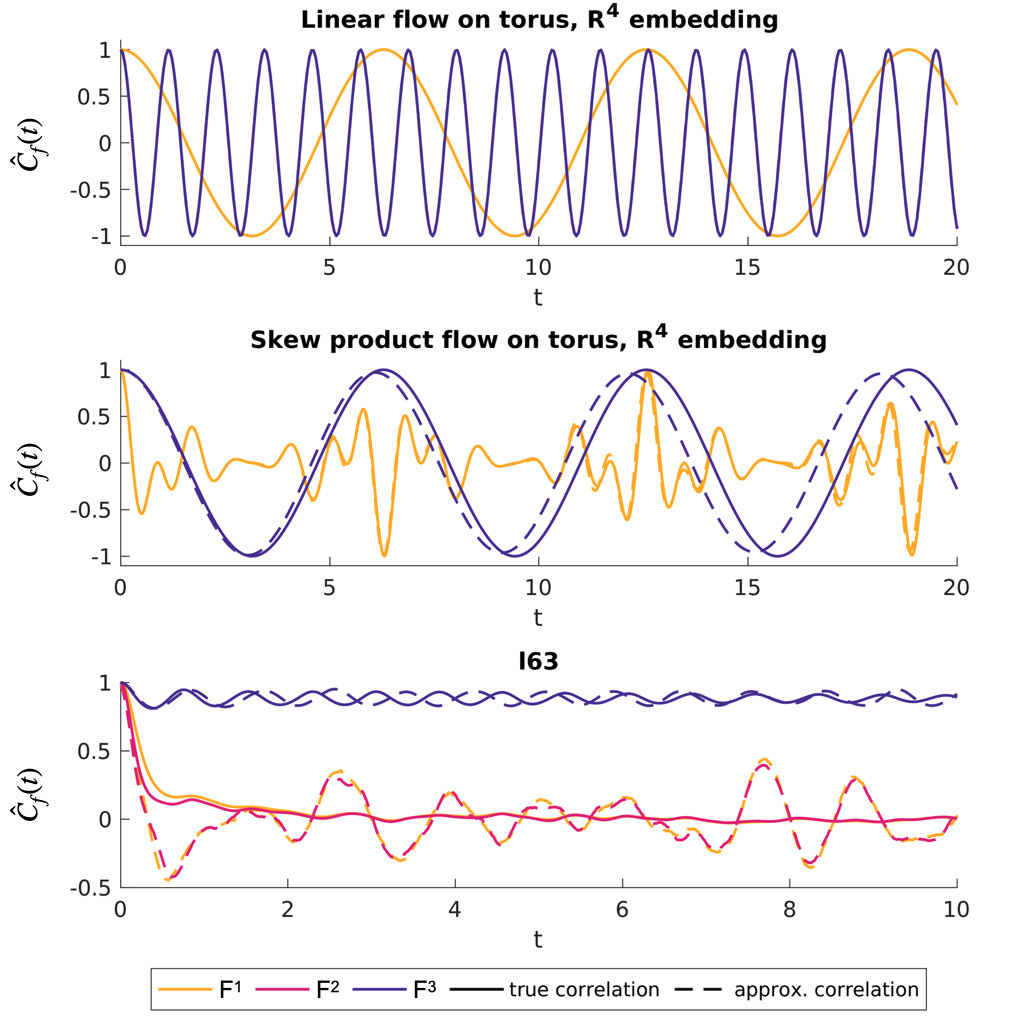"}
    \caption{Comparison of the empirical autocorrelation functions $\hat C_f$ (dashed lines) of selected coordinate functions $F^{(j)}$ to the autocorrelation functions $\hat C_{f,\tau}$ calculated from the approximated generator $V_{z,\tau}$ (solid lines). Only $F^{(1)}$ and $F^{(3)}$ are plotted in the two torus systems, as the correlation functions are the same in the $\mathbb R^4$ embedding for $j = 1$ and $j = 2$ as well as $j = 3$ and $j = 4$.}
    \label{fig:correlations}
  \end{figure}
\section{Conclusions}
\label{sec:conclusions}

In this paper, we have developed a method for spectrally-accurate approximations of the Koopman generator on $L^2$ for continuous-time, ergodic, measure-preserving dynamical systems. The primary element of this method is a regularization of the Koopman resolvent using a semigroup of Markovian kernel integral operators which produces a one-parameter family of compact operators with discrete spectra and complete orthonormal families of associated eigenfunctions. These compactified operators act as resolvents of approximations $V_{z,\tau}$ to the generator that preserve important properties (namely, skew-adjointness and unboundedness) and converge spectrally to the generator in a limit of vanishing regularization parameter, $\tau \searrow 0$. In particular, the family $V_{z,\tau}$ generates unitary evolution groups with discrete spectra that converge strongly to the Koopman group on $L^2$. Using a posteriori pseudospectral criteria, we identify elements of the spectra of $V_{z,\tau}$ whose corresponding eigenfunctions evolve with approximately cyclic behavior and slow correlation decay (i.e., behave as approximate Koopman eigenfunctions).

This method translates well to the data-driven setting. In particular, the use of the integral representation of the resolvent allows the use of integral quadrature methods with a resolvent parameter $z>0$ that can be tuned to the time-resolution of the observed data (which, in practice, is difficult to control). Moreover, our use of data-driven basis functions obtained by eigendecomposition of kernel integral operators naturally addresses the problem of constructing well-conditioned dictionaries for Koopman operator approximation with asymptotic convergence guarantees in the large data limit. Using examples with pure point spectra (torus rotation and skew-rotation) and mixing dynamics (Lorenz 63 system), we have demonstrated this method can accurately approximate Koopman eigenvalues and eigenfunctions for those systems that have them, and is also effective at identify slowly decorrelating observables in systems with mixing dynamics and non-trivial continuous Koopman spectra.

A possible topic for future work would be to characterize aspects of the dependence of the eigenvalues and eigenfunctions of $V_{z,\tau}$ on the regularization parameter $\tau$. As mentioned in \cref{sec:intro}, the fact that $V_{z,\tau}$ is a family of operators with compact resolvent (rather than compact operators as in the approach of \cite{DasEtAl21}) opens the possibility that this dependence is analytic, which could lead to new data-driven algorithms that take advantage of this analyticity. It would also be fruitful to explore applications of higher-order quadrature schemes  for resolvent approximation and/or positive-frequency projection, as well as applications of recently-proposed residual-based criteria for filtering spurious eigenvalues from the spectra of $V_{z,\tau}$ \cite{ColbrookTownsend24}. Finally, while the focus of this work has been on $L^2$ spaces, implementations of our approach with positive-definite kernels should have natural analogs in RKHSs (cf.\ \cite{DasEtAl21}). Working in an RKHS setting should be useful for supervised learning of models for predicting observables, which we plan to pursue in future work.

\appendix
\section{Kernel construction}
\label{app:markov}

In this appendix, we outline the construction of the families of Markov kernels $p_\tau: \mathcal M \times \mathcal M \to \mathbb R_+$ and $p_{\tau,N} : \mathcal M \times \mathcal M \to \mathbb R_+$, $\tau > 0 $, associated with the integral operators $G_\tau : H \to H$ and $G_{\tau,N}: \hat H_N \to \hat H_N $, respectively. Our approach follows closely \cite{Giannakis19,DasEtAl21}, who use results of \cite{VonLuxburgEtAl08,CoifmanHirn13,BerryHarlim16} to build these kernels and study their properties. Here, we summarize the main steps of the construction, referring the reader to \cite{Giannakis19,DasEtAl21} for further details.

\subsection{Choice of kernel}

Our starting point is a strictly positive, measurable kernel function $k: \mathcal M \times \mathcal M \to \mathbb R_+$ that is continuous the compact set $M$. In data-driven applications, we define $k$ as the pullback of a kernel $\kappa : Y \times Y \to \mathbb R_+$ on data space, i.e.,
\begin{equation}
    k(x,x') = \kappa(F(x),F(x'));
    \label{eq:kernel-pullback}
\end{equation}
see \cref{sec:numimplement}. As a concrete example, the Gaussian radial basis function (RBF) kernel,
\begin{equation}
    \kappa(y,y') = e^{-\lVert y - y'\rVert^2 / \epsilon^2}, \quad \epsilon > 0,
    \label{eq:rbf}
\end{equation}
is continuous and strictly positive on $Y=\mathbb R^d$, and thus leads to a pullback kernel $k$ from~\eqref{eq:kernel-pullback} that meets our requirements under the continuity and injectivity assumptions on the observation map $F$ from \cref{sec:training}.

In the experiments of \cref{sec:examples} we work with a variable-bandwidth generalization of \eqref{eq:rbf} proposed in \cite{BerryHarlim16},
\begin{equation}
    \kappa(y,y') = \exp\left(-\frac{\lVert y - y'\rVert^2}{\epsilon^2\sigma(y)\sigma(y')}\right), \quad \epsilon > 0,
    \label{eq:vb}
\end{equation}
where $\sigma: Y \to \mathbb R_+$ is a continuous, strictly positive bandwidth function obtained via a kernel density estimation procedure. Intuitively, the role of $\sigma$ is to increase (decrease) the locality of the kernel in regions of high (low) sampling density with respect to the Lebesgue measure on data space. See \cite[Appendix~A]{DasEtAl21} and \cite[Algorithm~1]{Giannakis19} for a precise definition and pseudocode. We tune the kernel bandwidth parameter $\epsilon$ automatically using the method described in \cite{CoifmanEtAl08,BerryHarlim16}; see again \cite[Algorithm~1]{Giannakis19}.

We should note that in applications the bandwidth function $\sigma$ depends on the sampling measure $\mu_N$ underlying the training data; however, it converges uniformly to a data-independent function in the large data limit, $N\to\infty$. In this appendix, we suppress the dependence of $\sigma$ on $\mu_N$ from our notation for $\kappa$ from~\eqref{eq:vb} for simplicity.

\subsection{Markov normalization}
\label{app:markov-normalization}

With the kernel $k$ from~\eqref{eq:kernel-pullback}, we apply a normalization procedure based on the approach of \cite{CoifmanHirn13} to obtain a symmetric Markov kernel $p : \mathcal M \times \mathcal M \to \mathbb R_+$ with respect to the invariant measure $\mu$, as well as Markov kernels $p_N : \mathcal M \times \mathcal M \to \mathbb R_+$ with respect to the sampling measures $\mu_N$.

Given a Borel probability measure $\nu$ with support contained in $M$, we construct a symmetric Markov kernel $p^{(\nu)} : \mathcal M \times \mathcal M \to \mathbb R+$ with respect to $\nu$ through the sequence of operations
\begin{equation}
    \begin{gathered}
        d(x) = \int_M k(x,x')\, d\nu(x'), \quad q(x) = \int_M \frac{k(x,x')}{d(x')}\,d\nu(x'),\\
        p^{(\nu)}(x,x') = \int_M \frac{k(x,x'') k(x'',x')}{d(x)q(x'')d(x')}\,d\nu(x'').
    \end{gathered}
    \label{eq:bistoch}
\end{equation}
One readily verifies that $p^{(\nu)}$ is symmetric, strictly positive, and continuous on $M \times M$, and satisfies the Markov condition
\begin{equation*}
    \int_M p^{(\nu)}(x,\cdot)\,d\nu = 1, \quad \forall x \in \mathcal M.
\end{equation*}
Moreover, it can be shown that the corresponding integral operator $G^{(\nu)} : L^2(\nu) \to L^2(\nu)$ defined as
\begin{equation}
    G^{(\nu)} f = \int_M p^{(\nu)}(\cdot, x) f(x) \, d\nu(x)
    \label{eq:g-op}
\end{equation}
is a trace class, strictly positive, ergodic Markov operator on $L^2(\nu)$.

Henceforth, we will let $p \equiv p^{(\mu)}$ and $\hat p_N \equiv p^{(\mu_N)}$ be the kernels obtained via~\eqref{eq:bistoch} for the invariant measure $\mu$ and sampling measure $\mu_N$, respectively, and $G \equiv G^{(\mu)}$ and $\hat G_N \equiv G^{(\mu_N)}$ be the corresponding integral operators on $H=L^2(\mu)$ and $\hat H_N=L^2(\mu_N)$, respectively, defined via \eqref{eq:g-op}. We will also let
\begin{equation}
    G \phi_j = \lambda_j \phi_j, \quad \hat G_N \phi_{j,N} = \lambda_{j,N}\phi_{j,N}
    \label{eq:g-g_gn-ops}
\end{equation}
be eigendecompositions of $G$ and $\hat G_N$, where $ \{ \phi_j \}_{j=0}^\infty$ and $ \{ \phi_{j,N} \}_{j=0}^{N-1}$ are real orthonormal bases of $H$ and $\hat H_N$, and the corresponding eigenvalues are ordered as $1=\lambda_0 > \lambda_1 \geq \lambda_2 \geq \cdots \searrow 0^+$ and $1=\lambda_{0,N} > \lambda_{1,N} \geq \lambda_{2,N} \geq \cdots \lambda_{N-1,N} > 0$, respectively. Note that $\hat G_N$ is represented by an $N\times N$ bistochastic matrix $\bm G = [\hat p_N(x_i,x_j)]_{i,j=0}^{N-1}$ computed from the training data $y_i = F(x_i)$, and the eigenvalues and eigenvectors of $\hat G_N$ can be obtained by eigendecomposition of $\bm G$ (cf.\ \cref{sec:basis}). However, it turns out that for the class of kernels from~\eqref{eq:bistoch} the eigendecomposition of $\hat G_N$ can be computed from the singular value decomposition of an $N\times N$ non-symmetric kernel matrix $\hat{\bm K}$ that factorizes $\bm G$ as $\bm G = \hat{\bm K} \hat{\bm K}^\top$; see \cite[Appendix~B]{DasEtAl21}. We use the latter approach in our numerical experiments as it avoids explicit formation of the kernel matrix $\bm G$.

\subsection{Markov semigroups}
\label{app:markov_semigroup}

We use the integral operators $G$ and $\hat G_N$ from \ref{app:markov-normalization} to define semigroups of Markov operators $G_\tau: H \to H$ and $G_{\tau,N} : \hat H_N \to \hat H_N$ parameterized by the regularization parameter $\tau\geq 0$. To that end, we first define the non-negative numbers
\begin{equation*}
    \eta_j = \left(\frac{1}{\lambda_j} -1\right)\frac{1}{\lambda_1}, \quad \eta_{j,N} = \left(\frac{1}{\lambda_{j,N}} -1\right)\frac{1}{\lambda_{1,N}},
\end{equation*}
and the self-adjoint operators $\Delta : D(\Delta) \to H$ and $\Delta_N : \hat H_N \to \hat H_N$ with
\begin{equation*}
    \Delta = \frac{G^{-1} - I}{\lambda_1}, \quad \Delta_N = \frac{\hat G^{-1}_N - I}{\lambda_{1,N}},
\end{equation*}
where the domain $D(\Delta) \subset H$ of $\Delta$ is given by $D(\Delta) = \{ f \in H: \sum_{j=0}^\infty \eta_j \lvert \langle \phi_j, f \rangle \rvert^2 < \infty \}$.

The operators $\Delta$ and $\Delta_N$ are positive operators with discrete spectra $ \{ \eta_j \}_{j=0}^\infty$ and $ \{ \eta_{j,N} \}_{j=0}^{N-1}$ that generate semigroups $ \{ G_\tau \}_{\tau\geq 0}$ and $ \{ G_{\tau,N} \}_{\tau \geq 0}$ of compact self-adjoint operators $G_\tau = e^{-\tau\Delta}$ and $G_{\tau,N} = e^{-\tau\Delta_N}$, respectively. One readily verifies the eigendecompositions
\begin{equation*}
    G_\tau \phi_j = \lambda_j^\tau \phi_j, \quad G_{\tau,N} \phi_{j,N} = \lambda_{j,N}^{\tau} \phi_{j,N},
\end{equation*}
where $\lambda_j^\tau = e^{-\tau \eta_j}$ and $\lambda_{j,N}^\tau = e^{-\tau \eta_{j,N}}$ are strictly positive eigenvalues for every $\tau \geq 0$. Moreover, it can be shown \cite[Theorem~1]{DasEtAl21} that for any $\tau>0$, $G_\tau$ and $G_{\tau,N}$ are trace class integral operators,
\begin{equation*}
    G_\tau f = \int_M p_\tau(\cdot, x) f(x) \,d\mu(x), \quad G_{\tau,N} f = \int_M p_{\tau,N}(\cdot, x)f(x)\,d\mu_N(x).
\end{equation*}
Here, $p_\tau : \mathcal M \times \mathcal M \to \mathbb R_+$ and $p_{\tau,N}: \mathcal M \times \mathcal M \to \mathbb R_+$ are Markov kernels that are continuous on $M\times M$, and are given by uniformly convergent Mercer expansions
\begin{equation}
    \label{eq:p_tau}
    p_\tau(x,x') = \sum_{j=0}^\infty \lambda_j^\tau \varphi_j(x) \varphi_j(x'), \quad p_{\tau,N}(x,x') = \sum_{j=0}^{N-1} \lambda_{j,N}^\tau \varphi_{j,N}(x) \varphi_{j,N}(x'), \quad \forall x,x' \in M.
\end{equation}
In \eqref{eq:p_tau}, $\varphi_j, \varphi_{j,N} \in C(M)$ are the continuous representatives of $\phi_j$ and $\phi_{j,N}$, respectively (cf.~\eqref{eq:varphi}),
\begin{equation*}
    \varphi_j = \frac{1}{\lambda_j^\tau} \int_M p(\cdot, x) \phi_j(x)\,d\mu(x), \quad \varphi_{j,N} = \frac{1}{\lambda_{j,N}^\tau} \int_M \hat p_N(\cdot, x) \phi_{j,N}(x)\,d\mu_N(x).
\end{equation*}
In summary, the kernels $p_\tau$ and $p_{\tau,N}$ satisfy all requirements laid out in \cref{sec:kernel,sec:basis}.

Before closing this appendix, we note that the Gaussian RBF kernel~\eqref{eq:rbf} on $\mathbb R^d$ is translation-invariant (i.e., $\kappa(y+s,y'+s) = \kappa(y,y')$ for all $y,y',s \in \mathbb R^d$), but the variable-bandwidth kernel~\eqref{eq:vb} is not translation-invariant if $\sigma$ is a non-constant function. In the linear torus rotation and skew torus rotation examples in \cref{sec:lineartor,sec:skewtor}, respectively, the data-dependent bandwidth function $\sigma$ is, in general, non-constant but converges as $N\to\infty$ to a constant function on the support of the sampling distribution  of the data when using the $F_4$ (``flat'') embedding~\eqref{eq:F4}. The latter, implies that the associated kernel integral operator $G_\tau$ commutes with the Koopman operators of the linear and skew torus rotations, and thus that any given Koopman eigenspace can be spanned by finitely many basis functions $\phi_j$. In the case of the torus examples using the $F_3$ embedding~\eqref{eq:F3} into $\mathbb R^3$ or the L63 system (\cref{sec:L63}), the bandwidth function $\sigma$ is nonconstant even in the $N\to\infty$ limit. In such cases, we cannot expect to completely recover Koopman eigenspaces with finite basis expansions.

\section{Koopman eigenfunctions for linear and skew torus rotations}
\label{app:eigf}
We plot selected known eigenfunctions for the Koopman generator of  the systems described in \cref{sec:lineartor} and \cref{sec:skewtor} for ease of comparison between the analytically known eigenfunctions and the approximate numerical eigenfunctions. See \cref{fig:knowneigf}.

\begin{figure}
    \centering
    \includegraphics[width=\textwidth]{"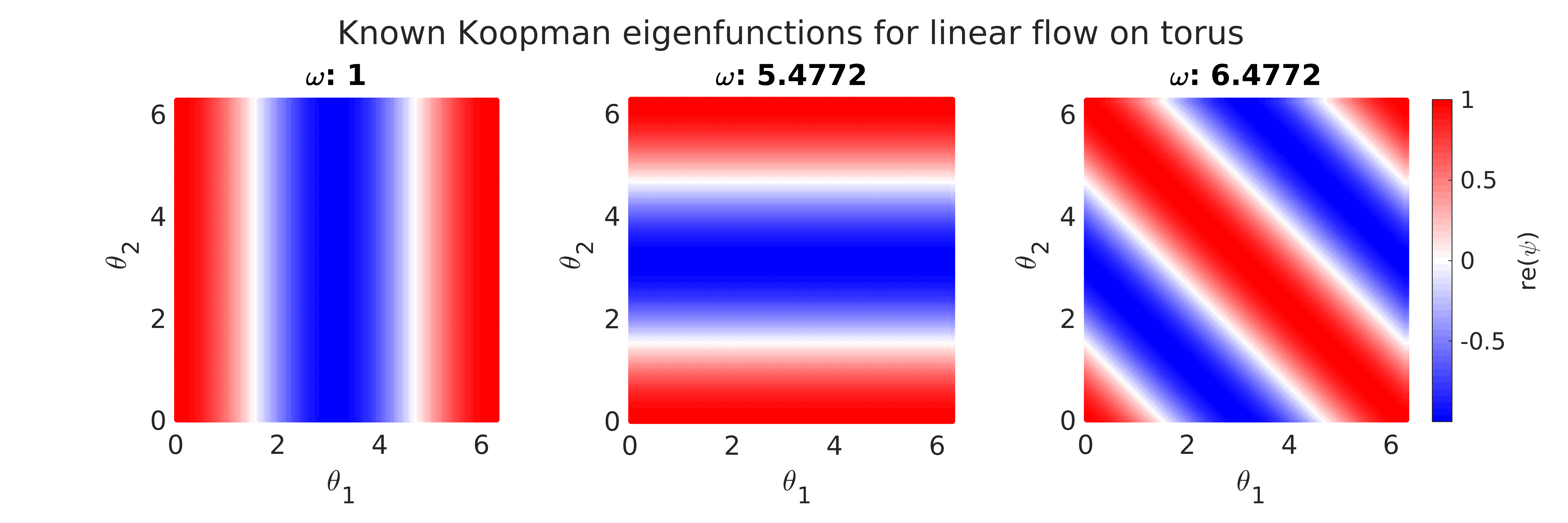"}
    \includegraphics[width=\textwidth]{"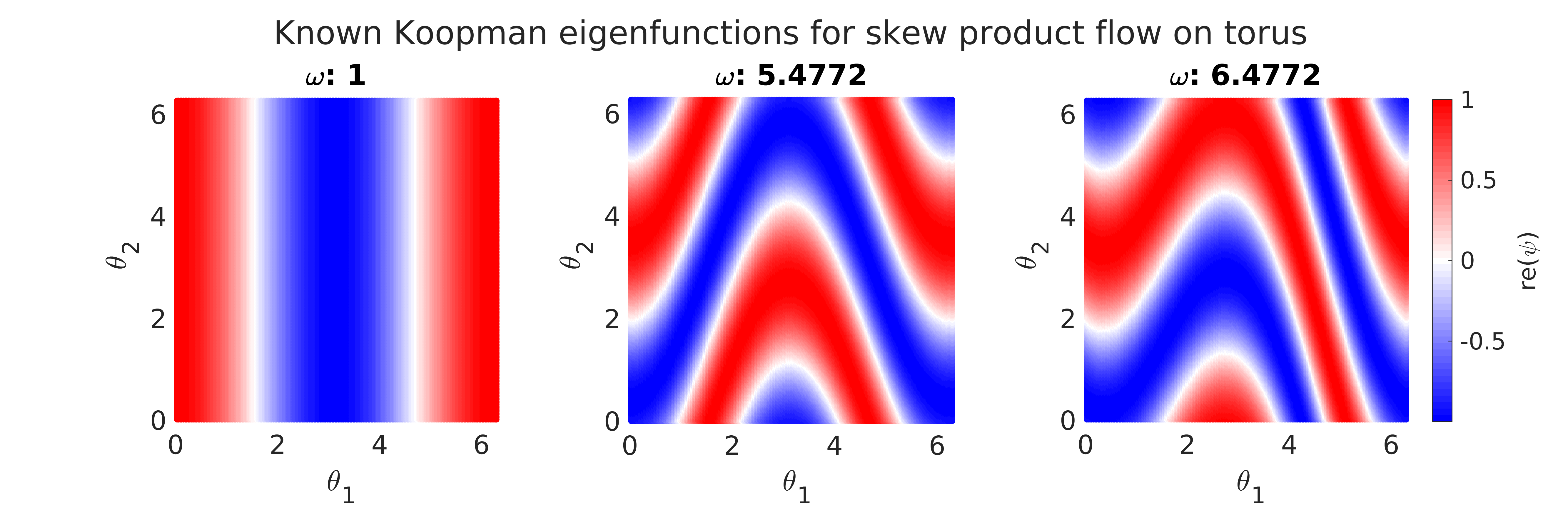"}
    \caption{Scatterplots of real parts of Koopman eigenfunctions $\psi_j$ on the torus parameterized by $\theta_1$ and $\theta_2$ for eigenfrequencies equal to $\omega_j = 1$ (left), $\omega_j = \sqrt{30}$ (middle), and $\omega_j = 1 + \sqrt{30}$ (right) with the same colormaps as the figures in \cref{sec:examples}. Top: Koopman eigenfunctions for the linear flow on the torus.
    Bottom: Koopman eigenfunctions for the skew-product flow on the torus}
    \label{fig:knowneigf}
  \end{figure}

\section*{Acknowledgments}
We thank Michael Montgomery, Keefer Rowan, and Travis Russell for helpful conversations. DG acknowledges support from the US National Science Foundation under grant DMS-1854383, the US Office of Naval Research under MURI grant N00014-19-1-242, and the US Department of Defense, Basic Research Office under Vannevar Bush Faculty Fellowship grant N00014-21-1-2946. CV was supported by the US National Science Foundation Graduate Research Fellowship under grant DGE-1839302.

\section*{References}

\bibliographystyle{plain}

\end{document}